\documentclass[12pt]{amsart}  
\usepackage[margin=.85 in]{geometry}
\usepackage{tikz-cd}
\usepackage{amscd}
\usepackage{amssymb}
\usepackage{thmtools} 
\usepackage{thm-restate}
\usepackage{dsfont}
\usepackage{bbm}
\usepackage[pdfencoding=auto, psdextra]{hyperref}
\usepackage{mathrsfs}
\usepackage{scalerel}
\theoremstyle{plain}
\newtheorem{lem}{Lemma}[subsection]
\newtheorem{prop}[lem]{Proposition}
\newtheorem{thm}[lem]{Theorem}

\theoremstyle{definition}
\newtheorem{const}{Construction}[subsection]
\newtheorem{defi}[const]{Definition}

\theoremstyle{remark}
\newtheorem*{remark}{Remark}
\newtheorem{example}{Example}[subsection]
\newtheorem{nota}{Notation}[section]\newcommand{\trop}{\textnormal{trop}}

\numberwithin{equation}{section}

\newcommand{\legval}{\operatorname{legval}}
\newcommand{\HS}{\operatorname{HS}}
\newcommand{\VS}{\operatorname{VS}}

\newcommand{\bbC}{\mathbb{C}}

\newcommand{\bbG}{\mathbb{G}}

\newcommand{\bbP}{\mathbb{P}}
\newcommand{\bbQ}{\mathbb{Q}}
\newcommand{\bbR}{\mathbb{R}}

\newcommand{\bbZ}{\mathbb{Z}}

\newcommand{\cM}{\mathcal{ M}}

\newcommand{\cC}{\mathcal{ C}}

\newcommand{\cX}{\mathcal{ X}}
\newcommand{\cY}{\mathcal{ Y}}
\newcommand{\cZ}{\mathcal{ Z}}
\newcommand{\op}{\textnormal{op}}

\newcommand{\Hom}{\operatorname{Hom}}

\newcommand{\Top}{\operatorname{Top}}

\newcommand{\Lin}{\operatorname{Lin}}

\newcommand{\Id}{\operatorname{Id}}

\newcommand{\abs}[1]{\lvert #1 \rvert}

\newcommand{\colim}{\operatorname{colim}}

\newcommand{\Star}{\operatorname{Star}}

\newcommand{\val}{\operatorname{val}}

\newcommand{\Mtrop}{\operatorname{M}^{\trop}}
\renewcommand{\Im}{\operatorname{Im}}

\newcommand{\ord}{\operatorname{ord}}

\newcommand{\Aut}{\operatorname{Aut}}

\newcommand{\POIC}{\operatorname{POIC}}
\newcommand{\POICCmplxs}{\operatorname{pCmplxs}}
\newcommand{\POICSpcs}{\operatorname{pSpcs}}

\newcommand{\dist}{\operatorname{dist}}
\newcommand{\ST}{\operatorname{ST}}
\newcommand{\st}{\operatorname{st}}
\newcommand{\est}{\operatorname{est}}

\newcommand{\ram}{\operatorname{Ram}}

\newcommand{\bbAC}{\mathbb{A}\mathbb{C}}
\newcommand{\bbEAC}{\mathbb{E}\mathbb{A}\mathbb{C}}
\newcommand{\src}{\textnormal{src}}
\newcommand{\trgt}{\textnormal{trgt}}
\newcommand{\virt}{\operatorname{virt}}
\newcommand{\rel}{\operatorname{rel}}
\newcommand{\Eq}{\operatorname{Eq}}
\newcommand{\Obj}{\operatorname{Obj}}
\newcommand{\ft}{\operatorname{ft}}
\newcommand{\cAC}{\mathcal{A}\mathcal{C}}

\newcommand{\tint}{\textnormal{int}}
\newcommand{\pspcs}{\textnormal{s}}
\newcommand{\pcmplxs}{\textnormal{c}}
\newcommand{\AC}{\mathrm{AC}}
\newcommand{\EST}{\mathrm{EST}}

\newcommand{\Subd}{\operatorname{Subd}}
\newcommand{\mt}{\textnormal{mult}}
\title{Tropical cycles of discrete admissible covers.}
\author{Diego A. Robayo Bargans}
\address{FB Mathematik, RPTU Kaiserslautern, 67663 Kaiserslautern, Germany}
\email{robayo@mathematik.uni-kl.de}

\begin{document}
\begin{abstract}
    This article applies the technical framework developed in previous work by the author to discrete admissible covers and their moduli spaces. More precisely, we construct a poic-space that parameterizes the discrete admissible covers after fixing the genus of the target, the number of marked legs, and prescribing the ramification profiles above these marked legs. We then construct a linear poic-fibration over this poic-space and show that the usual weight assignment of covers produces an equivariant cycle of this poic-fibration. This poic-fibration comes naturally with source and target maps, and after taking the weak pushforward in top dimension through the source map and subsequently forgetting the marking, this gives rise to an equivariant tropical cycle of the corresponding spanning tree fibration. Through this framework we obtain a generalization of previously known results on these cycles pertaining to the gonality of tropical curves. Special cases of these tropical cycles yield the following result: \emph{A generic genus-$g$ $r$-marked tropical curve $\Gamma$ (with $g+r$ even) has $C_{\frac{g+r}{2}}$ many $r$-marked discrete admissible covers of degree $\frac{g+r}{2}+1$ of a tree (counted with multiplicity) that have a tropical modification of $\Gamma$ as a source, where $C_{\frac{g+r}{2}}$ is the $(\frac{g+r}{2})$th Catalan number}.
\end{abstract}



\maketitle

\section{Introduction}
The origin of this project lies in the divisor theory of graphs and tropical curves (\cite{BakerNorine}), more precisely in the gonality of tropical curves and the results of \cite{VargasDraisma} and \cite{VargasThesis}. It is shown in \cite{VargasDraisma} that any genus-$g$ tropical curve has a tropical modification that is the source of a (full-rank) DT-morphism of degree $(\lceil\frac{g}{2}\rceil+1)$ to a tree. The methods therefrom follow a detailed analysis and classification of (certain) DT-morphisms of degree $\lceil\frac{g}{2}\rceil+1$ from genus-$g$ connected tropical curves to trees. By means of the technical framework developed in \cite{DARB}, we manage to recover the aforementioned result by showing that the corresponding loci of tropical curves are top dimensional equivariant tropical cycles of the corresponding spanning tree fibrations. More generally, we show that if we fix the genus $h$ and marking $m$ of the target, the degree $d$ of the covers (DT-morphisms), and the ramification profile above the $m$ marked legs, then the locus of genus-$g$ $n$-marked tropical curves (here $g$ and $n$ are determined by the previous data) having a tropical modification that is the source of a degree-$d$ full-rank DT-morphism of a genus-$h$ $m$-marked tropical curve with the prescribed ramification above the marked legs is an equivariant tropical cycle of the corresponding spanning tree fibration. In addition, the irreducibility of this linear poic-fibration implies an enumerative result generalizing that of \cite{VargasThesis}. Namely, we show that generic genus-$g$ $r$-marked tropical curve $\Gamma$ (with $g+r$ even) has $C_{\frac{g+r}{2}}$ many $r$-marked discrete admissible covers of degree $\frac{g+r}{2}+1$ of a tree (counted with multiplicity) that have a tropical modification of $\Gamma$ as a source, where $C_{\frac{g+r}{2}}$ is the $(\frac{g+r}{2})$th Catalan number

In what follows, a tropical curve simply means a discrete graph (it may have legs) with a given metric. It is well known that general harmonic morphisms to rational tropical curves provide a method of obtaining rank $\geq1$ divisors on a given curve. Special types of harmonic morphisms are tropical covers or (discrete) admissible covers, where there is more control on the Riemann-Hurwitz numbers and the situation mimics its algebraic counterpart. We seek to study the loci of curves that admit a discrete admissible cover from (a tropical modification of) themselves onto a tree, and more generally the moduli space of discrete admissible covers. We introduce marked legs to control the tropical modifications and phrase our enterprise in terms of the framework developed in \cite{DARB}. We are interested in a description of a moduli space of discrete admissible covers that equally permits the introduction and study of tropical cycles therein. 

More precisely, for these moduli spaces of discrete admissible covers we consider integers $h,d\geq0$ and a vector $\vec{\mu}=(\mu_i)_{i=1}^m$ of integer partitions of $d$ with $2h+m-2>0$, set $n = \sum_{i=1}^m\ell(\mu_i)$ and let $g$ be defined by the Riemann-Hurwitz formula (this directly implies that $2g+n-2>0$). With these data we construct a poic-space $\AC_{d,h,\vec{\mu}}$ parameterizing degree-$d$ discrete admissible covers of a genus-$h$ $m$-marked curve with ramification above the marked ends prescribed by $\vec{\mu}$. This poic-space comes with two natural source and target morphisms of poic-spaces $\src\colon \AC_{d,h,\vec{\mu}}\to \Mtrop_{g,n}$ and $\trgt\colon \AC_{d,h,\vec{\mu}}\to \Mtrop_{h,m}$.

By means of these maps and the linear poic-fibrations $\st_{g,n}$ and $\st_{h,m}$, we construct a linear poic-fibration $\est_{d,h,\vec{\mu}}\colon \EST_{d,h,\vec{\mu}}\to \AC_{d,h,\vec{\mu}}$
that also comes with the corresponding source and target morphisms $\mathfrak{src}\colon \est_{d,h,\vec{\mu}}\to \st_{g,n}$ and $\mathfrak{trgt}\colon \est_{d,h,\vec{\mu}}\to \st_{h,m}$. Motivated by the usual weight assignment (\cite{CavalieriMarkwigRanganathan}), we produce an $\est_{d,h,\vec{\mu}}$-equivariant Minkowski weight whose weak pushforward in top dimension through the morphisms $\mathfrak{src}$ and forgetting the marking produces an $\st_{g,n}$-equivariant cycle whose support corresponds to the sought after loci of curves. The balancing condition on this weight is reduced to several ``local'' balancing conditions, which are obtained by means of J. Li's degeneration formula. It is worth noting that these balancing conditions can also be obtained through the log geometric version of this degeneration formula as developed in \cite{KimLhoRuddat}. In specific cases, we subsequently use the balancing of this Minkowski weight to produce a top dimensional equivariant tropical cycle of the corresponding spanning tree fibration. These tropical cycles yield the following enumerative result: A generic genus-$g$ $r$-marked tropical curve $\Gamma$ (with $g+r$ even) has $C_{\frac{g+r}{2}}$ many $r$-marked discrete admissible covers of degree $\frac{g+r}{2}+1$ of a tree (counted with multiplicity) that have a tropical modification of $\Gamma$ as a source, where $C_{\frac{g+r}{2}}$ is the $(\frac{g+r}{2})$th Catalan number.

The structure of the paper is as follows: In section 2 we gather some preliminaries from \cite{DARB}, these refer to discrete graphs and partially open polyhedral cones, constructions and results on poic-complexes, -spaces, and -fibrations. We point the reader toward loc. cit. for examples, details, and proofs, as we have spared these from this article due to space constraints. In section 3 we follow the spirit of the construction of both $\Mtrop_{g,n}$ and $\st_{g,n}$ (for integers $g,n\geq0$ with $2g+n-2>0$), and introduce the relevant categories of discrete admissible covers, together with the respective poic-space and linear poic-fibration. More precisely, we introduce the poic-space $\AC_{d,h,\vec{\mu}}$ of degree-$d$ discrete admissible covers of genus-$h$ $m$-marked graphs with ramification $\vec{\mu}$ and the extended spanning tree fibration $\est_{d,h,\vec{\mu}}$ of degree-$d$ discrete admissible covers of genus-$h$ $m$-marked graphs with ramification $\vec{\mu}$, where $h,m\geq0$ are integers and $\vec{\mu}$ is a length-$m$ vector of partitions of $d$. In section 4 we describe the standard weight on $\est_{d,h,\vec{\mu}}$ and show that it is an $\est_{d,h,\vec{\mu}}$-equivariant Minkowski weight, which we then use to obtain the previously mentioned enumerative statement. We close by relating our results with those of \cite{VargasDraisma} and \cite{VargasThesis}.

\subsection*{Acknowledgments:} I would like to thank A. Gathmann for the multiple helpful and inspiring discussions that made this project possible and brought it into existence, as well as for proofreading, commentary, and patience. Interactions made possible by the SFB-TRR 195 of the DFG have also benefited this project. I also thank A. Vargas for helpful conversations, discussions, and proofreading of preliminary versions. This project has been carried out and realized through financial support of the DAAD.

\section{Preliminaries}
In this section we recall several definitions, notational conventions, and constructions from \cite{DARB}. More precisely, we quickly cover most of the significant constructions of loc. cit. and point the reader thereto for examples, proofs, and much more detail. 

\subsection{Discrete graphs.} In the spirit of clarity, we use this section to fix notation and state precise definitions for everything concerning discrete graphs that will be significant to our work. We follow the definition of a graph with legs of \cite{LenUlirschZakharovATC}.
\begin{defi}
    A \emph{discrete graph} $G$ consists of the data $(F(G),r_G,\iota_G)$ where:
    \begin{itemize}
        \item $F(G)$ is a finite set.
        \item $r_G:F(G)\to F(G)$ is a map of sets (the root map).
        \item $\iota_G:F(G)\to F(G)$ is an involution, such that $\iota_G\circ r_G=r_G$.
    \end{itemize}
The \emph{set of vertices} of $G$ is the subset $V(G)$ of $F(G)$ given by $\Im(r_G)$, and its complement is denoted by $H(G)$.
\end{defi}

Suppose $G$ is a discrete graph. The involution $\iota_G$ of $F(G)$ restricts to an involution of $H(G)$, and hence partitions $H(G)$ into orbits of size $1$ and $2$. The following definitions will be used:
\begin{itemize}
    \item A \emph{leg} of $G$ is a size $1$ orbit of $H(G)$. The set of legs will be denoted by $L(G)$, and the \emph{boundary of a leg} $\ell\in L(G)$ is the subset $\partial e :=r_G(\ell) \subset V(G)$.
    \item An \emph{edge} of $G$ is a size $2$ orbit of $H(G)$. The set of edges will be denoted by $E(G)$, and the \emph{boundary of an edge} $e\in E(G)$ is the subset $\partial e:= r_G(e)\subset V(G)$.
    \item A vertex is \emph{incident} to an edge or a leg if it belongs to its boundary.
    \item The \emph{valency} of $V\in V(G)$ is simply $\val(V) := \# r_G^{-1}(V)-1$. 
    \item A \emph{subgraph} $K$ of $G$ is a graph $(F(K),r_K,\iota_K)$ such that $F(K)\subset F(G)$ with $r_K=r_G|_{F(K)}$ and $\iota_K=\iota_G|_{F(K)}$.
    \item  The graph $G$ is \emph{connected}, if for any two vertices $V,W\in V(G)$ there exists a sequence $e_1,\dots,e_n\in E(G)$ such that:
        \begin{enumerate}
            \item $V\in \partial e_1$ and $W\in \partial e_n$.
            \item For any $1\leq i\leq n-1$, $\partial e_i\cap \partial e_{i+1} \neq \varnothing$.
        \end{enumerate}
    \item A \emph{connected component} of $G$ is a maximal connected subgraph. Any graph decomposes as the disjoint union of its components.
    \item The \emph{genus} of a connected graph $G$ is the number $g(G):=\#E(G)-\#V(G)+1$.
    \item A \emph{tree} is a connected graph of genus-$0$. A \emph{forest} is a graph whose connected components are trees. A tree is \emph{trivial}, if it consists of just a vertex. A forest is \emph{trivial}, if it consists of trivial trees.
\end{itemize}

We kindly refer the reader to \cite{DARB} for examples and depictions of the previous definitions. We now turn our attention towards maps of graphs.

\begin{defi}\label{defi: maps of graphs}
 Suppose $G_1$ and $G_2$ are discrete graphs. A \emph{map of graphs} $f\colon G_1\to G_2$ is a map of sets $f\colon F(G_1)\to F(G_2)$ that commutes with the root maps (i. e. $r_{G_2}\circ f = f\circ r_{G_1}$) and the involution maps (i. e. $\iota_{G_2}\circ f = f\circ \iota_{G_1}$). If a map of graphs is a bijection, then we call it a bijective map of graphs.
\end{defi}

It follows directly from the definition that a map of graphs cannot map a leg to an edge or vice versa. However, it is possible that a map of graphs contracts edges or legs into vertices.

\begin{const}\label{const: edge contraction}\label{lem: edge contraction}
    Suppose $G$ is a discrete graph and $e\in E(G)$. Let $F(G/e)$ denote the set obtained from $F(G)$ by identifying the subset $e\cup\partial e\subset F(G)$ into a single element $V_e$. This identification gives rise to a natural surjective map $p_e:F(G)\to F(G/e)$. Furthermore, the following maps are well-defined:
    \begin{align*}
        &r_{G/e}: F(G/e)\to F(G/e),& H\mapsto \begin{cases}
            p_e(r_G(H)), &\textnormal{if }H\neq V_e, \\
            r_{G/e}(V_e) = V_e,&\textnormal{if } H=V_e,\\
        \end{cases}\\
        &\iota_{G/e}:F(G/e)\to F(G/e),&H \mapsto \begin{cases}
            p_e(\iota_G(H)), &\textnormal{if }H\neq V_e,\\
            \iota_{G/e}(V_e) = V_e,&\textnormal{if }H=V_e.\\
        \end{cases}\\
    \end{align*}
    A direct computation shows that the maps $r_{G/e}$ and $\iota_{G/e}$ satisfy $\iota_{G/e}\circ r_{G/e} = r_{G/e}$. Therefore the triple $(F(G/e),r_{G/e},\iota_{G/e})$ is a discrete graph, which we call the \emph{contraction of $e$ in $G$} and denote by $G/e$. We remark that if $e_1,e_2\in E(G)$, then $(G/e_1)/e_2$ is naturally bijective to $(G/e_2)/e_1$, and the distinction between them only comes from the newly introduced vertices in the last step of the successive contractions.
\end{const}

\begin{defi}\label{defi: contraction of legless subgraph}
    Let $K$ be a legless subgraph of $G$, with $K=K_1\sqcup \dots \sqcup K_m$ its connected components. The \emph{contraction of $K$ in $G$} is the graph $G/K$ obtained by the succesive contractions of all the edges of $K_1, \dots ,K_m$, where the new vertex obtained from the contraction of all the edges of $K_i$ is denoted by $V_{K_i}$.
\end{defi}

Observe that if $K$ is a subgraph of $G$ then the contraction of $K$ in $G$ is a new graph $G/K$ whose set of vertices consists of $V(G)\backslash V(K)$ and one additional vertex from each component of $K$. If $G$ is connected, then for arbitrary $K$ $G/K$ is connected and $g(G/K) = g(G)-\sum_{i=1}^mg(K_i)$ where the $K_i$ are the connected components of $K$.

\begin{lem}\label{lem: map of graphs is contraction}
    Suppose $f\colon G_1\to G_2$ is a map of graphs such that:
    \begin{enumerate}
        \item For every $V\in V(G_2)$, the inverse image $f^{-1}(V)\subset F(G_1)$ is a legless tree.
        \item For every $e\in E(G_2)\cup L(G_2)$, the inverse image under $f$ consists of a single edge or leg.
    \end{enumerate}
    Then there is a natural bijective map $G_1/f^{-1}\left(V(G_2)\right)\to G_2$.
\end{lem}
\begin{proof}
    The reader is kindly referred to Lemma 2.2.5 of \cite{DARB}
\end{proof}

\subsection{Partially open integral polyhedral cones.} Let $N$ be a free abelian group of finite rank and consider the vector space $N_\bbR:=N\otimes_\bbZ\bbR$. The dual of $N$ is the free abelian group (of finite rank) $N^\vee = \Hom_\bbZ(N,\bbZ)$, and there is a natural morphism $N^\vee\to N_\bbR^\vee$ that gives rise to an isomorphism $N^\vee \otimes_\bbZ\bbR\cong N_\bbR^\vee$. Therefore, an element $f\in N^\vee$ gives rise to a closed and an open half-space of $N_\bbR$ respectively:
\begin{align*}
    H_f &\colon = \{ x\in N_\bbR \colon f(x)\geq 0\}, &H_f^o &\colon = \{ x\in N_\bbR\colon  f(x)>0\}.
\end{align*} 

\begin{defi}\label{defi: poic of Nr}
A \emph{partially open integral (polyhedral) cone of $N_\bbR$} is a non-empty subset $\sigma\subset N_\bbR$ given as a finite intersection of open and closed half-spaces. That is $\sigma = \bigcap_{i=0}^nH_{f_i} \cap \bigcap_{j=0}^m H_{g_j}^o,$
where $f_i\in N^\vee$ and $g_j\in N^\vee$ for $0\leq i\leq n$ and $0\leq j\leq m$.
\end{defi}

\begin{nota}
If $\sigma$ is a partially open integral cone of $N_\bbR$, then the relative interior of $\sigma$ is also a partially open integral cone of $N_\bbR$ which we denote by $\sigma^o$.
\end{nota}

\begin{defi}
Suppose $\sigma$ is a partially open integral cone of $N_\bbR$. A \emph{face} $\tau$ of $\sigma$ is a partially open integral cone given by $\tau = \sigma\cap H_f$, where $f\in N^\vee$ is such that $\sigma\subset H_{-f}$. We denote this situation by $\tau\leq \sigma$. If $\tau\leq \sigma$ and $\tau\neq \sigma$, then $\tau$ is called a \emph{proper face} and we denote this situation by $\tau\lneq\sigma$.
\end{defi}

\begin{nota}
     For a partially open integral cone (of $N_\bbR$) $\sigma$, we will denote by $\Lin(\sigma)$ the linear subspace generated by $\sigma$ and its dimension by $\dim(\sigma)$ (namely $\dim(\sigma) = \dim \Lin(\sigma)$).
\end{nota}

\begin{defi}\label{defi: poic}
A \emph{partially open integral (polyhedral) cone} is a pair $(\sigma, N)$ where 
\begin{itemize}
    \item $N$ is a free abelian group of finite rank, 
    \item $\sigma$ is a partially open integral cone of $N_\bbR$ with $\dim(\sigma)=\dim(N_\bbR)$.
\end{itemize}
A \emph{face} of a partially open integral cone $(\sigma,N)$ is $(\tau,\Lin_N(\tau))$ where $\tau$ is a face of $\sigma$ in $N_\bbR$ and $\Lin_N(\tau) = N\cap \Lin(\tau)$.
\end{defi}

\begin{remark}
    If $(\sigma,N)$ is a partially open integral cone, then $(\sigma^o,N)$ is also a partially open integral cone.
\end{remark}

Hereafter, we will refer to a partially open integral cone simply as a poic. We remark that Definition \ref{defi: poic of Nr} is different from Definition \ref{defi: poic}, as the latter consists of a pair, whereas the former merely refers to a subset of an specified vector space. 

\begin{nota}
If the pair $(\sigma,N)$ is a poic, then $N$ will be denoted by $N^\sigma$, and we will denote the pair just by $\sigma$. Suppose that $\xi$ is an additional poic, then there is a natural identification $N^\sigma_\bbR\oplus N^\xi_\bbR \cong (N^\sigma\oplus N^\xi)_\bbR$. It is clear that the product $\sigma\times \xi$ is a partially open integral polyhedral cone of $N^\sigma_\bbR\oplus N^\xi_\bbR$ and, abusing the previous natural identification, we denote by $\sigma\times\xi$ the corresponding cone of $(N^\sigma\oplus N^\xi)_\bbR$. Hence, we regard $\sigma\times \xi$ as a poic with $N^{\sigma\times\xi}=N^\sigma\oplus N^\xi$.
\end{nota}

\begin{defi}
Let $\sigma$ and $\xi$ be two poics. A \emph{morphism} $f\colon\sigma\to \xi$ consists of an integral linear map $f\colon  N^\sigma\to N^\xi$, such that the induced linear map $f_\bbR:N^\sigma_\bbR\to N^\xi_\bbR$ sends $\sigma$ to $\xi$. We say that a morphism $f\colon \sigma\to \xi$ is a \emph{face-embedding}, if the integral linear map is injective and $f_\bbR(\sigma)\leq \xi$. 
\end{defi}

\begin{remark}
    If $\sigma$ and $\xi$ are two poics, then both projections $\sigma\times\xi\to\sigma$ and $\sigma\times\xi\to \xi$ are morphisms of poics. If $\tau$ is a face of $\sigma$ then the inclusion $\tau\to \sigma$ is a morphism (and undoubtedly a face-embedding) where the morphism between the lattices is the natural inclusion.
\end{remark}

\begin{nota}
For a finite set $A$, we let $\bbR_{\geq0}^A$ denote the poic given by the non-negative orthant in the vector space $(\bbZ^A)_\bbR$ and the lattice $\bbZ^A$. We let $\bbR_{>0}^A$ denote its corresponding relative interior. It is clear that subsets $B\subset A$ give rise to face-embeddings $\bbR_{\geq0}^B\to \bbR_{\geq0}^A$.   
\end{nota}

The data of poics and morphisms thereof, together with natural composition of maps, defines a category, which we denote by $\POIC$. A poic $\sigma$ carries naturally an underlying topological space: the cone itself. More precisely, if $(\sigma,N^\sigma)$ is a poic, let $|\sigma|$ denote the topological space given by $\sigma\subset N_\bbR$ with the Euclidean topology. If $\xi$ is also a poic and $f\colon \sigma\to\xi$ is a morphism, then $f$ induces a linear map $f\colon N^\sigma_\bbR\to N^\xi_\bbR$ which restricts to a continuous map
$|f|\colon |\sigma|\to|\xi|$. This association gives rise to a functor
\begin{equation}
    |\bullet|\colon \POIC\to \Top,\label{eqdefi: realization of poics}
\end{equation}
which we call the \emph{realization functor}. Furthermore, if $\sigma$ is a poic then we call $|\sigma|$ the \emph{realization of $\sigma$} (analogously with morphisms). For brevity, we will refer to a morphism of poics simply as a poic-morphism.
\subsection{Categories of graphs and cones of metrics.} \label{subsect: cones of metrics of graphs.} We describe a category of graphs that is central to our constructions. Then we use this category to describe the moduli space of genus-$g$ $A$-marked tropical curves, for a non-negative integer $g$ and a finite set $A$ such that $2g+\#A-2>0$.
\begin{defi}
Let $A$ be a finite set. A \emph{discrete graph with $A$-marked legs} is a discrete graph $G$ with a bijection 
\begin{equation*}
\ell_\bullet(G)\colon A\to L(G),i\mapsto \ell_i(G),
\end{equation*} 
which we will call the \emph{marking}. For $a\in A$, the \emph{$a$-leg of $G$} is simply the leg $\ell_a(G)\in L(G)$. A \emph{map of $A$-marked graphs} is a map of graphs that commutes with the respective markings.
\end{defi}

\begin{defi}\label{defi: morphs Ggn}
    A map of graphs $f\colon G_1\to G_2$ is a \emph{contraction} if it satisfies the hypotheses of Lemma \ref{lem: map of graphs is contraction}, that is:
    \begin{itemize}
        \item For every $V\in V(G_2)$ the subgraph defined by the inverse image $f^{-1}(V)\subset F(G_1)$ is a legless tree of $G_1$.
        \item For every $e\in E(G_2)\cup L(G_2)$, the inverse image under $f$ consists of a single edge or leg.
    \end{itemize}
If both graphs are $A$-marked (for a finite set $A$), then we say that $f$ is a \emph{contraction of $A$-marked graphs} if it is additionally a morphism of $A$-marked graphs. Naturally, composition of contractions is a contraction. In addition, any bijective map of graphs is a contraction.
\end{defi}
With the previous definition, the following category of graphs is introduced.
\begin{defi}
    For a finite set $A$ and a non-negative integer $g$ with $2g+\#A-2>0$, the \emph{category of $A$-marked graphs of genus-$g$} is the category $\bbG_{g,A}$ specified by:
\begin{itemize}
    \item The class of objects consists of the connected genus-$g$ discrete graphs with $A$-marked legs whose vertices are at least $3$-valent.
    \item For two objects $G_1$ and $G_2$ of $\bbG_{g,A}$, the set of morphisms $\Hom_{\bbG_{g,A}}(G_1,G_2)$ consists of the contractions of $A$-marked graphs $f\colon G_1\to G_2$.
    \item Composition of morphisms is just composition of maps.
\end{itemize}
\end{defi}

For examples and depictions of these categories, we refer the reader to \cite{DARB}.

\begin{nota}
    If $g,n\geq0$ are integers with $2g+n-2>0$ and $A=\{1,\dots,n\}$, then we will denote the category $\bbG_{g,A}$ simply by $\bbG_{0,n}$.
\end{nota}

\begin{prop}\label{prop: Ggn iso classes}
    Any object of $\bbG_{0,A}$ has no non-trivial automorphisms and the skeleton of this category is a finite poset. For positive genus there are finitely many isomorphism classes in $\bbG_{g,A}$. 
\end{prop}
\begin{proof}
    Both statements can be directly verified, and are well-known facts.
\end{proof}

\begin{const}
    A \emph{metric} on a discrete graph $G$ is a function $\delta\colon  E(G)\to \bbR_{>0}$. 
   If $\delta$ is a metric on $G$, then the pair $(G,\delta)$ gives rise to a metric space as follows: To each edge is associated a closed interval of length prescribed by $\delta$, to each leg is associated a half line $\bbR_{\geq 0}$, and all these are glued according to the incidence relations of $G$.
\end{const}

\begin{defi}
   Let $A$ be a finite set. An \emph{$A$-marked tropical curve} $\Gamma$ is a metric space so arising from a pair $(G,\delta)$, where $G$ is an $A$-marked discrete graph and $\delta$ is a metric on $G$.
\end{defi}

\begin{remark}
     By definition, a metric on $G$ corresponds uniquely to a point of $\bbR_{>0}^{E(G)}$.
\end{remark}

\begin{const}\label{const: face-embeddings of subgraphs}
    Let $G$ be as above and suppose $K$ is an unmarked subgraph of $G$ whose connected components are non-trivial trees, then $G/K$ is an object of $\bbG_{g,A}$. Furthermore, $E(G/K)=E(G)\backslash K$, and the inclusion $E(G/K)\subset E(G)$ gives rise to a face-embedding
    \begin{equation*}
        \bbR^{E(G/K)}_{\geq 0}\to \bbR^{E(G)}_{\geq0}.
    \end{equation*}
    The \emph{cone of metrics} of $G$ is the poic $\sigma_G$ defined as the subcone of $\bbR_{\geq0}^{E(G)}$ that contains $\bbR^{E(G)}_{>0}$ and all the images of $\bbR_{>0}^{E(G/K)}$, where $K$ is as above, under the corresponding face-embeddings. 
\end{const}

\begin{remark}
    We regard the cone of metrics of $G$ as the poic $\sigma_G$ lying inside of $\bbR_{\geq0}^{E(G)}$ consisting of all the metrics of $G$ together with those of all its possible genus preserving contractions.
\end{remark}

\begin{const}
    Suppose $G^\prime$ is an additional object of $\bbG_{g,A}$ and $f\in \Hom_{\bbG_{g,A}}(G^\prime,G)$. By definition $f^{-1}(V(G))$ is an unmarked subgraph of $G^\prime$ (with possibly trivial connected components), let $K$ denote the subgraph of $G^\prime$ defined by removing isolated vertices from $f^{-1}(V(G))$. From Lemma \ref{lem: map of graphs is contraction}, it follows that $G^\prime/K\cong G$ and therefore the morphism $f$ gives rise to:
    \begin{itemize}
            \item A face-embedding $\iota_K\colon \sigma_{G^\prime/K}\to \sigma_{G^\prime}$, given by the natural inclusion $E(G^\prime/K)\subset E(G^\prime)$.
            \item A bijection $E(G)\to E(G^\prime/K)$ which gives rise to an isomorphism of poics $L_f\colon \sigma_G\to \sigma_{G^\prime/K}$.
    \end{itemize}
    The composition $\iota_K\circ L_f$ is then a face-embedding, which we will denote by $\iota_f\colon\sigma_G\to\sigma_{G^\prime}$. It is shown in Lemma 2.4.3 of \cite{DARB} that the cone of metrics construction gives rise to a functor $\Mtrop_{g,A}\colon \bbG_{g,A}^{\op}\to \textnormal{POIC},\label{eq: moduli space is poic-space}$ defined as follows: 
    \begin{itemize}
        \item At an object $G$ of $\bbG_{g,n}$, $\Mtrop_{g,n}(G)$ is defined to be the poic $\sigma_G$,
        \item At a morphism $f\in\Hom_{\bbG_{g,n}}(G^\prime,G)$ the poic-morphism $\Mtrop_{g,n}(f)$ is defined as the face-embedding $\iota_f\colon \sigma_G\to \sigma_{G^\prime}$.
    \end{itemize}
    It is also shown there that if $\tau$ is a face of $\sigma_{G}$, then there exists up to isomorphism a unique morphism $h\in \Hom_{\bbG_{g,n}}(G,H)$ such that $\Mtrop_{g,n}(h)\colon \sigma_H\to\sigma_G$ is isomorphic to $\tau\leq \sigma_G$.
\end{const}

Since $\bbG_{g,n}^{\op}$ has finitely many isomorphism classes, it is necessarily equivalent to a finite category. As small colimits are representable in $\Top$, this implies that the colimit of any functor from $\bbG_{g,n}^{\op}$ to $\Top$ is representable.

\begin{defi}
    For a finite set $A$ and an integer $g\geq0$ with $2g+\#A-2>0$, the \emph{moduli space of genus-$g$ $A$-marked tropical curves} is the topological space 
    \begin{equation*}
        \cM_{g,A}^\trop \colon = \colim_{\bbG_{g,A}^{\op}} \left|\Mtrop_{g,n}\right|.\end{equation*}    
\end{defi}

After making a choice for a full set of representatives for $\bbG_{g,A}^{\op}$, the above construction shows that $\cM_{g,A}^\trop$ consists of taking the cones of metrics and identifying points representing isometric genus-$g$ $A$-marked tropical curves. This presentation is useful for our purposes as it emphasizes the combinatorial relations between cones of metrics that lie at the core of our program: face relations by edge contractions and point identifications due to isometry.

\subsection{Poic-spaces and poic-complexes.} The idea for poic-spaces originates from spaces obtained by glueing finitely many cells given as quotients of partially open cones by actions of finite groups (preserving their integral structure). Of special interest are those where just finitely many partially open cones are glued with respect to face relations. Our motivating and leading class of examples consists of the spaces $\cM_{g,n}^\trop$ (for integers $g,n\geq0$ with $2g+n-2>0$).

\begin{defi}
    A \emph{poic-space} $\cX$ is a functor $\cX\colon C_\cX\to \POIC$, where $C_\cX$ is an essentially finite category with finite hom-sets, and $\cX$ is subject to the following conditions:
    \begin{itemize}
        \item The functor $\cX$ maps morphisms of $C_\cX$ to face-embeddings.
        \item If $x$ is an object of $C_\cX$ and $\tau$ is a face of $\cX(x)$, then there exists a unique up to isomorphism $w\to x$ in $C_\cX$ with $\cX(w\to x) \cong \tau\leq \cX(x)$.
    \end{itemize} 
\end{defi}

When no confusion seems near, we will refer to objects of $C_\cX$ as cones of $\cX$, and to the morphisms of $C_\cX$ as morphisms of $\cX$. For a cone $q$ of $C_\cX$ we will denote the dimension $\dim (\cX(q))$ simply by $\dim(q)$. We denote the set of isomorphism classes of $C_\cX$ by $[\cX]$, and the set of isomorphism classes of $k$-dimensional cones of $\cX$ by $[\cX](k)$ (these form a partition of $[\cX]$). The poic-space $\cX$ is called \emph{pure of dimension $n$} if every cone $t$ of $\cX$ appears as the source of a morphism to an $n$-dimensional cone of $\cX$.

\begin{defi}
    Suppose $\cX$ and $\cY$ are poic-spaces. A \emph{morphism of poic-spaces} $F\colon \cX\to\cY$ consists of the following data: 
    \begin{itemize}
        \item A functor $F\colon C_\cX\to C_\cY$.
        \item A natural transformation $\eta_F\colon \cX\implies \cY\circ F$.
    \end{itemize}
    These are subject to the following condition: for a cone $x$ of $\cX$, the cone $\eta_{F,x}\left(\cX(x)\right)$ does not lie in a proper face of $\cY(F(x))$.\\
    Composition of morphisms of poic-spaces is defined as follows. Let $\cZ$ denote an additional poic-space and $G\colon \cY\to\cZ$ a morphism. Then $G\circ F\colon \cX\to \cZ$ is given by the functor $G\circ F$, together with the natural transformation $\eta_{G\circ F}\colon \cX\implies {\cZ}\circ G\circ F$ defined at a cone $s$ of $\cX$ by
\begin{equation*}
    \eta_{G\circ F,s} := \eta_{G,F(s)}\circ \eta_{F,s}\colon \cX(s)\to \cZ(G\circ F(s)).
\end{equation*}
\end{defi} 

\begin{nota}
Let $F\colon \cX\to\cY$ be a morphism of poic-spaces, and let $s$ be a cone of $\cX$. To simplify notation we will denote the poic-morphism $\eta_{F,s}$ simply by $F_s$.
\end{nota}

\begin{const} Suppose $\cX$ and $\cY$ are poic-spaces. The product of poics is a poic, and the functors $\cX$ and $\cY$ give rise to the functor
\begin{equation*}{\cX\times\cY}\colon C_\cX\times C_\cY\to \POIC, (q,s)\mapsto \cX(q)\times\cY(s).\end{equation*}
It is shown in \cite{DARB} that this is an actual poic-space.
\end{const}

\begin{defi}
    A poic-space $\cZ$ is a \emph{poic-subspace} of a poic-space $\cX$, if $C_\cZ$ is a full subcategory of $\cX$ and $\cZ=\cX|_{C_\cZ}$.
\end{defi}

We remark that poic-spaces and their morphisms form a category, which we denote by $\POICSpcs$. The functor \eqref{eqdefi: realization of poics} can be extended naturally to poic-spaces, thereby giving rise to the \emph{realization functor}:
\begin{equation}
    |\bullet|\colon \POICSpcs\to\Top,\cX\mapsto \colim |\cX|.
\end{equation}

\begin{defi}
A \emph{poic-complex} $\Phi$ is a poic-space, where the underlying category $C_\Phi$ is essentially finite and thin. We denote by $\POICCmplxs$ the full subcategory of $\POICSpcs$ consisting of poic-complexes. Namely, a morphism of poic-complexes is just a morphism of the underlying poic-spaces. If $\Phi$ is a poic-complex and $\Upsilon$ is a poic-subspace, then $\Upsilon$ is also a poic-complex. We say then that $\Upsilon$ is a \emph{poic-subcomplex} of $\Phi$.
\end{defi}

\begin{const}
Suppose $N$ is a free abelian group of finite rank. The pair $(N_\bbR,N)$ is naturally a poic. Since the cone $N_\bbR$ has no faces, $\{(N_\bbR,N)\}$ is vacuously a poic-complex (the functor being tautological in this case). We will denote this poic-complex by $\underline{N_\bbR}$. If $M$ is an additional free abelian group of finite rank, then under the used notation $\underline{N_\bbR}\times \underline{M_\bbR} = \underline{(N\oplus M)_\bbR}$.
\end{const}

\begin{example}
    Proposition \ref{prop: Ggn iso classes} shows that for $n\geq 3$, the functor $\Mtrop_{0,n}$ is a poic-complex.
\end{example}

\begin{defi}\label{defi: general weights}
Suppose $\Phi$ is a poic-complex. A \emph{$k$-dimensional weight $\omega$ on $\Phi$} is simply a function $\omega\colon  [\Phi](k)\to \bbZ$. We denote the set of $k$-dimensional weights on $\Phi$ by $W_k(\Phi)$. This set is naturally an abelian group, and since $[\Phi](k)$ is necessarily a finite set, it is free of finite rank.
\end{defi}

\begin{defi}
A \emph{linear poic-complex} is a tuple $(\Phi,X)$ such that:
\begin{itemize}
    \item $\Phi$ is a poic-complex. 
    \item $X\colon \Phi\to \underline{(N_X)_\bbR}$ is a morphism of poic-complexes, where $N_X$ is a free abelian group of finite rank.
\end{itemize}
We will denote the linear poic-complex $(\Phi,X)$ simply by $\Phi_X$. If $\Phi_X$ and $\Psi_Y$ are linear poic-complexes, then a \emph{morphism of linear poic-complexes} $\phi:\Phi_X\to\Psi_Y$ consists of the following data: 
\begin{itemize}
    \item A morphism of poic-complexes $\phi\colon\Phi\to\Psi$
    \item An integral linear map $\phi_{\tint}\colon N_X\to N_Y$ (this is equivalent to a morphism of poic-complexes $\underline{\left(N_X\right)_\bbR}\to \underline{\left(N_Y\right)_\bbR}$).
    \item These yield a commutative square of poic-complexes.
\end{itemize}

\end{defi}
Clearly, if $\Phi_X$ is a linear poic-complex and $\Upsilon$ is a poic-subcomplex of $\Phi$, then $(\Upsilon,X|_\Upsilon)$ is also a linear poic-complex. In this case, we say that $\Upsilon_{X|_\Upsilon}$ (or simply $\Upsilon_X$) is a \emph{linear poic-subcomplex} of $\Phi_X$.

\begin{nota}
We use the following notation for a poic-complex $\Phi$:
\begin{itemize}
    \item If $s$ is a a cone of $\Phi$, then we denote the lattice $N^{\Phi(s)}$ simply by $N^s$.
    \item If $s$ is a cone of $\Phi$, then $\Star^1_\Phi(s)$ consists of the set of isomorphisms classes of morphisms $s\to t$ of $\Phi$, where $\dim(t)=\dim(s)+1$. The isomorphism class of such a morphism will be denoted by $[s\to t]$.
\end{itemize}
\end{nota}

\begin{const}[Relative linear maps]\label{const: rel lin maps}

Suppose $\Phi_X$ is a linear poic-complex and consider a morphism $s\to t$ of $\Phi$, with $\Phi(s)\leq \Phi(t)$ a codimension $1$ face. 
The quotient $N^t/N^s$ is a rank $1$ lattice and the morphism $A$ induces a linear map
\begin{equation} 
X_{s\to t}\colon  N^t/N^s\to N_X/X_s(N^{s}).\label{eq: relative linear maps}
\end{equation}
If $s\to t^\prime$ is an additional morphism of $\Phi$ isomorphic to $s\to t$, then it is shown in Lemma 3.2.8 of \cite{DARB} that under the maps $X_{s\to t}$ and $X_{s\to t^\prime}$ the images of the generators of $N^t/N^s$ and $N^{t^\prime}/N^s$ lying correspondingly in the projections of $\Phi(t)$ and $\Phi(t^\prime)$ coincide, and hence define a unique element of $N_X/X_s(N^s)$. We denote this element by $u^X_{[s\to t]}$ and call it the \emph{normal vector $u^X_{[s\to t]}$} associated to the class $[s\to t]\in\Star^1_\Phi(s)$.
\end{const}

\begin{defi}\label{defi: balanced weight}
Suppose $\Phi_X$ is a linear poic-complex and let $s$ be a cone of $\Phi$ with $d(s)=k-1$. A $k$-dimensional weight $\omega\in W_k(\Phi)$ is \emph{$X$-balanced at $s$}, if 
\begin{equation*}
    \sum_{[s\to t]\in\Star^1_\Phi(s)} \omega([t]) u_{[s\to t]}^X =0.
\end{equation*}
It is said to be \emph{$X$-balanced}, if it is $X$-balanced at every object of $\Phi$ of dimension $k-1$. A weight is $X$-balanced if and only if it is $X$-balanced at a set of representatives of $[\Phi](k-1)$ (Lemma 3.2.11 of loc. cit.). The set of \emph{$k$-dimensional Minkowski weights} is the subset $M_k(\Phi_X)$ of $W_k(\Phi)$ consisting of the $X$-balanced weights $\omega\in W_k(\Phi)$, which is actually a subgroup of $W_k(\Phi)$.
\end{defi}

Tropical cycles on a linear complex are be defined by means of subdivisions. For the purposes of this article, we will only provide the definition of a subdivision and kindly point the reader toward \cite{DARB} for every related detail. 

\begin{defi}
Let $\Phi$ be a poic-complex. A \emph{subdivision} $S\colon \Phi^\prime\to\Phi$ is a morphism of poic-complexes $S\colon \Phi^\prime\to\Phi$, such that: 
\begin{itemize}
    \item The underlying functor $S\colon C_{\Phi^\prime}\to C_\Phi$ is essentially surjective, and the underlying integral linear maps of the natural transformation $\eta_S\colon \Phi^\prime\implies \Phi\circ S$ are injective. Furthermore, if $t$ is a cone of $\Phi^\prime$ such that $d(t) = d(S(t))$, then the underlying integral linear map of $S_t\colon \Phi^\prime(t)\to \Phi(S(t))$ is an isomorphism.
    \item For every cone $s$ of $\Phi$ lying in the image of $S$, the relative interior of $\Phi(s)$ is the disjoint union of the images of the relative interiors of the cones $p$ of $\Phi^\prime$ with $S(p)\cong s$, and the images of these relative interiors characterize these cones of $\Phi^\prime$ up to isomorphism (if they coincide, then they are isomorphic in $C_{\Phi^\prime}$).
    \item For every cone $t$ of $\Phi^\prime$ and every $s\to S(t)$, there exists $q\to t$ in $\Phi^\prime$ such that $S(q\to t) \cong s\to S(t)$.
\end{itemize}
The \emph{category of subdivisions of $\Phi$} is the category $\Subd(\Phi)$ whose class of objects consists of subdivisions $S\colon \Phi^\prime\to \Phi$ and morphisms are commuting subdivision morphisms.
\end{defi}

Let $k\geq0$ be an integer. If $\Phi_X$ is a linear poic-complex and $S\colon \Phi^\prime\to\Phi$ is a subdivision, then $S$ induces an injective linear map $S^*\colon M_k(\Phi_X)\to M_k(\Phi^\prime_{X\circ S})$. Furthermore, by means of these maps, the Minkowski weights give rise to a functor
    \begin{equation} M_k(\bullet)\colon  \Subd(\Phi)^{\op} \to \textnormal{Ab},\label{eq: minkowski weights functor}\end{equation}
    whose colimit is representable.
\begin{defi}
For a linear poic-complex $\Phi_X$, the \emph{group of tropical $k$-cycles} $Z_k(\Phi_X)$ is defined as the colimit of \eqref{eq: minkowski weights functor}.
\end{defi}


\begin{defi}
    A linear poic-complex $\Phi_X$ pure of dimension $n$ is called \emph{irreducible}, if $Z_n(\Phi_X)$ is free of rank $1$.
\end{defi}

It is shown in Proposition 3.4.12 of \cite{DARB} that if $\Phi_X$ and $\Psi_Y$ are irreducible linear poic-complexes with $\Phi$ pure of dimension $n$ and $\Psi$ pure of dimension $m$, then $(\Phi\times\Psi)_{X\times Y}$ is also irreducible.

\subsection{Poic-fibrations and their cycles.}
We make use of poic-fibrations to describe cycles of a poic-space. More specifically, we speak of equivariant tropical cycles with respect to the poic-fibration. These poic-fibrations are special kind of morphisms from a (linear) poic-complex to a poic-space with certain lifting properties. In what follows, suppose $\Phi$ is a poic-complex and $\cX$ is a poic-space.
\begin{defi}
     A morphism of poic-spaces $\pi\colon \Phi\to \cX$ is called a \emph{poic-fibration}, if:
\begin{enumerate}
    \item The functor is essentially surjective.
    \item For any cone $p$ of $\Phi$ the map $\eta_{\pi,p}\colon  \Phi(p)\to \cX\left( \pi(p)\right)$ induces an isomorphism between their relative interiors.
    \item For any object $p$ of $\Phi$ and any morphism $f\colon \pi(p)\to x$ in $\cX$, there exists a morphism $h\colon p\to q$ in $\Phi$ and an isomorphism $g\colon  \pi(q)\to x$ such that $f=g\circ \pi(h)$, where $h$ is unique up to isomorphism under $p$, and $g$ is unique up to isomorphism over $x$.
\end{enumerate}
If $\Phi_X$ is additionally a linear poic-complex, and $\pi\colon\Phi\to\cX$ is a poic-fibration, then we say that $\pi\colon\Phi_X\to\cX$ is a linear poic-fibration. A morphism of poic-fibrations is simply a pair of morphisms yielding a commutative square. However, for convenience of notation, we use the following convention. Suppose $\rho\colon \Psi\to \cY$ is an additional poic-fibration. A morphism of poic-fibrations $\mathfrak{f}\colon \pi\to\rho$ consists of: 
\begin{itemize}
    \item A morphism of poic-complexes $\mathfrak{f}^\pcmplxs\colon \Phi\to\Psi$ (the superscript stands for complexes).
    \item A morphism of poic-spaces $\mathfrak{f}^\pspcs\colon\cX\to \cY$ (the superscript stands for spaces).
    \item These are subject to the condition $\rho\circ \mathfrak{f}^\pcmplxs = \mathfrak{f}^\pspcs\circ \pi$.
\end{itemize}
If, in addition, $\pi\colon \Phi_X\to\cX$ and $\rho\colon\Psi_Y\to\cY$ are linear poic-fibrations, then $\mathfrak{f}^\pcmplxs$ is also a morphism linear poic-complexes. 
\end{defi}

Let $\pi\colon\Phi_X\to\cX$ be a linear poic-fibration. The first and second condition imply that the fibration $\pi$ induces a surjective map of sets
    \begin{equation*}
    [\pi]\colon [\Phi](k)\to [\cX](k).
\end{equation*} 

\begin{defi}
    Let $k\geq 0$ be an integer. A \emph{$k$-dimensional $\pi$-equivariant Minkowski weight} $\omega$ is a weight $\omega\in M_k(\Phi_X)$ such that $\omega$ is constant on the fibers of $[\pi]$. The set of $\pi$-equivariant Minkowski weights is denoted by $M_k(\cX_{\pi,X})$, which is in fact a subgroup of $M_k(\Phi_X)$.
\end{defi}

As in the case of linear poic-complexes subdivisions of $\Phi$ are necessary to $\pi$-equivariant cycles. However, we are not interested in general subdivisions of $\Phi_X$, but rather on subdivisions compatible with the poic-fibration. We call these \emph{$\pi$-compatible subdivisions} and introduce them in Definition 4.3.3 of \cite{DARB}. The following facts can be found in loc. cit.:
\begin{itemize}
    \item If $S\colon\Phi^\prime\to\Phi$ is a subdivision, then there exists a subdivision $S^{\prime}\colon\Phi^{\prime\prime}\to\Phi^\prime$ such that $S\circ S^\prime\colon\Phi^{\prime\prime}\to\Phi$ is $\pi$-compatible (Lemma 4.3.4).
    \item For a $\pi$-compatible subdivision $S\colon\Phi^\prime\to\Phi$, being a $\pi$-equivariant Minkowski weight is equivalent to being invariant under a finite number of permutations of $[\Phi^\prime](k)$, and $M_k(\cX_{\pi,X\circ S})$ is an abelian subgroup of $M_k(\Phi^\prime_{X\circ S})$ (Lemma 4.3.7).
    \item If $S\colon\Phi^\prime\to\Phi$ and $S^\prime\colon \Phi^{\prime\prime}\to \Phi^\prime$ are subdivisions, such that $S$ and $S^\prime\circ S$ are $\pi$-compatible, then the map $(S^\prime)^*\colon M_k(\Phi^\prime_{X\circ S})\hookrightarrow M_k(\Phi^{\prime\prime}_{X\circ S^\prime})$ maps $\pi$-equivariant Minkowski weights to $\pi$-equivariant Minkowski weights for any integer $k\geq0$ (Lemma 4.3.8).
    \item If $\pi\textnormal{-}\Subd(\Phi)$ denotes the full subcategory of $\pi$-compatible subdivisions of $\Phi$, then for any integer $k\geq0$ the $k$-dimensional $\pi$-equivariant Minkowski weights give rise to a representable functor $M_k(\cX_{\pi,\bullet})\colon          \pi\textnormal{-}\Subd(\Phi)^{\op}\to \text{Ab}$.
\end{itemize}
\begin{defi}
    If $\pi:\Phi_X\to\cX$ is a poic-fibration, the \emph{group of tropical $k$-cycles of the fibration} is defined as the colimit of $M_k(\cX_{\pi,\bullet})$ over $\pi\textnormal{-}\Subd(\Phi)^{\op}$.
\end{defi}

It is also shown in Lemma 4.3.11 of loc. cit. that if $\pi\colon\Phi_X\to\cX$ is a linear poic-fibration, then the natural inclusions give rise to an injective linear map $\pi_k^*\colon Z_k(\cX_{\pi,X})\to Z_k(\Phi_X)$.

\begin{defi}
    A linear poic-fibration $\pi\colon \Phi_X\to \cX$ (with $\Phi_X$ pure of dimension $n$) is called \emph{irreducible}, if $Z_n(\cX_{\pi,X})$ is free of rank $1$.
\end{defi}

\subsection{Technichalities on pushforwards.} We want to pushforward Minkowski weights tropical cycles through morphisms. It is necessary to introduce the notion of weakly proper morphism to do so, since we are dealing with partially open cones and a general morphism does not automatically reflect the face relations of the target (which are relevant for balancing). Moreover, for the purposes of this paper we introduce the slightly more general notion of weakly proper morphism in a given dimension. We first discuss pushforwards in the context of (linear) poic-complexes, and then pushforwards in the context of linear poic-fibrations.

\begin{defi}
    Let $k\geq0$ be an integer. A morphism of poic-complexes $P\colon \Phi\to\Psi$ is called \emph{weakly proper in dimension $k$}, if for every $k$-dimensional cone $s$ of $\Phi$ and every facet $\tau$ in the closure of $P_s(\Phi(s))$ in $\Psi(P(s))$, there exists a facet $q\to s$ of $\Phi$ which under the morphism $P$ gives the facet $\tau$. A morphism of linear poic-complexes is called \emph{weakly proper in dimension $k$}, if the underlying morphism of poic-complexes is weakly proper in dimension $k$.
\end{defi}

\begin{const}\label{const: pushforward general weights}
    Consider two poic-complexes $\Phi$ and $\Psi$ together with a morphism of linear poic-complexes $\phi\colon \Phi\to\Psi$. To pushforward weights of $\Phi$ to weights of $\Psi$, we may need to pass to a subdivision of $\Psi$. This is described in detail in Section 3.5 of \cite{DARB}. For the purposes of this article we content ourselves with the assumption that $\Psi$ is sufficiently subdivided. This means that we assume the following: for every cone $s$ of $\Phi$ for which the integral linear part of $\phi_s$ is injective, there exists a cone $t$ of $\Psi$ such that: 
   \begin{itemize}
       \item There is an isomorphism $f\colon \phi(s) \xrightarrow{\sim} t$ of $\Psi$.
       \item The images of $\Psi(t)^o$ and $\Phi(s)^o$ under the linear part of $\Psi(f)\circ \phi_s$ coincide.
   \end{itemize}    
   In this case, it is clear that $\dim(t)=\dim(s)$ and such a $t$ is unique up to isomorphism. Following the notation, we write $(\phi)_*([s])= [t]$, for the classes of $s$ and $t$ as above. Let $k\geq0$ be an integer, and suppose $\omega\in W_k(\Phi)$. The \emph{pushforward of $\omega$} along $\phi$ is the weight
    \begin{equation*}
    (\phi)_*\omega\colon  [\Psi](k) \to \bbZ,[t]\mapsto \sum_{\substack{[p]\in [\Phi](k),\\
    (\phi)_*([p]) = [t]}}\omega([p])[N^{t}\colon  (F_\phi)_p N^p].
    \end{equation*}
\end{const}
\begin{lem}\label{lem: pushforward}
   Suppose $\Psi_Y$ is a linear poic-complex and $P\colon \Phi\to\Psi$ is a morphism of poic-complexes that is weakly proper in dimension $k$ for an integer $k\geq0$. If $\Psi$ is subdivided enough as in Construction \ref{const: pushforward general weights} and $\omega\in M_k(\Phi_{Y\circ P})$, then $P_*\omega\in M_k(\Psi_Y)$.
\end{lem}
\begin{proof}
    The proof of this Lemma follows verbatim that of Lemma 3.5.5 of \cite{DARB}, hence we omit it. 
\end{proof}

\begin{defi}
Let $k\geq0$ be an integer. A morphism of linear poic-complexes $\phi\colon \Phi_X\to \Psi_Y$ is called \emph{weakly proper in dimension $k$}, if the morphism of poic-complexes $\phi\colon \Phi\to\Psi$ is weakly proper in dimension $k$. In this case, the map $\phi_\tint$ induces a linear map $(\phi_\tint)_*\colon  M_k(\Phi_X)\to M_k(\Phi_{Y\circ \phi})$, and if $\Psi$ is subdivided enough as in Construction \ref{const: pushforward general weights}, then precomposition with $(\phi_\tint)_*$ induces the map $({\phi})_*\colon M_k(\Phi_{X})\to M_k(\Psi_{Y})$. In addition, the inclusion $M_k(\Psi_Y)\hookrightarrow Z_k(\Psi_Y)$ and the previous map define the \emph{weak pushforward map in dimension $k$}
\begin{equation}
    ({\phi})_*\colon M_k(\Phi_{X})\to Z_k(\Psi_{Y}). \label{eq: weak pushforward in dimension k}
\end{equation}
More generally, following Lemma 3.5.6 of loc. cit. the morphism \eqref{eq: weak pushforward in dimension k} is always defined (it is not necessary for $\Psi$ to be subdivided enough).
\end{defi}

We now describe the situation with poic-fibrations.

\begin{defi}\label{defi: proper morphism poic-fibrations}
    Suppose $\pi\colon\Phi\to\cX$ and $\rho\colon\Psi\to\cY$ are poic-fibrations.  A morphism of poic-fibrations $\mathfrak{f}\colon \pi\to\rho$ is called \emph{weakly proper in dimension $k$} if:
    \begin{enumerate}
        \item $\mathfrak{f}^\pcmplxs$ is a weakly proper in dimension $k$ morphism of poic-complexes.
        \item $\mathfrak{f}^\pspcs$ satisfies the following lifting property: if $f\colon s \to t$ is an isomorphism of $\cY$, then for every cone $s^\prime$ of $\cX$ with $\mathfrak{f}^\pspcs(s^\prime)=s$ there exists a unique isomorphism $f^\prime\colon s^\prime\to t^\prime$ of $\cX$ with $\mathfrak{f}^\pspcs(f^\prime)=f$ (in particular $\mathfrak{f}^\pspcs(t^\prime)=t$).
    \end{enumerate}
\end{defi}

Similar to the case of linear poic-complexes we must move to further subdivisions of this poic-complex and these must also be compatible with the poic-fibration $\rho$. For the sake of sparing numerous details and technicalities, for which we refer the reader to Section 4.4 of \cite{DARB}, we remark that if $k\geq0$ is an integer and $\mathfrak{f}\colon \pi_X\to\rho_Y$ is morphism of linear poic-fibrations weakly proper in dimension $k$, then the pushforward \eqref{eq: weak pushforward in dimension k} defines a linear map
\begin{equation}
    (\mathfrak{f})_*\colon M_k(\cX_{\pi,X})\to Z_k(\cY_{\rho,Y}),\label{eq: pushforward minkowkski weights of fibration}
\end{equation}
which we also call the \emph{weak pushforward in dimension $k$}.
\subsection{Rational weights} This is a brief digression towards rational Minkowski weights, which are relevant in \ref{sec: standard weight}. We just expand on the case of poic-complexes, leaving the analogous case of poic-fibrations to the reader.

\begin{defi}\label{defi: rational mw}
    Suppose $\Phi$ is a poic-complex and let $k\geq0$ be an integer. The \emph{group of rational weights on $\Phi$} is defined simply as $W_k(\Phi)_\bbQ:= W_k(\Phi)\otimes_\bbZ\bbQ$. If $\Phi_X$ is additionally a linear poic-complex, we analogously define:
    
    \begin{itemize}
        \item The \emph{group of $k$-dimensional rational Minkowski weights} as  $M_k(\Phi_X)_\bbQ:=M_k(\Phi_X)\otimes_\bbZ \bbQ$.
        \item The \emph{group of $k$-dimensional rational tropical cycles} as $Z_k(\Phi_X)_\bbQ := Z_k(\Phi_X)\otimes_\bbZ \bbQ$.
    \end{itemize}
\end{defi}

\begin{lem}
    Suppose $\Phi_X$ is a linear poic-complex and $k\geq0$ is an integer. The natural maps $M_k(\Phi^\prime_{X\circ S})_\bbQ\to Z_k(\Phi_X)_\bbQ$, where $S\colon \Phi^\prime\to\Phi$ is a subdivision, induce an isomorphism $\varinjlim\limits_{\Subd(\Phi)} M_k(\bullet_X)_\bbQ\cong Z_k(\Phi_X)_\bbQ$.
\end{lem}
\begin{proof}
    The functor $M_k(\bullet_X)_\bbQ$ is naturally isomorphic to $M_k(\bullet_X)\otimes_\bbZ\bbQ$. Hence, these natural isomorphisms induce an isomorphim $\varinjlim\limits_{\Subd(\Phi)} M_k(\bullet_X)_\bbQ\cong Z_k(\Phi_X)_\bbQ$, which factors through the natural maps $M_k(\bullet_X)\otimes_\bbZ\bbQ \to Z_k(\Phi_X)\otimes_\bbZ\bbQ$.
\end{proof}

\begin{remark}
    The above definition has the same outcome as considering $\bbQ$-valued functions in Definitions \ref{defi: general weights} and \ref{defi: balanced weight}. This direction would define a functor with the same properties regarding subdivisions, and the analogous colimit would necessarily be isomorphic to the rational tropical cycles defined above.
\end{remark}

\subsection{Realization of moduli spaces of rational tropical curves.} \label{ssection: realization of mod spaces of rat trop curves}
Let $n\geq 3$ be an integer. We quickly explain, following  \cite{GathmannKerberMarkwigTFMSTC}, how the topological space $\cM_{0,n}^{\trop}$ can be embedded as an $(n-3)$-dimensional fan through the distance map. As previously mentioned, a point in this space consists of an $n$-marked metric graph $\Gamma$ (as a metric space), where we denote the marked legs by $\ell_1(\Gamma),\dots,\ell_n(\Gamma)$. We let $\dist_\Gamma$ denote the distance function of this metric space. Consider the vector space $\bbR^{\binom{n}{2}}$ indexed by the $2$-element subsets of $\{1,\dots,n\}$ together with the linear map $M_n:\bbR^n\to \bbR^{\binom{n}{2}},(x_i)_{1\leq i\leq n}\mapsto (x_i+x_j)_{\{i,j\}\subset\{1,\dots,n\}}$. Let $Q_n$ denote the cokernel of $M_n$. It is shown in Theorem 4.2 of \cite{SpeyerSturmfels} that the distance map 
\begin{equation}
    \dist: \cM^{\trop}_{0,n}\to Q_n, \Gamma\mapsto (\dist_\Gamma(\ell_i(\Gamma),\ell_j(\Gamma)))_{\{i,j\}\subset \{1,\dots,n\}}\label{eq: distance map}
\end{equation}
embeds the space $\cM^{\trop}_{0,n}$ as a simplicial fan of pure dimension $(n-3)$, where the $k$-dimesional cones are given by the combinatorial types. These are the same (by definition) as the isomorphism classes of our category $\bbG_{0,n}$, and the cones of this fan correspond with the respective cones of metrics. Furthermore, this fan is rational with respect to the lattice $N_{\dist_n}$ generated by the vectors $v_I$, where $I\subset \{1,\dots,n\}$ with $1<\#I<n-1$, given by the image of the tree with a unique bounded edge of length $1$ having the marked legs with labels in $I$ on one side of the edge and the marked legs with labels in $\{1,\dots,n\}\backslash I$ on the other. So, in other words, the distance map \eqref{eq: distance map} defines a morphism of poic-complexes $\dist_n\colon \Mtrop_{0,n}\to \underline{\left(N_{\dist_n}\right)_\bbR}$. This makes $\Mtrop_{0,n}$ into a linear poic-complex. We call this linear poic-complex \emph{the linear poic-complex of $n$-marked trees}. We will always consider $\Mtrop_{0,n}$ as a linear poic-complex in this way, so we will simply denote it by $\Mtrop_{0,n}$. We further remark that this linear poic-complex is irreducible.
\subsection{Spanning trees and forgetting the marking} We describe the spanning tree fibrations (Definition 4.5.7 of \cite{DARB}), which consists of linear poic-fibrations over the moduli spaces of tropical curves. Suppose that $A$ is a finite set and $g\geq0$ is an integer with $2g+\#A-2>0$. From Proposition \ref{prop: Ggn iso classes}, it follows that $\Mtrop_{g,A}$ is a poic-space. Let $\mathbbm{g}$ denote the set $\{1,\dots,g,1^*,\dots,g^*\}$ where $i^*$ is just a symbol for $1\leq i\leq g$. We seek to produce a genus-$g$ $A$-marked discrete graph out of a $(A\sqcup \mathbbm{g})$-marked tree by joining the $i$- and $i^*$-legs into a new edge.

\begin{const}\label{const: glue cut cycles}
   Let $T$ be an object of $\bbG_{0,A\sqcup \mathbbm{g}}$. Observe that the following map
    \begin{equation*} \iota_{\st_{g,A}(T)}\colon  F(T)\to F(T), h\mapsto\begin{cases}
        \iota_T(h),& h\not\in L(T),\\
        \ell_i,& h=\ell_i \,(i\in A),\\
        \ell_{i^*},& h=\ell_{i} \,(1\leq i\leq g),\\
        \ell_{i},& h=\ell_{i^*} \,(1\leq i\leq g),
    \end{cases}\end{equation*}
    is a well defined an involution of $F(T)$ which together with the root map $r_T$ and the marking on the $A$-legs gives rise to a connected $A$-marked graph $\st_{g,A}(T)$ of genus-$g$ (the $\st$ stands for \emph{spanning tree}). In Proposition 4.5.2 of \cite{DARB} we show that this construction gives rise to an essentially surjective functor
    \begin{equation}
    {\st_{g,n}}\colon  \bbG^{\op}_{0,2g+n}\to \bbG_{g,n}^{\op}.\label{eq: spanning tree functor genus g n marked}
    \end{equation}
    Some notational remarks are in order. Out of $T$ we obtained an object $\st_{g,A}(T)$ of $\bbG_{g,A}^{\op}$. For $1\leq i\leq g$ let $e_i$ denote the edge of $\st_{g,A}(T)$ given by the orbit $\{\ell_{i}(T),\ell_{i^*}(T)\}$. Observe that the graph $\st_{g,A}(T)$ has the following set of edges
    \begin{equation*}
        E(\st_{g,A}(T)) = E(T) \cup \{e_1,\dots,e_g\}.
    \end{equation*} 
    Therefore the poic $\sigma_T$ can be identified with a face of $\sigma_{\st_{g,A}(T)}$. Furthermore, by identifying the $i$th-coordinate with the coordinate corresponding to the edge $e_i$, the poic $\bbR_{>0}^g$ corresponds to the relative interior of the face of $\sigma_{\st_{g,A}(T)}$ given by the $\{e_i\}_{i=1}^g$. Therefore the product $\sigma_T\times \bbR_{>0}^g$ has a natural inclusion morphism into $\sigma_{\st_{g,A} (T)}$. We let $\ST_T:=\sigma_T\times \bbR_{>0}^g$, and denote the inclusion morphism by $\eta_{\st_{g,A},T }\colon \ST_T\to \sigma_{\st_{g,A}(T)}$. In addition, the $\dist_{A\sqcup\mathbbm{g}}$ morphism and the natural projection give a morphism of poic-complexes 
    \begin{equation}
        D_{g,A}\colon \ST_{g,A}\to \underline{\left(N_{\dist_{A\sqcup\mathbbm{g}}}\right)_\bbR}.\label{eq: distance morphism general}
    \end{equation}
\end{const}
\begin{defi}
    The morphism of poic-complexes \eqref{eq: distance morphism general} makes $\ST_{g,A}$ into a linear poic-complex. We call this the \emph{linear poic-complex of spanning trees of genus-$g$ $A$-marked graphs}. In the spirit of notational simplicity, and, since we will always consider $\ST_{g,A}$ as a linear poic-complex in this way, we will just denote it by $\ST_{g,A}$. 
\end{defi}
It is shown in Proposition 4.5.8 that $\ST_{g,A}$ is pure of rank $3g+\#A-3$ and the poic-fibration $\st_{g,A}$ is irreducible. We now describe the forgetting the marking morphism which has readily been described in Section 4.6 of \cite{DARB}.
\begin{const}
    Let $G$ be an object of $\bbG_{g,A}$. We construct an object $\ft_{a}G$ of $\bbG_{g,A\backslash a}$ given by forgetting the $a$-leg of $G$. The construction follows a detailed analysis around the vertex incident to this leg. Let $V_a\in V(G)$ denote the vertex incident to $\ell_a(G)$.
    \begin{enumerate}
        \item If $\val (V_a) >3$, then set $F\left(\ft_{a}G\right) := F(G)\backslash\ell_a(G)$. The root, involution, and marking maps of $G$ restrict to respective root, involution and marking maps of $F(\ft_{a}G)$, and hence define an $A\backslash a$-marked graph $\ft_{a}G$.
        \item If $\val (V_a) =3$ with $V_a$ incident to two edges, then $r_G^{-1}(V_a) = \ell_a\cup\{f_1,f_2,V_a\}$. In this case, we set $F(\ft_{a}G) = F(G)\backslash r_G^{-1}(V_a)$ and $r_{\ft_{a}G} = r_G$. To define $\iota_{\ft_{a}G}$ it is only necessary to specify where to map $\iota_G(f_1)$ and $\iota_G(f_2)$, for which we simply put
        \begin{align*}
            &\iota_{\ft_{a}G}\left(\iota_G(f_1)\right) = \iota_G(f_2),
            &\iota_{\ft_{a}G}\left(\iota_G(f_2)\right) = \iota_G(f_1).
        \end{align*}
        In this case, the triple $(F(\ft_{a}G),r_{\ft_{a}G},\iota_{\ft_{a}G})$ defines an $A\backslash a$-marked graph $\ft_{a}G$.
        \item If $\val V_a = 3$ with $V_a$ incident to an additional leg, then $r_G^{-1}(V_a) = \ell_a\cup \{l,f,V_a\}$ with $\{l\}\in L(G)$. In this situation, set $F(\ft_aG) =F(G)\backslash \left(\ell_a\cup \{f,\iota_G(f), V_a\}\right)$ and $\iota_{\ft_aG} = \iota_G$. To define $r_{\ft_aG}$ it is sufficient to specify where to map $l$, for which we set $r_{\ft_aG}(l) = r_G(\iota_G(f))$. Therefore, the triple $(F(\ft_a(G)),r_{\ft_a(G)},\iota_{\ft_aG})$ defines an $A\backslash a$-marked graph $\ft_aG$.
    \end{enumerate}
    We refer directly to Construction 4.6.1 of \cite{DARB} for the case of morphisms and move forward to the introduction of poic-morphism $\eta_{\ft_a,G}\colon\sigma_G\to \sigma_{\ft_aG}$ by studying the same three cases. The only possibilities for $E(\ft_{a}G)$ are as follows:
    \begin{enumerate}
        \item If $\val(V_a)>3$, then $E(G)=E(\ft_{a}G)$ and we set $\eta_{\ft_{a},G}= \Id_{\bbR_{\geq0}^{E(G)}}$.
        \item If $\val(V_a) =3$ with $V_a$ incident to two edges, then $r_G^{-1}(V_a) = \ell_a\cup\{f_1,f_2,V_a\}$, and $e=\{\iota_G(f_1),\iota_G(f_2)\}$ defines an edge of $E(\ft_{a}G)$. For $i=1,2$, let $e_i\in E(G)$ denote the edge given by $\{f_i,\iota_G(f_i)\}$. In this case
            \begin{equation*}
                E(\ft_{a}G) = E(G)\backslash\{e_1,e_2\} \cup \{e\},
            \end{equation*}
        and we let $\eta_{\ft_{a},G}\colon\sigma_G\to\sigma_{\ft_{a}G}$ be the map given by 
            \begin{equation*}
                \eta_{\ft_{a},G}(\delta):=\begin{cases}
                \delta(h) = \delta(h), & \textnormal{ if }h\in E(\ft_{a}G)\backslash\{e\},\\
                \delta(h)=\delta(e_1)+\delta(e_2), &\textnormal{ if } h=e.\\
                \end{cases}
            \end{equation*}
        \item If $\val V_a = 3$ with $V_a$ incident to an additional leg, then $r_G^{-1}(V_a) = \ell_a\cup \{l,f,V_a\}$ with $\{l\}\in L(G)$. The set $\{f,\iota_G(f)\}$ defines an edge $e^\prime\in E(G)$ and $E(\ft_aG) = E(G)\backslash e^\prime$. We set $\eta_{\ft_a,G}\colon\sigma_G\to\sigma_{\ft_aG}$ as the projection to the face defined by $E(\ft_aG)\subset E(G)$. 
    \end{enumerate}
\end{const}
 It is shown in Lemma 4.6.2 of \cite{DARB} that the above defines a morphism of poic-spaces
    \begin{equation}
        \ft_{a}:\bbG_{g,A}\to\bbG_{g,A\backslash a}, G\mapsto \ft_{a}G. \label{eq: forgetting the a-marking functor}
    \end{equation}
    Moreover, if $b\in A\backslash a$ is an additional element, then ${\ft_{a}}\circ {\ft_{b}}={\ft_{b}}\circ {\ft_{a}}$. For a general subset $I\subset A$ we set (as forgetting everything but $I$) $\ft_I := \ft_{i_1}\circ \ft_{i_2}\circ\dots\circ \ft_{i_k}$, where $I^c=\{i_1,\dots,i_k\}$.
    
    In the case of trivial genus, the natural projection maps $\bbR^{A}\to\bbR^{A\backslash a}$ and $\bbR^{\binom{A}{2}}\to\bbR^{\binom{A\backslash a}{2}}$, given respectively by forgetting the $a$-coordinate and forgetting any coordinate containing $a$, induce a linear map between the lattices $p_a\colon N_{\dist_{A}}\to N_{\dist_{A\backslash a}}$. If we set $(\ft_{a})_\tint := p_a$, then Construction 4.6.4 of loc. cit. shows that we obtain a morphism of linear poic-complexes 
    \begin{equation*}
        \ft_{a}\colon \Mtrop_{0,A}\to\Mtrop_{0,A\backslash a},
    \end{equation*}
    and Lemma 4.6.2 of loc. cit. shows that it is actually a proper morphism of linear poic-complexes. We take this space to state the following useful lemma, which is Lemma 2.3 of \cite{GathmannMarkwigOchseTMSSMC}.
    \begin{lem}\label{lem: zero vectors in QN}
    A vector $\vec{v}\in (N_{\dist_A})_\bbR$ is zero if and only if $p_{I^c}(\vec{v}) = 0$ for every $I\subset A$ with $\# I = 4$.
    \end{lem}
To finalize this subsection, we discuss the extension of the forgetting the marking morphisms to the case of the linear poic-fibrations of spanning trees. 
\begin{const}\label{const: forgetting marking spanning tree}
   Let $T$ be an object of $\bbG_{0,\mathbbm{g}\sqcup A}$ and consider the morphism of poics 
   \begin{equation}
   \eta_{\mathfrak{ft}_a^\pcmplxs,T}\colon \ST_{g,A}(T)\to\ST_{g,A\backslash a}(\ft_aT) \label{eq: morphism of poics}    
   \end{equation}
   given by the linear map
    \begin{equation*}
        \begin{pmatrix}
        \eta_{\ft_a,T}&0\\
        M(T,a)&\Id_{g\times g}
        \end{pmatrix}\colon \sigma_T\times\bbR_{>0}^g\to \sigma_{\ft_aT}\times\bbR_{>0}^g,
    \end{equation*}
    where $M(T,a)$ is the $g\times\#E(T)$-matrix such that for $1\leq i\leq g$ and $e\in E(T)$ its $(i,e)$-entry is
    \begin{equation*}
        M(T,a)_{i,e} = \begin{cases}1,& \textnormal{ if }\val r_T\left(\ell_a(T)\right)=3 \textnormal{ and }\partial\ell_a(T)=\partial\ell_{i}(\ft_aT)\subset \partial e,\\
        1,& \textnormal{ if }\val r_T\left(\ell_a(T)\right)=3 \textnormal{ and }\partial\ell_a(T)=\partial\ell_{i^*}(\ft_aT)\subset \partial e,\\
        0,& \textnormal{ else}.
        \end{cases}
    \end{equation*}
    It is shown in Proposition 4.6.5 of \cite{DARB} that we obtain a weakly proper morphism of linear poic-fibrations 
    \begin{equation}
        \mathfrak{ft}_{a}\colon \st_{g,A}\to\st_{g,A\backslash a},\label{eq: forgetting the a-mark fibrations.}
    \end{equation}
    by letting $\mathfrak{ft}_a^\pcmplxs$ denote the morphism of linear poic-complexes given by the morphisms \eqref{eq: morphism of poics} and $(\mathfrak{ft}^\pcmplxs_a)_{\tint}:=(\ft_a)_{\tint} =p_a$, and letting $\mathfrak{ft}_a^\pspcs$ denote the morphism of poic-spaces $\ft_a\colon \Mtrop_{g,A}\to \Mtrop_{g,A\backslash a}$.
\end{const}

\section{Discrete admissible covers.}
Let $h,m\geq 0$ be integers with $2h+m-2>0$, let $d>0$ be an integer, and let $\vec{\mu}=(\mu_1,\dots,\mu_m)$ be a vector of partitions of $d$. In this section we introduce the category $\bbAC_{d,h}(\vec{\mu})$ of degree-$d$ discrete admissible covers of genus-$h$ $m$-marked curves with ramification profiles above the marked legs prescribed by $\vec{\mu}$. We recall some of the well-known combinatorial results underlying discrete admissible covers and subsequently introduce the cones of metrics associated to a cover. This association gives rise to the poic-space $\AC_{d,h}(\vec{\mu})$, which lies in the heart of our enterprise. This poic-space comes with poic-morphisms $\src\colon \AC_{d,h}(\vec{\mu})\to \Mtrop_{g,n}$ and $\trgt\colon\AC_{d,h}(\vec{\mu})\to \Mtrop_{h,m}$, where $n=\sum_i \ell(\mu_i)$ is the marking of the source curve and $g$ its genus (this is given by the Riemann-Hurwitz formula). Subsequently, we revisit the linear poic-fibrations $\st_{g,n}$ and $\st_{h,m}$, and use them together with these maps to introduce a collection of linear poic-fibrations $\est^J_{d,h,\vec{\mu}}\colon \EST^J_{d,h,\vec{\mu}}\to \AC_{d,h}(\vec{\mu})$ where $J\subset \{1,\dots,n\}$. The section is closed with a revisit of the forgetful morphisms on the spanning tree fibrations.

\subsection{Discrete admissible covers.} In this section we let $G_1$ and $G_2$ denote two discrete graphs. We will call a map of graphs $f\colon G_1\to G_2$ \emph{non-contracting} if it does not involve the contraction of edges or legs of $G_1$ to vertices. In other words, $f^{-1}(V(G_2))\subset V(G_1)$.
\begin{defi}
    A \emph{degree assignment} on $G_1$ is a map $d\colon F(G_1)\to \bbZ_{\geq0}$ such that $d\circ \iota = d$.
\end{defi}

\begin{nota}
    Let $d$ be a degree assignment on $G_1$. If $e\in E(G_1)$, then there is an $E\in F(G_1)$ such that $e=\{ E,\iota(E)\}$ and by definition $d(E)=d(\iota(E))$. We will abuse notation and refer to this common value simply as $d(e)$. Analogously, if $\ell\in L(G_1)$ and $\ell=\{L\}$, then $d(\ell)$ will denote $d(L)$.
\end{nota}

\begin{remark}
    It is clear that a degree assignment on $G_1$ is simply an integral valued function on $V(G_1)\cup E(G_1)\cup L(G_1)$. 
\end{remark}

\begin{defi}
     A \emph{harmonic morphism} $\pi\colon G_1\to G_2$ consists of a surjective non-contracting map of graphs $\pi\colon G_1\to G_2$ and a degree assignment $d_\pi$ on $G_1$ subject to the following condition: For any $V\in V(G_1)$ and $H^\prime\in r_{G_2}^{-1}\left(\pi(V)\right)\backslash\{\pi(V)\}$ 
     \begin{equation*}
     d_\pi(V) = \sum_{H\in \pi^{-1}(H^\prime)\cap r_{G_1}^{-1}(V)}d_\pi(H).
     \end{equation*}
    If $\pi\colon G_1\to G_2$ is a harmonic morphism, then the \emph{local degree at } $V\in V(G_1)$ is the integer $d_\pi(V)$.
\end{defi}
\begin{remark}
Suppose $\pi\colon G_1\to G_2$ is a harmonic morphism. As has been previously remarked, a map of graphs cannot map a leg to an edge, or an edge to a leg. Additionally, the condition on $\pi$ being non-contracting implies that it must map edges to edges, legs to legs, and vertices to vertices.
\end{remark}
\begin{lem}
Let $\pi\colon G_1\to G_2$ be a harmonic morphism. If $G_2$ is connected, then the following equality holds
    \begin{equation*}
    \sum_{V\in \pi^{-1}(A)}d_\pi(V) = \sum_{W\in \pi^{-1}(B)}d_\pi(W),
    \end{equation*}
    for any $A,B\in V(G_2)$.
\end{lem}
\begin{proof}
     This is a well-known result in the theory of harmonic morphisms on graphs.
\end{proof}
\begin{defi}
    The \emph{degree} of a harmonic morphism $\pi\colon G_1\to G_2$ with $G_2$ connected is defined as the integer
    \begin{equation*}
        \deg \pi := \sum_{V\in \pi^{-1}(W)}d_\pi(V),
    \end{equation*}
    where $W\in V(G_2)$ is any vertex. The previous lemma implies the invariance of this number with respect to the vertices of $G_2$.
\end{defi}
\begin{defi}
    Given a harmonic morphism $\pi\colon G_1\to G_2$, the \emph{Riemann-Hurwitz number (RH number) of $\pi$ at $V\in V(G_1)$} is defined as
    \begin{equation*}
    r_\pi(V):= \val (V)-2 - d_\pi(V)(\val(\pi(V))-2).
    \end{equation*}
    If $G_2$ is connected, then the degrees of the edges in the fiber of $e$ form a partition of $\deg \pi$. This unordered partition is called the \emph{ramification profile of $\pi$ at $e$} and we denote it by $\ram_\pi(e)$.
\end{defi}

\begin{prop}[Riemann-Hurwitz equality] Suppose $G_2$ is connected and $\pi\colon G_1\to G_2$ is a harmonic morphism. The genera $g(G_1)$ and $g(G_2)$ are related to the degree of $\pi$ by:
\begin{equation}
\#L(G_1) + 2(g(G_1)-1) = \deg\pi \cdot ( \#L(G_2) + 2(g(G_2)-1)) + \sum_{V\in V(G_1)} r_\pi(V).\label{eq: Riemann-Hurwitz}
\end{equation}
\end{prop}
\begin{proof}
    Again, this is a well-known result. However, to convey familiarity with our notation and for the sake of completeness we provide a proof. For $V\in V(G_1)$ the RH number is defined as
    \begin{equation*}
    r_\pi(V) = \val V-2 - d_\pi(V) ( \val \pi(V)-2).
    \end{equation*}
    Therefore 
    \begin{equation*}
        \sum_{V\in V(G_1)}r_\pi(V) = \sum_{V\in V(G_1)}(\val V -2) - \sum_{V\in V(G_1)}d_\pi(V)(\val \pi( V) -2).
    \end{equation*}
    Furthermore, we observe that $\sum_{V\in V(G_1)}\val V = \#L(G_1) + 2\#E(G_1)$, and the previous equation yields:
    \begin{equation*}
        \sum_{V\in V(G_1)}r_\pi(V)=\#L(G_1)+2(g(G_1)-1)-\sum_{V\in V(G_1)}d_\pi(V)(\val \pi (V)-2).
    \end{equation*}
  The set $V(G_1)$ is partitioned as $V(G_1) = \bigsqcup_{A\in V(G_2)} \pi^{-1}(A)$, and we can therefore reorganize the sum:
   \begin{align*}
       \sum_{V\in V(G_1)}d_\pi(V)(\val \pi( V) -2) &= \sum_{A\in V(G_2)} \sum_{V\in \pi^{-1}(A)}d_\pi(V) (\val A-2)= \deg\pi\sum_{A\in V(G_2)}( \val A-2)\\
       &= \deg\pi( \#L(G_2) +2(g(G_2)-1)). 
   \end{align*}
   From this, the claimed equality holds.
\end{proof}
We now study the interaction between harmonic morphisms and edge contractions of the target. 
\begin{const}\label{const: edge contraction hm}
Suppose $\pi\colon G_1\to G_2$ is a harmonic morphism, with both $G_1$ and $G_2$ connected. Let $e\in E(G_2)$ be such that $\pi^{-1}(e)$ is a forest with corresponding connected components $T_1,\dots,T_m$. Since $e$ is an edge, it follows that $L(T_i)=\varnothing$ for all $1\leq i\leq m$. Following the notation introduced in Definition \ref{defi: contraction of legless subgraph}, the graph $G_1/\pi^{-1}(e)$ has the set of vertices
\begin{equation*}
    V(G/\pi^{-1}(e)) =\left( V(G)\backslash \bigcup_{1\leq i\leq m}V(T_i)\right) \cup \{V_{T_1},\dots, V_{T_m}\}.
\end{equation*}
In addition, the observation following this same definition shows that $g(G_1/\pi^{-1}(e)) = g(G_1)$. Let $p\colon G_2\to G_2/e$ denote the corresponding edge contraction map. The map $\pi$ induces a map of graphs $\pi_e\colon G_1/\pi^{-1}(e)\to G_2/e$ given by
\begin{equation*}
    \pi_e\colon F(G_1/\pi^{-1}(e))\to F(G_2/e), H\mapsto \begin{cases}
    p(f(H)),&\textnormal{ if }H\neq V_{T_i},\\
    V_e,&\textnormal{ if }H=V_{T_i}.
\end{cases}
\end{equation*}
The degree assignment $d_\pi$ induces the degree assignment
\begin{equation*}
d_e\colon F(G_1/\pi^{-1}(e))\to \bbZ_{>0}, H\mapsto \begin{cases}
    d_\pi(H),& \textnormal{ if }H\neq V_{T_i},\\
    \sum_{h\in E(T_i)}d_\pi(h),& \textnormal{ if }H=V_{T_i}.
\end{cases}
\end{equation*}
\end{const}
\begin{prop}\label{prop: RH numbers add in contractions}
    Following the notation of Construction \ref{const: edge contraction hm}, the map $\pi_e$ and the degree assignment $d_e$ define a harmonic morphism $\pi_e\colon G_1/\pi^{-1}(e)\to G_2/e$, such that for any $1\leq i\leq m$
    \begin{equation*}
    r_{\pi_e}(V_{T_i}) = \sum_{V\in V(T_i)} r_\pi(V).
    \end{equation*}
\end{prop}
\begin{proof}
    Since $\pi$ is already a harmonic morphism, it is only necessary to show the condition on harmonic morphisms for $\pi_e$ at the vertices $V_{T_1},\dots,V_{T_m}$. Let $e\in E(G_2)$ be such that $\partial e = \{A,B\}$, and observe that for any $1\leq i\leq m$ the incident vertices to any edge of $T_i$ must lie over $A$ or $B$. The facts that $T_i$ is a tree and $\pi$ is harmonic, imply that
    \begin{equation*}
    \sum_{h\in E(T_i)}d_\pi(h) = \sum_{V\in V(T_i)\cap \pi^{-1}(A)}d_\pi(V) = \sum_{ V\in V(T_i)\cap \pi^{-1}(B)}d_\pi(V) .
    \end{equation*}
    Let $H\in H(G_2/e)$ be such that $r_{G_2/e}(H)=V_e$. It holds that $H\in F(G_2)$ and let us assume without loss of generality that $r_{G_2}(H) =A$. The set of flags of $G_1/\pi^{-1}(e)$ incident to $V_{T_i}$ and mapping to $H$ correspond to the flags of $G_1$ incident to the vertices of $T_i$ lying over $A$ and mapping to $H$. The latter implies that
    \begin{align*}
        \sum_{H^\prime\in {\pi_e}^{-1}(H)\cap r_{G_1/\pi^{-1}(e)}^{-1}(V_{T_i})}d_e(H^\prime) &= \sum_{V\in V(T_i)\cap \pi^{-1}(A)}\left(\sum_{H^\prime\in \pi^{-1}(H)\cap r_{G_1}^{-1}(V)}d_\pi(H^\prime)\right)\\
        &=\sum_{V\in V(T_i)\cap \pi^{-1}(A)}d_\pi(V)=d_e(V_{T_i}),
    \end{align*}   
    and shows that $\pi_e$ is a harmonic morphism. It remains to show the equality concerning RH numbers. As before, let $1\leq i\leq m$ be fixed, and let us keep the same notation. Notice that $\val V_e = \val A + \val B -2 $ and similarly
    \begin{equation}
    \val V_{T_i} = \sum_{V\in V(T_i)}\val (V)-2\#E(T_i) = 2+\sum_{V\in V(T_i)}(\val V-2).\label{eq: 3}
    \end{equation}
    We can rewrite the RH number of $V_{T_i}$ as
    \begin{equation*}
        r_{\pi_e}(V_{T_i})=\val V_{T_i}-2-\sum_{V\in V(T_i)\cap \pi^{-1}(A)}d_e(V) (\val A-2) - \sum_{V\in V(T_i)\cap \pi^{-1}(A)}d_e(V)(\val B-2).
    \end{equation*}
    Finally, we use the fact that $\pi^{-1}(A)\cap V(T_i)$ and $\pi^{-1}(B)\cap V(T_i)$ partition $V(T_i)$ (and express $\val V_{T_i}-2$ by means of \eqref{eq: 3}) to obtain
    \begin{equation*}
        r_{\pi_e}(V_{T_i}) =\sum_{V\in V(T_i)\cap \pi^{-1}(A)}r_\pi(V)+\sum_{V\in V(T_i)\cap \pi^{-1}(B)}r_\pi(V)=\sum_{V\in V(T_i)}r_\pi(V).
    \end{equation*}
    
\end{proof}
\begin{defi}
    A harmonic morphism $\pi\colon G_1\to G_2$ is a \emph{discrete admissible cover} if for every vertex $V\in V(G_1)$ the RH number vanishes: $r_\pi(V)=0$. From the previous proposition, it follows that for a discrete admissible cover the contraction of edges (with acyclic fibers) of the target gives rise to further discrete admissible covers.
\end{defi}

\begin{remark}
    We remark that the Riemann-Hurwitz formula shows that if $\pi\colon G_1\to G_2$ is a discrete admissible cover, then
    \begin{equation*}
        \#L(G_1) + 2(g(G_1)-1) = \deg\pi ( \#L(G_2) + 2(g(G_2)-1))
    \end{equation*}
\end{remark}

\subsection{Categories of discrete admissible covers.} 

In the same spirit as for $\bbG_{g,n}$, we construct a category of discrete admissible covers (and their contractions) by fixing the marking and genus of the target, as well as the degree of the covers and ramification profiles over the marked ends. For this we use the following notation:
\begin{nota}\label{nota: standing hdmmu}
    Let $h,d,m\geq 0$ be integers with $2h+m-2>0$, a vector $\vec{\mu}=(\mu_1,\dots,\mu_m)$ of partitions of $d$, set $n= \sum_{i=1}^m\ell(\mu_i)$, and let $g$ be defined by
\begin{equation*}
    n+2(g-1) = d\cdot(m+2(h-1)).
\end{equation*}
We assume that $g$ is an integer, which imposes further conditions on $d$, $m$ and $\vec{\mu}$. It is apparent (but worth mentioning) that this equation for $g$ comes from the Riemann-Hurwitz equality \eqref{eq: Riemann-Hurwitz}.
\end{nota}
\begin{defi}\label{defi: category achm mu}
Let $d$, $h$, $m$, $\vec{\mu}$, $n$ and $g$ be as above. The \emph{category of discrete admissible covers of genus-$h$ $m$-marked curves with ramification $\vec{\mu}$} is the category $\bbAC_{d,h}(\vec{\mu})$ given by the following data:
\begin{itemize}
    \item The objects of $\bbAC_{d,h}(\vec{\mu})$ consist of degree-$d$ discrete admissible covers $\pi\colon G\to H$ where:
    \begin{itemize}
        \item $H$ is an object of $\bbG_{h,m}$.
        \item $G$ is an object of $\bbG_{g,n}$ (it is immediate that $2g+n-2>0$). 
        \item For every $1\leq k\leq m$:
            \begin{equation*}
                \pi\left(\bigcup_{j=1}^{\ell(\mu_k)}\ell_{N+j}(G) \right) =\ell_k(H), 
            \end{equation*}
        where $N=\sum_{i<k}\ell(\mu_i)$.
        \item For every $1\leq k\leq m$ and every $1\leq j\leq \ell(\mu_k)$, the weight of the $\sum_{i<k}\ell(\mu_i)+j$-th leg of $G$ is $\mu_k^j$.
    \end{itemize}
    \item For two objects $\pi_1\colon G_1\to H_1$ and $\pi_2\colon G_2\to H_2$ of $\bbAC_{d,h}(\vec{\mu})$, the set of morphisms $\Hom_{\bbAC_{d,h}(\vec{\mu})}(\pi_1,\pi_2)$ consists of the tuples $f=(f_\src,f_\trgt)$ where:
    \begin{itemize}
        \item $f_\src\in \Hom_{\bbG_{g,n}}(G_1,G_2)$,
        \item $f_\trgt\in\Hom_{\bbG_{h,m}}(H_1,H_2)$,
        \item $\pi_2\circ f_\src= f_\trgt\circ \pi_1$. 
    \end{itemize}
    \item Composition of morphisms is just composition of tuples of maps.
\end{itemize}
If $\pi\colon G\to H$ is an object of $\bbAC_{d,h}(\vec{\mu})$, then we set $\src(\pi):=G$ (as in source) and $\trgt(\pi):=H$ (as in target).
\end{defi}
\begin{const}
    The objects of the category $\bbAC_{d,h}(\vec{\mu})$ come with two natural projections:
    \begin{itemize}
        \item Onto $\bbG_{h,m}$ by forgetting the source and the cover, but keeping the target.
        \item Onto $\bbG_{g,n}$ by forgetting the target and the cover, but keeping the source.
    \end{itemize}
    More precisely, there are the following natural associations
    \begin{align*}
        \src&\colon\Obj(\bbAC_{d,h}(\vec{\mu}))\to \Obj(\bbG_{g,n}), \pi\mapsto \src(\pi),\\
        \trgt&\colon\Obj(\bbAC_{d,h}(\vec{\mu}))\to \Obj(\bbG_{h,m}), \pi\mapsto \trgt(\pi).
    \end{align*}
    These are further improved to morphisms of $\bbAC_{d,h}(\vec{\mu})$ by setting for $f=(f_\src,f_\trgt)$:
    \begin{align*}
        &\src(f) := f_\src\in\Hom_{\bbG_{g,n}}(G_1,G_2), &\trgt(f) :=f_\trgt\in \Hom_{\bbG_{h,m}}(H_1,H_2),
    \end{align*}
    for $(f_\src,f_\trgt)\in \Hom_{\bbAC_{d,h}(\vec{\mu})}(\pi_1,\pi_2)$, and $\src(\pi_i)=G_i$ and $\trgt(\pi_i) = H_i$ (for $i=1,2$). It is clear that the above constructions give rise to functors:
    \begin{align*}
        &\src\colon\bbAC_{d,h}(\vec{\mu})\to \bbG_{g,n},
        &\trgt\colon\bbAC_{d,h}(\vec{\mu})\to \bbG_{h,m}.
    \end{align*} 
    In particular, the product of these two functors gives rise to the functor
\begin{equation}
    \src\times\trgt\colon \bbAC_{d,h}(\vec{\mu})\to \bbG_{g,n}\times \bbG_{h,m},\pi\mapsto (\src(\pi),\trgt(\pi)). \label{eq: prod of src and trgt achm mu}
\end{equation} 
Furthermore, for an object $\pi\colon G\to H$ of $\bbAC_{d,h}(\vec{\mu})$ there are the following poic-morphisms given by the corresponding projections
\begin{align}
    &\eta_{\src, \pi}\colon \sigma_\pi\to \sigma_G,
    &\eta_{\trgt, \pi}\colon \sigma_\pi\to \sigma_H.\label{eq: src trgt poic-morphism}
\end{align}
\end{const}

\subsection{Cones of metrics of covers and their moduli spaces.}\label{ssec: mod spaces of tropical admissible covers} We now explain how to relate metrics through a (harmonic morphism) discrete admissible cover. The underlying idea is to ensure that the realization of such a map induces a piecewise integral linear map between the realizations of the source and target graphs. 

\begin{defi}
    Suppose $\pi\colon G\to H$ is a discrete admissible cover and $\delta^\prime$ is a metric on $H$. The \emph{induced metric} on $G$ is the metric $\delta\colon E(G)\to\bbR_{>0}$ given by $\delta(e) = \frac{\delta^\prime\pi(e)}{d_\pi(e)}$.
\end{defi}
    
\begin{remark}
    Under these metrics, the map induced by the cover between the corresponding metric spaces is integral linear when restricted to each interval.
\end{remark}
\begin{defi}
    Suppose $\pi\colon G\to H$ is an admissible cover. The \emph{matrix associated to $\pi$} is the linear map $F_\pi:\bbR^{E(H)} \to \bbR^{E(G)},v\mapsto F_\pi(v),$ where for $h\in E(G)$ the $h$-coordinate of $F_\pi(v)$ is defined by
    \begin{equation}
    \left(F_\pi({v})\right)_h = \frac{\mathrm{lcm}\left((d_\pi(e))_{e\in \pi^{-1}(\pi(h))}\right)}{d_\pi(h)}({v})_{\pi(h)},\label{eq: assoc matrix}
    \end{equation}
     where $\mathrm{lcm}(\bullet)$ denotes the least common multiple.
\end{defi}
\begin{example}
    Consider the object of $\bbAC_{2,0}((2^4))$ depicted in Figure \ref{fig: admissible cover from genus 1}. The weight of the middle edges of the source is $1$ and the weight of every leg of the source is $2$. 
    \begin{figure}[h!]
        \centering
        \begin{tikzpicture}
            \draw[fill=none,line width=2pt] (0,1) ellipse [x radius = 2, y radius = 0.75];
            \draw[red,line width=2pt,dashed] (-2,1) --(-4,1.5);
            \node[red] at (-4.2,1.5) {$2$};
            \draw[blue,line width=2pt,dashed] (-2,1) --(-3.5,0.5);
            \node[blue] at (-3.7,0.5) {$2$};
            \draw[purple,line width=2pt,dashed] (2,1) --(3.5,1.5);
            \node[purple] at (3.7,1.5) {$2$};
            \draw[olive,line width=2pt,dashed] (2,1) --(4,0.5);
            \node[olive] at (4.2,0.5) {$2$};
            \node at (0,2) {$x$};
            \node at (0,0) {$y$};
            \draw[fill=black] (-2,1) circle (3pt);
            \draw[fill=black] (2,1) circle (3pt);

            \draw[line width=2pt] (-2,-1)--(2,-1);
            \draw[red,line width=2pt,dashed] (-2,-1) --(-4,-0.5);
            \draw[blue,line width=2pt,dashed] (-2,-1) --(-3.5,-1.5);
            \draw[purple,line width=2pt,dashed] (2,-1) --(3.5,-0.5);
            \draw[olive,line width=2pt,dashed] (2,-1) --(4,-1.5);
            \node at (0,-1.3) {$t$};
            \draw[fill=black] (-2,-1) circle (3pt);
            \draw[fill=black] (2,-1) circle (3pt);
            
            \draw (-4.5,2.3)--(6.5,2.3);
            \draw (-4.5,-0.3)--(6.5,-0.3);
            \draw (-4.5,-2)--(6.5,-2);
            \draw (4.5,2.3)--(4.5,-2);
            \draw (6.5,2.3)--(6.5,-2);
            \draw (-4.5,2.3)--(-4.5,-2);
            \node at (5.5,1) {Source};
            \node at (5.5,-1.1) {Target};
        \end{tikzpicture}
        \caption{An object of $\bbAC_{2,0}((2^4))$.}
        \label{fig: admissible cover from genus 1}
    \end{figure}
    We remark that $\bbAC_{2,0}((2^4))$ is a groupoid (every morphism is an isomorphism). To obtain representative objects of the other isomorphism classes, it suffices to permute the markings (colors in the picture) of the legs accordingly. The matrix is then simply given by the map: $\bbR \to \bbR^2, t\mapsto \begin{pmatrix}
        t\\
        t
        \end{pmatrix}$.

\end{example}
\begin{defi}
    Let $\pi\colon G\to H$ be an admissible cover. The \emph{cone of metrics of $\pi$} is the cone $\sigma_{\pi}$ given as the closure of the graph of the restriction of $F_\pi$ to $\sigma_H^o$ in $\sigma_H\times \sigma_G$.
    In other words, for an admissible cover $\pi\colon G\to H$ the partially open polyhedral cone of $\left(N^{\sigma_H}\oplus N^{\sigma_G}\right)_\bbR$ underlying the poic $\sigma_\pi$ is given as the set of points
\begin{equation*}
    \sigma_\pi = \{(x,F_\pi(x)): x\in F_\pi^{-1}(\sigma_G)\cap\sigma_H\} \subset \sigma_H\times \sigma_G.
\end{equation*}
The corresponding abelian group is just 
\begin{equation*}
    N^{\sigma_\pi}:= \{(v,F_\pi(v)) :v\in N^{\sigma_H}\}\subset N^{\sigma_H}\oplus F_\pi (N^{\sigma_H}).
\end{equation*}
By definition $\dim(\sigma_\pi) = \dim(\sigma_H)$, and in fact projection onto the first factor yields an isomorphism $\sigma_\pi\cong F_\pi^{-1}(\sigma_G)\cap\sigma_H$. We work out some examples.
\end{defi}

\begin{example}
    In the previous example, the cone of metrics of the target is the half line $\bbR_{\geq0}$, while the cone of metrics of the source can be given as the intersection of the non-negative orthant $\bbR^2_{\geq0}$ and the poic $\xi$ defined by the partially open integral cone of $\bbR^2$ given by $x+y>0$. The matrix maps the relative interior of $\bbR_{\geq0}$ to $\bbR^2_{\geq0}\cap \xi$, and, in particular, the cone of metrics of the cover is the poic given by the partially open integral cone of $\left( \bbZ^2\oplus \langle \begin{pmatrix}
        1\\
        1
    \end{pmatrix}\rangle \right)_\bbR\cong \bbR^3$ defined by $\{(x,y,z) : x=y=z>0\}\subset \bbR^2_{\geq0}\times\bbR_{\geq0}$.
\end{example}

\begin{example}
    Let $h=0$, $m=5$, $\vec{\mu} = ({\color{red}(2)},{\color{blue}(2)},{\color{purple}(2)},{\color{olive}(2)},{\color{green}(1,1)})$, $n=6$ and $g=1$. Below we depict some objects of $\bbAC_{2,0}({\color{red}(2)},{\color{blue}(2)}$,${\color{purple}(2)}$,${\color{olive}(2)}$,${\color{green}(1,1)})$, where we abuse notation slightly and ignore the source markings. In Figure \ref{fig: cones of metrics ac20} we explicitly write their associated matrix and depict their corresponding cones of metrics (as a subcone of that from the target tree).
    \begin{figure}[h!]
        \centering
        \begin{tikzpicture}
            \draw[fill=none, line width=2pt ] (-1,4) ellipse [x radius=1,y radius=0.5];
            \draw[line width= 4pt](0,4)--(2,4) node [midway, above]{$c$} node [midway, below] {$2$};
            \draw[line width = 2pt, dashed,red](-2,4)--(-4,4.5) node [pos = 1.1] {$2$};
            \draw[line width = 2pt, dashed,blue](-2,4)--(-3.5,3.5) node [pos = 1.2] {$2$};
            \draw[line width = 2pt, dashed,purple](0,4)--(1,3) node [pos = 1.2] {$2$};
            \draw[line width=2pt, dashed, olive](2,4)--(4,3.5) node [pos = 1.1] {$2$};
            \draw[line width=2pt, dashed, green](2,4)--(3.5,4.2);
            \draw[line width=2pt, dashed, green](2,4)--(3.5,4.7);
            \draw[fill = black] (-2,4) circle (3pt);
            \draw[fill = black] (0,4) circle (3pt);
            \draw[fill = black] (2,4) circle (3pt);
            \node[] at (-1,4.7) {$a$};
            \node[] at (-1,3.8) {$b$};

            \draw[line width =2pt, dashed, red] (-2,2) -- (-4,2.5);
            \draw[line width =2pt, dashed, blue] (-2,2) -- (-3.5,1.5);
            \draw[line width = 2pt] (-2,2)--(0,2) node [midway, below] {$x$};
            \draw[line width =2pt, dashed, purple] (0,2) -- (1,1);
            \draw[line width = 2pt] (0,2)--(2,2) node [midway, below] {$y$};
            \draw[line width=2pt, dashed, olive](2,2)--(4,1.5);
            \draw[line width=2pt, dashed, green](2,2)--(3.5,2.2);
            \draw[fill = black] (-2,2) circle (3pt);
            \draw[fill = black] (0,2) circle (3pt);
            \draw[fill = black] (2,2) circle (3pt);

            \matrix (1col0) [matrix of math nodes] at (5,3) {
            a\\
            b\\
            c\\
            };
            \matrix (1row0) [matrix of math nodes] at (6,4) {
            x&y\\
            };
            \matrix (1m) [matrix of math nodes,left delimiter={[ },right delimiter={]}] at (6,3) {
            1&0\\
            1&0\\
            0&1\\
            };

            \draw[fill=gray,opacity=0.3,line width=0pt] (8,1.5)--(8,4.5)--(11,4.5)--(11,1.5);
            \draw[dashed,line width=2pt](8,1.5)--(8,4.5) node [midway, left] {$y$};
            \draw[line width=2pt](8,1.5)--(11,1.5) node [midway, below] {$x$};
            \draw[line width=1pt,fill=white](8,1.5) circle (3pt);

            \draw[fill=none, line width= 2pt] (0,-2) ellipse [x radius = 2, y radius = 0.5];
            \draw[line width = 2pt, dashed,red](-2,-2)--(-4,-1.5) node [pos = 1.1] {$2$};
            \draw[line width = 2pt, dashed,blue](-2,-2)--(-3.5,-2.5) node [pos = 1.2] {$2$};
            \draw[line width = 2pt, dashed,purple](2,-2)--(3.5,-1.5) node [pos = 1.2] {$2$};
            \draw[line width=2pt, dashed, olive](2,-2)--(4,-2.5) node [pos = 1.1] {$2$};
            \draw[line width=2pt, dashed, green](0,-1.5)--(1,-2.5);
            \draw[line width=2pt, dashed, green](0,-2.5)--(1,-3.5);
            \draw[fill = black] (-2,-2) circle (3pt);
            \draw[fill = black] (0,-1.5) circle (3pt);
            \draw[fill = black] (0,-2.5) circle (3pt);
            \draw[fill = black] (2,-2) circle (3pt);
            \node[] at (-1,-1.3) {$a$};
            \node[] at (-1,-2.2) {$b$};
            \node[] at (1,-1.3) {$c$};
            \node[] at (1,-2.2) {$d$};

            \draw[line width =2pt, dashed, red] (-2,-4) -- (-4,-3.5);
            \draw[line width =2pt, dashed, blue] (-2,-4) -- (-3.5,-4.5);
            \draw[line width = 2pt] (-2,-4)--(0,-4)node [midway, below] {$x$};
            \draw[line width = 2pt] (0,-4)--(2,-4)node [midway, below] {$y$};
            \draw[line width=2pt, dashed, green](0,-4)--(1,-5.2);
            \draw[line width=2pt, dashed, olive](2,-4)--(4,-4.5);
            \draw[line width =2pt, dashed, purple] (2,-4) -- (3.5,-3.5);
            \draw[fill = black] (-2,-4) circle (3pt);
            \draw[fill = black] (0,-4) circle (3pt);
            \draw[fill = black] (2,-4) circle (3pt);
            
            \matrix (2col0) [matrix of math nodes] at (5,-3.3) {
            a\\
            b\\
            c\\
            d\\
            };
            \matrix (2row0) [matrix of math nodes] at (6,-2) {
            x&y\\
            };
            \matrix (2m) [matrix of math nodes,left delimiter={[ },right delimiter={]}] at (6,-3.3) {
            1&0\\
            1&0\\
            0&1\\
            0&1\\
            };
            \draw[fill=gray,opacity=0.3,line width=0pt] (8,-4.5)--(8,-1.5)--(11,-1.5)--(11,-4.5);
            \draw[line width=2pt](8,-4.5)--(8,-1.5)node [midway, left] {$y$};
            \draw[line width=2pt](8,-4.5)--(11,-4.5) node [midway, below] {$x$};
            \draw[line width=1pt,fill=white](8,-4.5) circle (3pt);

            \draw (-5,5.2)--(12,5.2);
            \draw (-5,6)--(4.5,6) node [midway, below] {Object of $\bbAC_{2,0}({\color{red}(2)},{\color{blue}(2)}$,${\color{purple}(2)}$,${\color{olive}(2)}$,${\color{green}(1,1)})$};
            \draw (4.5,6)--(7.3,6) node [midway, below]{ Ass. map};
            \draw (7.3,6)--(12,6) node [midway, below] {Cone of metrics};
            \draw (-5,0.5)--(12,0.5);
            \draw (-5,-5.5)--(12,-5.5);
            \draw (-5,6)--(-5,-5.5);
            \draw (4.5,6)--(4.5,-5.5);
            \draw (7.3,6)--(7.3,-5.5);
            \draw (12,6)--(12,-5.5);

        \end{tikzpicture}
        \caption{Cones of metrics of some objects of $\bbAC_{2,0}({\color{red}(2)},{\color{blue}(2)}$,${\color{purple}(2)}$,${\color{olive}(2)}$,${\color{green}(1,1)})$.}
        \label{fig: cones of metrics ac20}
    \end{figure}
\end{example}

\begin{const}\label{const: functoriality of achm}
   Following Notation \ref{nota: standing hdmmu} consider objects $\pi\colon G\to H$ and $\pi^\prime\colon G^\prime\to H^\prime$ of $\bbAC_{d,h}(\vec{\mu})$ together with a morphism $f\colon\pi\to\pi^\prime$. As in Lemma \ref{lem: map of graphs is contraction}, these morphisms induce natural identifications $G/f_\src^{-1}(V(G^\prime))\cong G^\prime$ and $H/f_\trgt^{-1}(V(H^\prime))\cong H^\prime$. Moreover, since $\pi^{-1}(f_\trgt^{-1}(V(H)))\subset f_\src^{-1}(V(G))$, it follows that $\pi$ induces a map $\pi/(f_\src,f_\trgt)\colon G/f_\src^{-1}(V(G^\prime))\to H/f_\trgt^{-1}(V(H^\prime))$. This map together with the natural identifications $G/f_\src^{-1}(V(G^\prime))\cong G^\prime$ and $H/f_\trgt^{-1}(V(H^\prime))\cong H^\prime$ yields a commutative diagram of maps of graphs
    \begin{equation*}
        \begin{tikzcd}
            G/f_\src^{-1}(V(G^\prime))\arrow{rr}{\sim} \arrow{d}[swap]{\pi/(f_\src,f_\trgt)}&&G^\prime \arrow{d}{\pi^\prime}\\
            H/f_\trgt^{-1}(V(H^\prime))\arrow{rr}[swap]{\sim}&& H^\prime
        \end{tikzcd}
    \end{equation*}
    In particular, the morphism $f$ induces the face-embedding $\sigma_{f_\src}\times\sigma_{f_\trgt}\colon \sigma_{H^\prime}\times\sigma_{G^\prime}\to \sigma_H\times\sigma_G$. We can then restrict this face-embedding to $\sigma_\pi$, and, in this way, obtain a face-embedding $\sigma_{(f_\src,f_\trgt)}\colon\sigma_\pi\to \sigma_{\pi^\prime}$.
\end{const}

\begin{lem}\label{lem: functoriality of achm}
     Following the notation of Construction \ref{const: functoriality of achm}, the cone of metrics construction gives rise to a functor
     \begin{equation*}
        \AC_{d,h,\vec{\mu}}\colon\bbAC_{d,h}^{\op}(\vec{\mu})\to\POIC,\pi\mapsto\sigma_\pi.
     \end{equation*}
     Moreover, if $\pi$ is an object of $\bbAC_{d,h}(\vec{\mu})$ and $\tau$ is a face of $\sigma_\pi$, then there exists a morphism $\pi\to \pi^\prime$ of $\bbAC_{d,h}(\vec{\mu})$ unique up to isomorphism with $\AC_{d,h,\vec{\mu}}(\pi^\prime\to\pi)$ isomorphic to this face-embedding. In other words, $\AC_{d,h,\vec{\mu}}$ is a poic-space, and, in addition, the poic-morphisms \eqref{eq: src trgt poic-morphism} define natural transformations 
     \begin{align*}
         &\eta_{\src}\colon \AC_{d,h,\vec{\mu}}\to \Mtrop_{g,n}\circ \src,
         &\eta_{\trgt}\colon \AC_{d,h,\vec{\mu}}\to \Mtrop_{h,m}\circ \trgt,
     \end{align*}
     which give rise to morphisms of poic-spaces $\src\colon \AC_{d,h,\vec{\mu}}\to \Mtrop_{g,n}$ and $\trgt\colon \AC_{d,h,\vec{\mu}}\to \Mtrop_{h,m}$.
\end{lem}
\begin{proof}
    This follows from the properties of \eqref{eq: moduli space is poic-space}. Analogously, the fact that the poic-morphisms \eqref{eq: src trgt poic-morphism} define natural transformations follows from the definition of the Hom-sets of the category $\bbAC_{d,h}(\vec{\mu})$ and the cone of metrics of the cover.
\end{proof}

We remark that $\bbAC_{d,h}(\vec{\mu})^{\op}$ has finitely many isomorphism classes (for all the possible choices of $d$, $h$, $m$, and $\vec{\mu}$), and thus is necessarily equivalent to a finite category. As before, this implies that the colimit of any functor from $\bbAC_{d,h}(\vec{\mu})^{\op}$ to $\Top$ is representable.

\begin{defi}
    Let $d$, $h$, $m$, $\vec{\mu}$, $n$ and $g$ be as in Definition \ref{defi: category achm mu}. The \emph{moduli space of discrete admissible covers with ramification $\vec{\mu}$ of genus-$h$ $m$-marked curves} is the topological space 
    \begin{equation*}
        \cAC_{d,h}^\trop(\vec{\mu}) := \colim_{\bbAC_{d,h}(\vec{\mu})^{\op}} \left|\AC_{d,h,\vec{\mu}}\right|.\end{equation*}    
\end{defi}
The above construction shows that $\cAC_{d,h}^\trop(\vec{\mu})$ can be constructed by making a choice of a full set of representatives for $\bbAC_{d,h}(\vec{\mu})^{\op}$, taking the corresponding cones of metrics, and identifying points representing isomorphic covers. This presentation not only emphasizes the combinatorial relations between cones of metrics given by face relations and isomorphic identifications, but also our construction of the associated cone of metrics allows us to take edge contractions with respect to the target curve while controlling them with respect to the source curve.

\subsection{Extended spanning trees.}
We now proceed with a construction for our categories of discrete admissible covers analogous to that of the spanning tree fibrations. Suppose $h\geq0$, $d\geq 0$, $m\geq 0$, $\vec{\mu}=(\mu_1,\dots,\mu_m)$, $n$ and $g$ are as in Notation \ref{nota: standing hdmmu}. We first observe that the product of linear poic-fibrations is a linear poic-fibration, and therefore the linear poic-fibrations
\begin{align*}
    &\st_{g,n}\colon \ST_{g,n}\to \Mtrop_{g,n},
    &\st_{h,m}\colon \ST_{h,m}\to \Mtrop_{h,m},
\end{align*}
give a linear poic-fibration 
\begin{equation*}
    \st_{g,n}\times\st_{h,m}\colon \ST_{g,n}\times\ST_{h,m}\to \Mtrop_{g,n}\times\Mtrop_{h,m}.
\end{equation*}
In order to get a poic-fibration over $\AC_{d,h,\vec{\mu}}$ we pullback the previously constructed linear poic-fibration through the functor $\src\times\trgt\colon \bbAC_{d,h}(\vec{\mu})\to \bbG_{g,n}\times\bbG_{h,m}$. More precisely, let $J\subset \{1,\dots,n\}$ be arbitrary. We will construct a linear poic-complex $\EST^J_{d,h,\vec{\mu}}$ together with a poic-fibration $\est^J_{d,h,\vec{\mu}}\colon \EST^J_{d,h,\vec{\mu}}\to \AC_{d,h}(\vec{\mu})$, in a similar way, but by pulling back $\est_{g,J}\times\est_{h,m}$ through $\ft_{J^c}\circ \mathfrak{src}\times \mathfrak{trgt}$. The motivation underlying such a construction is to study tropical  cycles in $\AC_{d,h,\vec{\mu}}$ with respect to a linear poic-complex with a morphism $\mathfrak{src}\colon \est_{d,h,\vec{\mu}}^J\to \st_{g,J}$ that is weakly proper in top dimension.

\begin{defi}\label{defi: J extended ac}
Let $h$, $m$, $d$, $\vec{\mu}$, $n$, and $g$ be as in Notation \ref{nota: standing hdmmu}, and let $J\subset\{1,\dots,n\}$ be arbitrary. The \emph{category of $J$-marked extended admissible covers of degree $d$ to genus-$h$ $m$-marked graphs and ramification $\vec{\mu}$} is the category $\bbEAC_{g,h}^J(\vec{\mu})$ given by the following data:
\begin{itemize}
    \item The objects are the triples $(T_s,T_t,\pi)$, where $T_s$ is an object of $\bbG_{0, J\sqcup \mathbbm{g}}$, $T_t$ is an object of $\bbG_{0,m+2h}$, and $\pi$ is an object of $\bbAC_{d,h}(\vec{\mu})$ such that 
    \begin{align*}
        &\ft_{J^c}\left(\src(\pi)\right) = \st_{g,J}(T_s), & &\trgt(\pi) = \st_{h,m}(T_t). 
    \end{align*}
    \item For two objects $(T_s,T_t,\pi)$ and $(T^\prime_s,T^\prime_t,\pi^\prime)$ the set of morphisms is defined as the subset of $\Hom_{\bbG_{0,J\sqcup\mathbbm{g}}}(T_t,T^\prime_t)\times\Hom_{\bbG_{0,h+2m}}(T_s,T^\prime_s)\times\Hom_{\bbAC_{d,h}(\vec{\mu}}(\pi,\pi^\prime) $ consisting of the triples $(f_s,f_t,f)$ such that
    \begin{align}
        &\ft_{J^c}(f_\src) = \st_{g,J}(f_s), &f_\trgt = \st_{h,m}(f_t).
    \end{align}
    \item Composition of morphisms is just composition of triples.
\end{itemize}
It is clear that this category is essentially finite. By construction, this category comes naturally equipped with three functors:
\begin{align}
        {\est^J_{d,h,\vec{\mu}}}&\colon \bbEAC^J_{d,h}(\vec{\mu})\to \bbAC_{d,h}(\vec{\mu}), &&(T_s,T_t,\pi)\mapsto \pi, \label{eq: underlying functor estJ fibration}\\
        \src&\colon \bbEAC^J_{d,h}(\vec{\mu})\to \bbG_{0,J\sqcup\mathbbm{g}}, &&(T_s,T_t,\pi)\mapsto T_s,  \label{eq: src functor eachmJ mu}\\
        \trgt&\colon \bbEAC^J_{d,h}(\vec{\mu})\to \bbG_{0,2h+m}, &&(T_s,T_t,\pi)\mapsto T_t. \label{eq: trgt functor eachmJ mu}
    \end{align}
\end{defi}
\begin{lem}\label{lem: extended is thin}
Let $h$, $m$, $d$, $\vec{\mu}$, $n$, and $g$ be as in Notation \ref{nota: standing hdmmu}. The category $\bbEAC^J_{d,h}(\vec{\mu})$ is thin.
\end{lem}
\begin{proof}
Let $(T_s,T_t,\pi)$ and $(T^\prime_s,T^\prime_t,\pi^\prime)$ be two objects of $\bbEAC^J_{d,h}(\vec{\mu})$, and consider two morphisms 
\begin{equation}
(f_s,f_t,f),(f_s^\prime,f_t^\prime,f^\prime)\in\Hom_{\bbEAC^J_{d,h}(\vec{\mu})}\left( (T_s,T_t,\pi),(T^\prime_s,T^\prime_t,\pi^\prime) \right).
\end{equation}
Since $\bbG_{0,2h+m}$ and $\bbG_{0,J\sqcup\mathbbm{g}}$ are thin, it follows that $f_t=f_t^\prime$ and $f_s=f_s^\prime$. In particular, $f_\trgt = f^\prime_\trgt$. Moreover, $\ft_{I^c}(f_\src) = \ft_{I^c}(f_\src)$, implies that $\ft_{J\cup J^c}(f_\src) = \ft_{J\cup J^c}(f_\src)$. Since $f_\src$ and $f^\prime_\src$ must preserve the $n$-marking, the previous equality forces $f_\src = f_\src^\prime$. 
\end{proof}

\begin{const}\label{const: various cones for eacJ.}
    Let $(T_s,T_t,\pi)$ be an object of $\bbEAC_{d,h}^J(\vec{\mu})$. In the spirit of notational simplicity, let $G= \st_{g,J}(T_s)$ and $H=\st_{h,m}(T_t)$. The poic-fibrations $\st_{g,J}$ and $\st_{h,t}$ give a poic-morphism
    \begin{equation}
       \eta_{\st_{h,m},T_t}\times \eta_{\st_{g,J},T_s}: \ST_{T_t}\times\ST_{T_s} \to \sigma_H\times \sigma_G, \label{eq: linear map product of the two fibrations J}
    \end{equation}
    which is an embedding. On the other hand, the admissible cover $\pi:G\to H$ gives the poic $\sigma_\pi$, and forgetting the marking yields a morphism 
    \begin{equation}
        \Id\times\eta_{\ft_{J^c},\src(\pi)}\colon \sigma_\pi\to\sigma_H\times\sigma_G
    \end{equation}
    \end{const}
    \begin{defi}
    Following the notation ofConstruction \ref{const: various cones for eacJ.}, the \emph{extended cone of metrics} of the triple $(T_s,T_t,\pi)$ is the poic $\EST^J_{(T_s,T_t,\pi)}$ defined as the inverse image of $\Id\times\eta_{\ft_{J^c},\src(\pi)}\left(\sigma_\pi\right)$ under the map \eqref{eq: linear map product of the two fibrations J}. This poic naturally comes with an inclusion 
    \begin{equation}
        \eta_{\est^J_{d,h,\vec{\mu}},(T_s,T_t,\pi)}:\EST^J_{(T_s,T_t,\pi)}\to \Id\times\ft_{J^c}\left(\sigma_\pi\right), \label{eq: embedding for the triples J}
    \end{equation}
    as well as poic-morphisms given by the projection maps
    \begin{align}
        &\eta_{\src,(T_s,T_t,\pi)}\colon \EST^J_{(T_s,T_t,\pi)}\to \sigma_{T_s}, &\eta_{\trgt,(T_s,T_t,\pi)}\colon \EST^J_{(T_s,T_t,\pi)}\to \sigma_{T_t}\label{eq: trgt nt eachmJ mu}.
    \end{align}
    The natural projections, $\ST_{T_s}\times\ST_{T_t}\to\ST_{T_t}$ and $\ST_{T_s}\times\ST_{T_t}\to\ST_{T_s}$, together with the poic-morphisms, $(D_{g,J})_{T_s}$ and $(D_{h,m})_{T_t}$, define a poic-morphism
    \begin{equation}
        (D^J_{d,h,\vec{\mu}})_{(T_s,T_t,\pi)}\colon \EST^J_{(T_s,T_t,\pi)}\to \underline{\left(N_{\dist_{J\sqcup\mathbbm{g}}}\oplus N_{\dist_{2h+m}}\right)_\bbR}. \label{eq: linear structure for the triples J}
    \end{equation}
\end{defi}
\begin{remark}
    For an object $(T_s,T_t,\pi)$ of $\bbEAC^J_{d,h}(\vec{\mu})$, the relative interior of the cone $\EST^J_{(T_s,T_t,\pi)}$ corresponds to all possible metrics on $T_s$ and $T_t$ that are related under $\pi$ and forgetting the marking. Additionally, the faces of this cone correspond to the possible contractions of both $T_s$ and $T_t$ that are related through $\pi$, which are governed by contractions of the target.
\end{remark}

\begin{prop}\label{prop: est J as linear poic-complex}
    The association
    \begin{equation}
        \EST^J_{d,h,\vec{\mu}}: \bbEAC^J_{d,h}(\vec{\mu})^\op\to \POIC, (T_s,T_t,\pi)\mapsto \EST_{(T_s,T_t,\pi)}\label{eq: estjdhmu}
    \end{equation}
    is a poic-complex. The poic-morphisms \eqref{eq: linear structure for the triples J} define a morphism of poic-complexes
    \begin{equation}
        D^J_{d,h,\vec{\mu}}\colon\EST^J_{d,h,\vec{\mu}}\to \underline{\left(N_{\dist_{J\sqcup\mathbbm{g}}}\oplus N_{\dist_{2h+m}}\right)_\bbR},\label{eq: linear poic-complex est J}.
    \end{equation}
    The poic-morphisms \eqref{eq: embedding for the triples J} define a natural transformation $\eta_{\est^J_{d,h,\vec{\mu}}}\colon\EST^J_{d,h,\vec{\mu}}\to\AC_{d,h,\vec{\mu}}$, which together with \eqref{eq: underlying functor estJ fibration} (and the linear poic-complex structure given by \eqref{eq: linear poic-complex est J}) defines a linear poic-fibration
    \begin{equation}
        \est^J_{d,h,\vec{\mu}}:\left(\EST^J_{d,h,\vec{\mu}}\right)_{D^J_{d,h,\vec{\mu}}}\to \AC_{d,h,\vec{\mu}}. \label{eq: estJ fibration}
    \end{equation}
z<    The poic-morphisms \eqref{eq: trgt nt eachmJ mu} define corresponding natural transformations, which together with the integral linear maps given by the natural projections define morphisms of linear poic-complexes 
    \begin{align}
        &\src \colon \EST^J_{d,h,\vec{\mu}}\to \ST_{g,J},\label{eq: src lpcJ}, &\trgt\colon \EST^J_{d,h,\vec{\mu}}\to \ST_{h,m},
    \end{align}
    where $\src$ is weakly proper in top dimension. If $J=\{1,\dots, n\}$, then $\src$ is weakly proper.
\end{prop}
\begin{proof}
    It is clear that $\bbEAC_{d,h}^J(\vec{\mu})$ is essentially finite, and we have readily shown in Lemma \ref{lem: extended is thin} that it is thin. The association $\EST^J_{d,h,\vec{\mu}}$ is a functor, because $\ST_{g,n}$, $\ST_{h,m}$, and $\AC_{d,h,\vec{\mu}}$ are functors. The definition of the functor $\EST^J_{d,h,\vec{\mu}}$ at triples, together with the fact that $\ST_{g,n}$, $\ST_{h,m}$, and $\AC_{g,h,\vec{\mu}}$ are also poic-complexes, implies that $\EST^J_{d,h,\vec{\mu}}$ is also a poic-complex.vIt follows from the definition of the inclusion \eqref{eq: embedding for the triples J}, and the fact that $\eta_{\st_{h,m}}$ and $\eta_{\st_{g,n}}$ are natural transformations, that $\eta_{\est^J_{d,h,\vec{\mu}}}$ is also a natural transformation. Hnece, $\est^J_{d,h,\vec{\mu}} = ({\est^J_{d,h,\vec{\mu}}},\eta_{\est^J_{d,h,\vec{\mu}}})$ is a morphism of poic-spaces, and it is clear that the functor ${\est^J_{d,h,\vec{\mu}}}$ is essentially surjective. The category $\bbEAC_{d,h}^J (\vec{\mu})$ and the morphism $\est^J_{d,h,\vec{\mu}}$ are obtained by pulling back the (linear) poic-fibration $\st_{h,m}\times\st_{g,n}$. Thus, it can be readily checked that they preserve the required lifting properties. For the condition on relative interiors, let $(T_s,T_t,\pi)$ be an object of $\bbEAC^J_{d,h}(\vec{\mu})$. Since both $\eta_{\st_{h,m},T_H}$ and $\eta_{\st_{g,n},T_G}$ induce an isomorphism between the corresponding relative interiors, this also holds for their product and, by definition, for $\eta_{\est^J_{d,h,\vec{\mu}}}$, because it is the restriction of this product to a dense subcone. As $D_{g,n}$ and $D_{h,m}$ are morphisms of poic-complexes, it is clear that the poic-morphisms \eqref{eq: linear structure for the triples J} define the morphism of poic-complexes \eqref{eq: linear poic-complex est J}. The statement concerning the poic-morphisms \eqref{eq: src lpcJ} is analogous to that of $\bbAC_{d,h}(\vec{\mu})$, so it is omitted. To finalize, we show that $\src$ is weakly proper in top dimension. Consider an object $(T_s,T_t,\pi)$ of $\bbEAC_{d,h}^J(\vec{\mu})$. If $J=\{1,\dots,n\}$, then the only possible edge contractions of the source come from edge contractions of the target, and every face of $\EST^J_{(T_s,T_t,\pi)}$ comes from one of $\ST_{T_s}$. If $J\subsetneq\{1,\dots,n\}$, then we obtain weak properness in the top dimension because the forgetting the marking morphism gives further positivity conditions between several edges of the target tree $T_t$.
\end{proof}

\begin{defi}\label{defi: extended spanning tree J}
    Suppose $h$, $m$, $d$, $\mu$, $n$, $g$, and $J$ are as in Definition \ref{defi: J extended ac}. We call $\EST^J_{d,h,\vec{\mu}}$, with the linear poic-complex structure \eqref{eq: linear poic-complex est J}, the \emph{linear poic-complex of $J$-marked extended admissible covers of degree $d$ with ramification $\vec{\mu}$ to $m$-marked genus-$h$ graphs}, and the poic-fibration \eqref{eq: estJ fibration} the \emph{extended spanning tree fibration of $J$-marked discrete admissible covers of genus-$h$ $m$-marked graphs with ramification $\vec{\mu}$}.\\
    If $J=\{1,\dots,n\}$, then we will denote $\EST^J_{d,h,\vec{\mu}}$ and $\est^J_{d,h,\vec{\mu}}$ simply by $\EST_{d,h,\vec{\mu}}$ and $\est_{d,h,\vec{\mu}}$. In this case, we call $\EST_{d,h,\vec{\mu}}$ the \emph{linear poic-complex of extended admissible covers of degree $d$ with ramification $\vec{\mu}$ to $m$-marked genus-$h$ graphs}, and the poic-fibration $\est_{d,h,\vec{\mu}}$ the \emph{extended spanning tree fibration of discrete admissible covers of genus-$h$ $m$-marked graphs with ramification $\vec{\mu}$}. 
\end{defi}

\begin{prop}\label{prop: est linear poic fib J}
    The morphisms of linear poic-complexes \eqref{eq: src lpcJ}, together with the morphisms of poic-spaces $\ft_{J^c}\circ \src\colon \AC_{d,h,\vec{\mu}}\to \Mtrop_{g,J}$ and $\trgt\colon \AC_{d,h,\vec{\mu}}\to \Mtrop_{h,m}$ define morphisms of linear poic-fibrations:
        \begin{align}
            \mathfrak{src}\colon \est^J_{d,h,\vec{\mu}}\to \st_{g,J},&&
            \mathfrak{trgt}\colon \est^J_{d,h,\vec{\mu}}\to\st_{h,m}, \label{eq: src and trgt poic-fibration J}
        \end{align}
        where $\mathfrak{src}^\pcmplxs$ and $\mathfrak{src}^\pspcs$ are the corresponding $\src$ morphisms, and $\mathfrak{trgt}^\pcmplxs$ and $\mathfrak{trgt}^\pspcs$ are the corresponding $\trgt$ morphisms. Additionally, the morphism of linear poic-complexes $\mathfrak{src}\colon\est_{d,h,\vec{\mu}}\to\st_{g,n}$ is weakly proper in top dimension.
\end{prop}
\begin{proof}
    The commutativity relations follow directly from the definition, and the weak properness in top dimension of $\mathfrak{src}^\pcmplxs$ has readily been stablished in Proposition \ref{prop: est J as linear poic-complex}. It remains to establish the lifting property of both $\mathfrak{src}^\pspcs$ and $\mathfrak{trgt}^\pspcs$, which are a direct consequence of the following facts: 
    \begin{itemize}
        \item isomorphisms of $\bbG_{g,n}$ (and $\bbG_{g,J}$) define corresponding degree-$1$ discrete admissible covers,
        \item the degree of a composition of harmonic morphisms is the product of the degrees. 
    \end{itemize}
\end{proof}

\section{Cycles of discrete admissible covers.}

Let $d$, $h$, $m$, $\vec{\mu}$, $n$, and $g$ be as in Notation \ref{nota: standing hdmmu}, and let $J\subset\{1,\dots,n\}$. In this section, we prove that the usual weight assignment on discrete admissible covers (\cite{CavalieriMarkwigRanganathan}) defines an $\est^J_{d,h,\vec{\mu}}$-equivariant Minkowski weight. The proof of this general result requires an additional algebrogeometric input and occupies a significant portion of the section. We explore a particular enumerative application of this theorem afterwards concerning covers of tropical rational curves (namely, metric trees).

\subsection{The standard weight}\label{sec: standard weight}
Let $d$, $h$, $m$, $\vec{\mu}$, and $g$ be as in Notation \ref{nota: standing hdmmu}, and let $J\subset\{1,\dots,n\}$. Consider the extended spanning tree fibration $\est^J_{d,h,\vec{\mu}}:\EST^J_{d,h,\vec{\mu}}\to \AC_{d,h,\vec{\mu}}$ of $J$-marked discrete admissible covers of genus-$h$ $m$-marked graphs with ramification $\vec{\mu}$ from Definition \ref{defi: extended spanning tree J}. It follows from \eqref{eq: estjdhmu} and the definition of $\EST^J_{d,h,\vec{\mu}}$ that every cone of this poic-complex is of dimension at most $m+3(h-1)$. We define the standard weight on the $(m+3(h-1))$-dimensional cones following the usual assignment (Definition 22 of \cite{CavalieriMarkwigRanganathan}\footnote{It is important to remark that in (W3) only the internal edges are considered. See Theorem 34 of loc. cit. for instance.}).
\begin{defi}\label{defi: weight of the cover}
    For a cover $\pi\colon G\to H$ of $\bbAC_{d,h}(\vec{\mu})$ the \emph{standard weight $\varpi^J_{d,h,\vec{\mu}}(\pi) $ of the cover} is the number given by 
    \begin{equation}
        \varpi^J_{d,h,\vec{\mu}}(\pi) = \frac{\left(\prod_{e\in E(G)}d_\pi(e)\right)\cdot \left(\prod_{V\in V(G)} H(V)\cdot \mathrm{CF}(V)\right)}{\prod_{h\in E(H)}\mathrm{lcm}\left((d_\pi(e))_{e\in\pi^{-1}(h)}\right)} ,\label{eq: standard weight}
    \end{equation}
    where for $V\in V(G)$:
    \begin{itemize}
        \item $H(V)$ denotes the local rational connected Hurwitz number. Namely, around $V$ the cover induces integer partitions partitions $\lambda_1,\dots,\lambda_k$ of $d_\pi(V)$ and therefore $H(V) = H_{0\to 0}(\lambda_1,\dots,\lambda_k)$\footnote{This Hurwitz number $H_{0\to 0}(\lambda_1,\dots,\lambda_k)$ is just the product of $\frac{1}{d!}$ with the number of isomorphism classes of branched coverings $f\colon\bbP^1\to\bbP^1$, where the branching locus consists of $m$ distinct points $p_1,\dots,p_m$ and the ramification profile of $f$ at $p_i$ is $\lambda_i$ (for $1\leq i\leq m$).}. 
        \item $\mathrm{CF}(V)$ is the product of factors of $k!$ for each $k$-tuple of adjacent edges or legs of the same weight that map to the same edge of the target.
    \end{itemize}
    The \emph{standard weight $\varpi^J_{d,h,\vec{\mu}}$} of $\EST^J_{d,h,\vec{\mu}}$ is the rational weight
\begin{equation*}
    \varpi^J_{d,h,\vec{\mu}}: [\bbEAC^J_{d,h}(\vec{\mu})](m+3(h-1))\to \bbQ, [(T_s
    ,T_t,\pi)]\mapsto \varpi^J_{d,h,\vec{\mu}}(\pi).
\end{equation*}
It is clear that if $\pi\cong \pi^\prime$, then $\varpi^J_{d,h,\vec{\mu}}(\pi)=\varpi^J_{d,h,\vec{\mu}}(\pi^\prime)$. This means that 
\begin{equation}
\varpi^J_{d,h,\vec{\mu}}\in [\est^J_{d,h,\vec{\mu}}]^*W_{m+3(h-1)}(\AC_{d,h,\vec{\mu}}).\label{eq: equivariance}
\end{equation}
\end{defi}

\begin{nota}
If $J=\{1,\dots,n\}$, then we simply denote $\varpi^J_{d,h,\vec{\mu}}$ by $\varpi_{d,h,\vec{\mu}}$.
\end{nota}

\begin{restatable}{thm}{globalbalancing}\label{thm: global balancing}
    The standard weight $\varpi^J_{d,h,\vec{\mu}}$ is an $\est^J_{d,h,\vec{\mu}}$-equivariant  Minkowski weight.
\end{restatable}

The proof is postponed until the end of the next section. It is essentially split into two parts: a local deformation picture, and a local-to-global balancing argument. More precisely, the balancing of the standard weight comes from studying how an admissible cover can be deformed through edge contractions and extensions of the base graph. We first restrict ourselves to the behavior of a discrete admissible cover around a fixed edge of the target graph. From this, we obtain multiple admissible covers (by genus-$0$ curves) of a $3$-valent $4$-marked tree with a single edge. Then the contraction and extension of this single edge can be studied at each of these (our local deformation picture), and the global behavior of the cover can be understood out of these local cases (local-to-global balancing). As it turns out, the most intricate part corresponds to the local deformation picture which demands an algebrogeometric input. We do remark that we are not approaching any result concerning the corresponding existence problem. In our terms, this translates to determining when the categories $\bbAC_{d,h}(\vec{\mu})$ and $\bbEAC_{d,h}(\vec{\mu})$ are non-empty, and when do these have top dimensional cones with non-vanishing weight. On the contrary, we just show that this weight assignment (which might be trivial if the categories are empty) defines an equivariant Minkowski weight. We do subsequently  handle specific cases where these categories are non-empty and produce non-trivial cycles.

\subsection{Balancing of the standard weight} Let $d\geq 0$ be an integer and consider integer partitions $\alpha,\beta,\gamma,\delta \vdash d$, such that
\begin{equation} \ell(\alpha)+\ell(\beta)+\ell(\gamma)+\ell(\delta) = 2(d+1).\label{eq: RH for 4 valent}\end{equation}
We first consider $J =\{1,\dots,2(d+1)\}$ and focus on the linear poic-fibration $\est_{d,0,(\alpha,\beta,\gamma,\delta)}$. The condition given by the equation \eqref{eq: RH for 4 valent} has several implications:
\begin{itemize}    
    \item If $\pi\colon G\to T$ is an object of $\bbAC_{d,h}(\alpha,\beta,\gamma,\delta)$, then $g(G) = 0$.
    \item The category $\bbEAC_{d,h}(\alpha,\beta,\gamma,\delta)$ coincides with $\bbAC_{d,h}(\alpha,\beta,\gamma,\delta)$, and the poic-fibration $\est_{d,0,(\alpha,\beta,\gamma,\delta)}$ is just the identity functor.
    \item The category $\bbEAC_{d,h}(\alpha,\beta,\gamma,\delta)$ has an initial object. This can be represented as the cover of a graph defined by a single $4$-valent vertex from a graph defined by single vertex with valency $2(d+1)$, where the first $\ell(\alpha)$-marked legs lie above the first leg, the next $\ell(\beta)$-marked legs lie above the second leg, the next $\ell(\gamma)$-marked legs lie above the third leg, and the last $\ell(\delta)$-legs lie above the fourth leg.
\end{itemize}
In this case, the morphism of poic-complexes $D_{d,0,(\alpha,\beta,\gamma,\delta)}$ \eqref{eq: linear poic-complex est J} presents this linear poic-complex $\EST_{d,0,(\alpha,\beta,\gamma,\delta)}$ as a $1$-dimensional fan in the underlying vector space of $\cM_{0,2(d+1)}^\trop\times \cM_{0,4}^\trop$. Let $\Theta\in[\bbEAC_{d,h}(\alpha,\beta,\gamma,\delta)](1)$, with $(G_\Theta,H_\Theta,\pi_\Theta)$ an object of $\bbEAC_{d,h}(\alpha,\beta,\gamma,\delta)$ representing this isomorphism class. The integral generator of the ray given by $\Theta$ is $u_\theta =(u_{\src(\Theta)},u_{\trgt(\Theta)})$, where 
\begin{itemize}
    \item $u_{\trgt(\Theta)}$ is the integral generator of the ray in $\cM_{0,4}^\trop$ given by the isomorphism class $[H_\Theta]$,
    \item $u_{\src(\Theta)}$ is the integral generator of the ray generated by $ F_\pi(u_{\trgt(\Theta)})$, where $F_\pi$ is the matrix \eqref{eq: assoc matrix} associated to the corresponding admissible cover $\pi$ (this is invariant with respect to the isomorphism class).
\end{itemize}
The balancing of the fundamental cycle $\varpi_{d,0,(\alpha,\beta,\gamma,\delta)}$ is equivalent to showing the equation:
\begin{equation}
    \sum_{\Theta\in [\bbEAC_{d,h}(\alpha,\beta,\gamma,\delta)](1)} \varpi_{d,0,(\alpha,\beta,\gamma,\delta)}\left(\Theta\right)\cdot u_\Theta = 0. \label{eq: local balancing}
\end{equation}

Similar situations have already been studied in \cite{GathmannMarkwigOchseTMSSMC}, \cite{GathmannOchseMSCTV}, and \cite{BuchholzMarkwigTCCMSs} (and our approach to prove the balancing follows their spirit). These cover some special cases of \eqref{eq: local balancing}, but the general situation resisted its phrasing under these previous works. To show \eqref{eq: local balancing} we make use of Jun Li's degeneration formula in an analogous algebrogeometric situation. Namely, we look at (relative) stable maps of degree $d$ from rational curves to $\bbP^1$ with $4$ marked points, where the ramification above these marked points is prescribed by our partitions, and how these degenerate when two of the marked points come together into a new branch. This set-up fits under the deep and more general machinery of \cite{JUnLiSMSS} and \cite{JunLiDF}, and we apply the results therefrom for our case at hand. More specifically, we seek to use the degeneration formula from \cite{JunLiDF} to our advantage, which, for convenience, we recall together with its whole set-up under the notation of \cite{JUnLiSMSS} and \cite{JunLiDF} (we also intend to provide exact references thereto). Before proceeding, it is also worth mentioning that the log geometric version of this machinery \cite{KimLhoRuddat} can also be applied to this context, and shortcuts some of our combinatorial technicalities arising in the proof of Lemma 4.2.1. 

We take $\bbC$ as a ground field. Consider $\bbP^1$ with a closed point $0\in \bbP^1$, and a flat and projective family $W\to \bbP^1$ such that:
\begin{itemize}
    \item The fiber $W_t$ over any closed $t\in\bbP^1$ with $t\neq 0$ is isomorphic to $\bbP^1$.
    \item The fiber over $0\in\bbP^1$ is the union of two $\bbP^1$'s transversally intersecting at a point. For notational simplicity we denote this fiber by $W_0$ with the point of intersection $P$. In addition, we denote the corresponding $\bbP^1$'s by $Y_1$ and $Y_2$, and denote the point $P$ lying in $Y_i$ by $P_i\in Y_i$.
\end{itemize}
Let $\mathfrak{W}$ denote the stack of expanded degenerations of $W/\bbP^1$ (see Section 1.2 and Definition 1.9 of \cite{JUnLiSMSS}), and let $\mathfrak{M}(\mathfrak{W},\Gamma)$, with $\Gamma=(0,2(d+1),d)$, denote the moduli stack of stable morphisms to $\mathfrak{W}$ (see Section 3.1 of loc. cit.) of degree $d$ from $2(d+1)$-marked rational curves. It is shown in loc. cit. that this is a Deligne-Mumford stack that is separated and proper over $\bbP^1$ (Theorem 3.10 of loc. cit.). In addition, when $t\neq 0$ the fiber products $\mathfrak{M}(\mathfrak{W},\Gamma) \times_{\bbP^1} \{t\}$ are naturally isomorphic to the moduli stack of stable morphisms to $W_t$ of the prescribed topological type. It is shown in Theorem 2.5 of \cite{JunLiDF} that at the closed point $0\in\bbP^1$ the fiber product 
\begin{equation*}
    \mathfrak{M}(\mathfrak{W}_0,\Gamma):=\mathfrak{M}(\mathfrak{W},\Gamma) \times_{\bbP^1} \{0 \}
\end{equation*}
has a perfect obstruction theory and carries a well defined virtual fundamental cycle. It is shown in Theorem 3.15 of loc. cit. that this cycle satisfies the following degeneration formula
\begin{equation}
[\mathfrak{M}(\mathfrak{W}_0,\Gamma)]^{\virt}
=\sum_{\eta\in\overline{\Omega}}\frac{m(\eta)}{\abs{\Eq(\eta)}}\cdot {\Phi_\eta}_* [\mathfrak{M}(\mathfrak{Y}_1^{\rel},\Gamma_1)]^{\virt} \times [\mathfrak{M}(\mathfrak{Y}_2^{\rel},\Gamma_2)]^{\virt},\label{eq: degeneration formula full}
\end{equation}
where the notation is as follows:
\begin{itemize}
    \item The $\Gamma_i$ denote topological types, which are specified by admissible graphs (Definition 4.6 of \cite{JUnLiSMSS}) for $(Y_i,P_i)$ (for $i=1,2$). An admissible graph (not to be confused with our previous discrete graphs) $\Gamma_i$ for $(Y_i,P_i)$ is simply a finite collection of vertices, legs and roots (legs and roots are line segments with only one end attached to a vertex), with:
    \begin{enumerate}
        \item An ordering of the legs, an ordering and weight on the roots, and two weight functions on the set of vertices
        \begin{align*}
            g:V(\Gamma_i)\to \bbZ_{\geq0},  & &b:V(\Gamma_i)\to H_2Y_i,
        \end{align*}
        where $b$ assigns algebraic homology classes to vertices of $\Gamma_i$. In our case of interest, the function $g$ is constant with value $0$, and the functions $b$ are determined by an integer (the degree).
        \item The graphs $\Gamma_i$ are relatively connected. This means, by definition, that either $V(\Gamma_i)$ is a single element, or each vertex in $V(\Gamma_i)$ has at least one root attached to it.
    \end{enumerate}
    The admissible graphs $\Gamma_1$ and $\Gamma_2$ appearing in the formula are not arbitrary. These are given by the $\eta\in\overline{\Omega}$ (this will be shortly explained), thus they have some compatibility conditions.
    \item The set $\bar{\Omega}$ is the quotient of the set $\Omega$, which we now describe, by an equivalence relation, which we describe shortly afterwards. The set $\Omega$ consists of the triples $\left(\Gamma_1, \Gamma_2, I\right)$, where:
    \begin{enumerate}
        \item $\Gamma_1$ and $\Gamma_2$ are two admissible graphs that have the same number of roots and such that the weights of their $j$-th roots coincide. In addition, our case of interest requires the total number of legs adds to $2(d+1)$ and the total sum of the degrees is $d$.
        \item The graph obtained by joining all the $j$-th roots is a connected tree (the tree condition is due to our case of interest) and has no roots itself.
        \item The $I$ in the triple is an ordering on the set given by the legs of both $\Gamma_1$ and $\Gamma_2$. It is assumed to extend the ordering of the legs of each $\Gamma_i$.
    \end{enumerate}
    The equivalence relation is induced from the natural action of the symmetric groups on the roots. More precisely, let $\eta\in\Omega$. If the graphs of $\eta$ have $r$ roots, then any permutation $\sigma \in S_r$ defines a new element $\eta^\sigma$ by reordering the roots according to $\sigma$. For $\eta_1, \eta_2 \in \Omega$, we say $\eta_1 \sim \eta_2$ if $\eta_1=\eta_2^\sigma$ for some $\sigma$. Finally, the set $\bar{\Omega}$ is the set of equivalence classes $\Omega / S_r$. For an $\eta\in\overline{\Omega}$, the number $m(\eta)$ is just the product of the weights of the roots and $\Eq(\eta)$ denotes the stabilizer in $S_r$ of any representative of this class.
    \item The notation $\mathfrak{Y}_i^{\rel}$ refers to the stack of expanded relative pairs of $(Y_i,P_i)$ (see Section 4.1 and Definition 4.4 of \cite{JUnLiSMSS}).
    \item The cycle $[\mathfrak{M}(\mathfrak{Y}_i^{\rel},\Gamma_i)]^{\virt}$ refers to the virtual fundamental cycle (Proposition 3.9 of \cite{JunLiDF}) of $\mathfrak{M}(\mathfrak{Y}_i^{\rel},\Gamma_i)$, the moduli stack of relative stable morphisms to $\mathfrak{Y}_i^{\rel}$ of topological type $\Gamma_i$ (Definition 4.9 of \cite{JUnLiSMSS}). This is also a separated proper Deligne-Mumford stack (Proposition 4.10 of loc. cit.).
    \item The notation $\Phi_\eta$, for $\eta = (\Gamma_1,\Gamma_2,I)\in \overline{\Omega}$ as above, refers to the morphism defined by the root glueing construction (joining together the $j$-th roots of the underlying admissible graphs and reordering the legs according to $I$) applied to the corresponding (families) maps (Definition 4.11 and (4.10) of loc. cit.):
    \begin{equation*}
        \Phi_\eta: \mathfrak{M}\left(\mathfrak{Y}_1^{\rel}, \Gamma_1\right) \times \mathfrak{M}\left(\mathfrak{Y}_2^{\rel}, \Gamma_2\right) \longrightarrow \mathfrak{M}\left(\mathfrak{W}_0, \Gamma\right),
    \end{equation*}
    where $\Gamma=(0,2(d+1),d)$ is the topological type that we have been considering. It is shown in Proposition 4.13 of loc. cit. that this map is finite \'etale of pure degree $\# \Eq(\eta)$. The image $\mathfrak{M}(\mathfrak{Y}_1^{\rel}\sqcup\mathfrak{Y}_2^{\rel}, \Gamma)$ is the substack obtained by the glueing process. The degeneration formula takes the following form with the image stack (Corollary 3.13 of \cite{JunLiDF})
    \begin{equation}
        [\mathfrak{M}(\mathfrak{W}_0,\Gamma)]^{\virt}
=\sum_{\eta\in\overline{\Omega}}{m(\eta)}\cdot [\mathfrak{M}(\mathfrak{Y}_1^{\rel}\sqcup\mathfrak{Y}_2^{\rel},\eta)]^{\virt}.\label{eq: degeneration formula}
    \end{equation}
\end{itemize}

In general, the moduli stack $\mathfrak{M}(\mathfrak{W},\Gamma)$ does not contain as a substack the usual moduli space of stable maps $\overline{M}_{0,2(d+1)}(W,d)$, where $d$ refers to $d\cdot[\text{the fiber}]$. However, because of our setting we can fix our attention towards a common substack and aim to use \eqref{eq: degeneration formula} in order to gather information about our balancing weights. We fix distinct sections $s_\alpha$, $s_\beta$, $s_\gamma$, and $s_\delta$ of $W\to \bbP^1$ such that
\begin{equation*}Q_\alpha:=s_\alpha\times_{\bbP^1}\{0\},Q_\beta:=s_\beta\times_{\bbP^1}\{0\}\in Y_1 \textnormal{ and } Q_\gamma:=s_\gamma\times_{\bbP^1}\{0\},Q_\delta:=s_\delta\times_{\bbP^1}\{0\}\in Y_2
\end{equation*}
are distinct from the $P_i$ and consider inside the moduli space $\overline{M}_{0,2(d+1)}(W,d)$ the moduli subspace (substack) $\overline{M}_{0,(\alpha,\beta|\gamma,\delta)}(W,d)$ given by all stable maps $\cC=(C,x_1,\dots, x_{2(D+1)},\pi)$ where:
\begin{align*}
    \pi^*s_\alpha &=\sum_{i=1}^{\ell(\alpha)}  \alpha_i x_{i},& &\pi^*s_\beta =\sum_{i=1}^{\ell(\beta)} \beta_i x_{\ell(\alpha)+i},\\
    \pi^*s_\gamma &=\sum_{i=1}^{\ell(\gamma)} \gamma_i x_{\ell(\alpha)+\ell(\beta)+i},& &\pi^*s_\delta =\sum_{i=1}^{\ell(\delta)} \delta_i x_{\ell(\alpha)+\ell(\beta)+\ell(\gamma)+i},
\end{align*}
and $\alpha=(\alpha_i)$, $\beta=(\beta_i)$, $\gamma=(\gamma_i)$ and $\delta=(\delta_i)$.

This is a $1$-dimensional moduli subspace (substack), and it is also a substack of $\mathfrak{M}(\mathfrak{W},\Gamma)$. Our conditions on the partitions and special points force that there are no contracted components to take into account in the expanded degenerations. Indeed, in this specific context the expanded degenerations would add chains of $\bbP^1$'s but our conditions make that the new added components do not carry any special points different from the two nodes where other components are glued to. This is impossible because the curve has to have genus $0$ and must also be stable.

We denote the fiber product $\overline{M}_{0,(\alpha,\beta|\gamma,\delta)}(W,d)\times_{\bbP^1}\{0\}$ by $\overline{M}_{0,(\alpha,\beta|\gamma,\delta)}(W_0,d)$, which is zero dimensional. The underlying combinatorics of an isomorphism class $\Theta\in [\bbEAC_{d,h}(\alpha,\beta,\gamma,\delta)](1)$ determines a subset of points $\overline{M}_\Theta \subset \overline{M}_{0,(\alpha,\beta|\gamma,\delta)}(W_0,d)$, whenever the legs of the target tree with ramification $\alpha$ and $\beta$ share a common vertex. In fact, these subsets form a partition of $\overline{M}_{0,(\alpha,\beta|\gamma,\delta)}(W_0,d)$. Similar constructions and considerations give rise to the analogous moduli subspaces (substacks) $\overline{M}_{0,(\alpha,\gamma|\beta,\delta)}(W_0,d)$ and $\overline{M}_{0,(\alpha,\delta|\beta,\gamma)}(W_0,d)$.

\begin{nota}
    For a general subset $I=\{i,j,k,l\}$ of $\{1,\dots, 2(d+1)\}$ we have the forgetful morphism $\ft_{I^c}:\mathfrak{M}_{0,2(d+1)}(\mathfrak{W}_0,d)\to \overline{M}_{0,I}$, whose restriction to $\overline{M}_{0,(\alpha,\beta|\gamma,\delta)}(W_0,d)$ will be denoted identically. 
\end{nota}

\begin{lem}\label{lem: estimation of the order}
    Let $\pi\colon G\to H$ be an object of $\bbEAC_{d,h}(\alpha,\beta,\gamma,\delta)$, let $\Theta\in [\bbEAC_{0,4}(\alpha,\beta,\gamma,\delta)](1)$ denote its isomorphism class, and suppose that that the legs of $H$ with ramification $\alpha$ and $\beta$ share a common vertex. For a subset of distinct points $I=\{i,j,k,l\}$ of $\{1,\dots, 2(d+1)\}$, the order of  $\cC_\Theta\in \overline{M}_\Theta$ at the pull-back by $\ft_{I^c}$ of the divisor $(ij|kl)$ of $\overline{M}_{0,I}$ is 
    \begin{align*}
        \ord_{\cC_\Theta} \ft_{I^c}^*(ij|kl) =\prod_{V\in V(G)}\mathrm{CF}(V)\cdot \prod_{e\in E(G)}d_\pi(e)\cdot \sum_{e\in G(ij|kl)}\frac{1}{d_\pi(e)},
    \end{align*} 
    where $G(ij|kl)$ denotes the set of edges of $G$ between the tuples of legs $\{i,j\}$ and $\{k,l\}$ (that is, $G(ij|kl)$ consists of the edges lying in the intersection of: the path from $i$ to $k$, the path from $i$ to $l$, the path from $j$ to $k$, and the path from $j$ to $l$)\footnote{
    In the proof of this lemma we will do several computations with cross-ratios, for which we use the following notation and convention. For $a,b,c,d\in\bbP^1$, we let $\lambda(a,b,c,d)\in\bbP^1$ denote the point given by the image of $d$ under the unique automorphism of $\bbP^1$ that sends the points $a,b,c$ to $0,1,\infty\in\bbP^1$ respectively. More explicitly: $\lambda(a,b,c,d) = \frac{(b-c)}{(b-a)}\cdot \frac{(d-a)}{(d-c)}$.}.
\end{lem} 
\begin{proof}  
    Observe that if $G(ij|kl)$ is empty, then the order is zero and the sum is also zero. So we have to show the lemma when $G(ij|kl)$ is non-empty. We prove the statement about the order by several reductions. Let $\cC_\Theta = (C_\Theta,x_1,\dots,x_{2(D+1)},\pi_\Theta)\in \overline{M}_\Theta$, and observe that by definition $C_\Theta$ with the marked points is a stable curve of genus $0$. This means that $C_\Theta$ is a tree of $\bbP^1$'s and the set of edges of $G(ij|kl)$ determines a set of nodes of $C_\Theta$. We remark that $G(ij|kl)$ actually defines a path (of the source graph), and we order this set of edges so that the induced orientation is directed from $i$ or $j$ to $k$ or $l$. In terms of $C_\Theta$, the set $G(ij|kl)$ determines a chain of projective lines $C_1,\dots,C_N$ such that
    \begin{itemize}
        \item $C_i$ and $C_{i+1}$ intersect at a node that has weight $d_i$,
        \item $C_i\cap C_j$ is non-trivial if and only if $j = i\pm 1$,
        \item if $C_i$ maps to $Y_1$ (respectively $Y_2$), then $C_{i+1}$ maps to $Y_2$ (respectively $Y_1$). This is due to the non-emptiness of $G(ij|kl)$.
    \end{itemize}
    We depict this situation in Figure \ref{fig: cr1 chain of lines}, where we assume that the marked {\color{red}red} points lie over $Q_\alpha$, the {\color{olive}olive} points lie over $Q_\beta$, the {\color{purple}purple} lie over $Q_\gamma$ and the {\color{blue}blue} over $Q_\delta$. 
    \begin{figure}[!h]
        \centering
        \begin{tikzpicture}
            \draw[] (-4,-1)--(-0.5,2.5);
            \draw[] (-1.5,2.5)--(2,-1);
            \draw[fill=black] (-1,2) circle (3pt);
            \draw[fill=black] (1.5,-0.5) circle (3pt);
            \draw[fill=red] (-3,0) circle (3pt);
            \draw[red,fill=red] (-2.85,0.15) circle (1pt);
            \draw[red,fill=red] (-2.75,0.25) circle (1pt);
            \draw[red,fill=red] (-2.65,0.35) circle (1pt);
            \draw[fill=red] (-2.5,0.5) circle (3pt);
            \draw[fill=olive] (-2,1) circle (3pt);
            \draw[olive, fill=olive] (-1.85,1.15) circle (1pt);
            \draw[olive, fill=olive] (-1.75,1.25) circle (1pt);
            \draw[olive, fill=olive] (-1.65,1.35) circle (1pt);
            \draw[fill=olive] (-1.5,1.5) circle (3pt);
            \draw[fill=purple] (-0.5,1.5) circle (3pt);
            \draw[purple,fill=purple] (-0.35,1.35) circle (1pt);
            \draw[purple,fill=purple] (-0.25,1.25) circle (1pt);
            \draw[purple,fill=purple] (-0.15,1.15) circle (1pt);
            \draw[fill=purple] (0,1) circle (3pt);
            \draw[fill=blue] (0.5,0.5) circle (3pt);
            \draw[blue,fill=blue] (0.65,0.35) circle (1pt);
            \draw[blue,fill=blue] (0.75,0.25) circle (1pt);
            \draw[blue,fill=blue] (0.85,0.15) circle (1pt);
            \draw[fill=blue] (1,0) circle (3pt);

            \draw[] (1,-1)--(2,0);
            \draw[] (3,0)--(4,-1);
            \node at (2.55,0.1) {\Huge$\cdots$};

            \draw[] (3,-1)--(6.5,2.5);
            \draw[] (5.5,2.5)--(9,-1);
            \draw[fill=black] (6,2) circle (3pt);
            \draw[fill=black] (3.5,-0.5) circle (3pt);
            \draw[fill=red] (4,0) circle (3pt);
            \draw[red,fill=red] (4.15,0.15) circle (1pt);
            \draw[red,fill=red] (4.25,0.25) circle (1pt);
            \draw[red,fill=red] (4.35,0.35) circle (1pt);
            \draw[fill=red] (4.5,0.5) circle (3pt);
            \draw[fill=olive] (5,1) circle (3pt);
            \draw[olive, fill=olive] (5.15,1.15) circle (1pt);
            \draw[olive, fill=olive] (5.25,1.25) circle (1pt);
            \draw[olive, fill=olive] (5.35,1.35) circle (1pt);
            \draw[fill=olive] (5.5,1.5) circle (3pt);
            \draw[fill=purple] (6.5,1.5) circle (3pt);
            \draw[purple,fill=purple] (6.65,1.35) circle (1pt);
            \draw[purple,fill=purple] (6.75,1.25) circle (1pt);
            \draw[purple,fill=purple] (6.85,1.15) circle (1pt);
            \draw[fill=purple] (7,1) circle (3pt);
            \draw[fill=blue] (7.5,0.5) circle (3pt);
            \draw[blue,fill=blue] (7.65,0.35) circle (1pt);
            \draw[blue,fill=blue] (7.75,0.25) circle (1pt);
            \draw[blue,fill=blue] (7.85,0.15) circle (1pt);
            \draw[fill=blue] (8,0) circle (3pt);
            
        \end{tikzpicture}
        \caption{Depiction of our situation with a chain of projective lines determined by the edges $G(ij|kl)$}
        \label{fig: cr1 chain of lines}
    \end{figure}
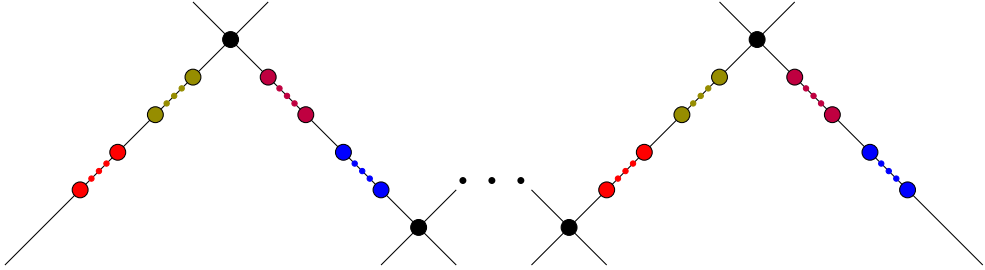
    We apply the usual multiplicative relations of cross ratios several times to understand how each node of this chain contributes to the cross ratio $\lambda(i,j,k,l)$.
    \begin{enumerate}
        \item We can assume without loss of generality that:
        \begin{itemize}
            \item The points $i$, $j$ and the node determined by the first edge of $G(ij|kl)$ lie on the same component. 
            \item The points $k$, $l$ and the node determined by the last edge of $G(ij|kl)$ lie on the same component.
        \end{itemize}
        Indeed, suppose $m$ is an additional point of $C_\Theta$, then
        \begin{equation}
            \lambda(i,j,k,l) = \lambda(i,j,k,m)\cdot \lambda(i,m,k,l). \label{eq: mult rels cr}
        \end{equation}
        Now, if we assume that 
        \begin{itemize}
            \item $i$ and $m$ (or $j$ and $m$) lie in the same component as the node given by the first edge of $G(ij|kl)$, 
            \item $j$ (or $i$) lies in a different component that is connected to this component at a different node, 
        \end{itemize}
        then $\lambda(i,j,k,m)$ is a unit.         
    \item The cross ratio $\lambda(i,j,k,l)$ can be expressed as 
    \begin{equation}
        \lambda(i,j,k,l) = \left(\textnormal{unit}\right)\cdot \prod_{s=1}^{N-1}\lambda(x_s,y_s,x_{s+1},y_{s+1}),\label{eq: decompostion of cross ratio}
    \end{equation}
    where $x_1=i$, $y_1= j$, $x_N=k$ and $y_N=l$, and for each $1\leq s\leq N-1$ the points $x_s$ and $y_s$ are distinct marked points of the component $C_s$.
    \\
    To see this claim suppose that $m$ and $n$ are two additional distinct points of $C_\Theta$. From \eqref{eq: mult rels cr} and making use of the multiplicative relations of cross ratios, it follows that
    \begin{equation}
    \lambda(i,j,k,l) = \lambda(i,j,m,n)\cdot \lambda(i,n,m,l)\cdot \lambda(j,m,n,k)\cdot \lambda(m,n,k,l), \label{eq: mult cross ratios 1}
    \end{equation}
    We observe that if 
    \begin{itemize}
        \item $i$ and $j$ lie in the same component,
        \item $m$ and $n$ lie in the same component different from the previous one,
        \item $k$ and $l$ lie in the same component different from the previous two,
    \end{itemize}
    then both $\lambda(j,m,n,k)$ and $\lambda(l,m,n,i)$ are units. Now, we can just fix $x_1=i$, $y_1=j$, $x_N=k$ and $y_N=l$, and proceed inductively to define the remaining points. Namely, after $x_i$ and $y_i$ have been defined, we just let $x_{i+1}$ and $y_{i+1}$ be distinct marked points of the ensuing component.

\item Following the notation of the previous item, we can further assume in \eqref{eq: decompostion of cross ratio} that for $1\leq s\leq N$ the points $x_s$ and $y_s$ lie over different fibers of the marked points. \\
Indeed, if $x^\prime$ denotes an additional point in the component $C_s$, then (without loss of generality $s\leq N-1$):
    \begin{equation}
        \lambda(x_s,y_s,x_{s+1},y_{s+1}) = \lambda(x^\prime,y_s,x_{s+1},y_{s+1})\cdot \lambda(x_s,y_s,x^\prime,y_{s+1}). \label{eq: mult cross ratios 2}
    \end{equation}
We observe that $\lambda(x_s,y_s,x^\prime,y_{s+1})$ is a unit, and hence our claim (see Figure \ref{fig: cr2 chain of lines multiplicative same component} for a depiction of this situation).
\begin{figure}[!ht]
        \centering
        \begin{tikzpicture}
            \draw[] (-7,-2)--(-2,0.5);
            \draw[] (-4,0.5)--(1,-2);
            
            \draw[fill=black] (-6,-1.5) circle (3pt);
            \draw[fill=black] (-5,-1) circle (3pt);
            \draw[fill=black] (-4,-0.5) circle (3pt);
            \draw[fill=black] (-3,0) circle (3pt);
            \draw[fill=black] (-2,-0.5) circle (3pt);
            \draw[fill=black] (-1,-1) circle (3pt);
            \node at (-6,-1) {$x^\prime$};
            \node at (-5,-0.5) {$x_s$};
            \node at (-4,0) {$y_s$};
            \node at (-2,0) {$x_{s+1}$};
            \node at (-1,-0.5) {$y_{s+1}$};
        \end{tikzpicture}
        \caption{Situation of the third item.}
        \label{fig: cr2 chain of lines multiplicative same component}
\end{figure}
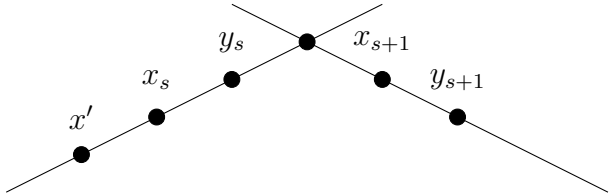
\end{enumerate} 

In conclusion, we are interested then in the contributions to the order of vanishing that come from the nodes because of \eqref{eq: decompostion of cross ratio} and the last item. Hence, we restrict ourselves to compute the order of vanishing of $\lambda(i,j,k,l)$ where 
\begin{itemize}
    \item the points $i$ and $j$ lie on the same component but on different fibers,
    \item the points $k$ and $l$ lie on a different component and on different fibers.
\end{itemize}

\noindent Now, we apply Proposition 1.1 of \cite{Vakil} (or 3.7 in the published version) showing that deformations are local around special loci, so that we can content ourselves with the following very special case: Let $\alpha = (d)$, $\beta=(1^d)$, $\gamma = (d)$ and $\delta = (1^d)$, consider $\bbP^1$ with the four marked points $0$, $1$, $\infty$, and $t$. We look at the family of smooth covers $(C,x_{t1},\dots,x_{td},x_\infty,x_0,x_{11},\dots,x_{1d}, \pi )$ of $\bbP^1$ of degree $D$ satisfying:
    \begin{align*}
       && \pi^* 0= d\cdot x_0, && \pi^*1 = \sum_i x_{1i},
        &&\pi^*\infty = d\cdot x_\infty, && \pi^*t= \sum_i x_{ti}.
    \end{align*}
On the source curve we set $x_0 = 0$ and $x_\infty= \infty$, and parameterize this family by $\pi ([z_0:z_1]) = [\lambda z_0^d: z_1^d]$, with  $\lambda\in \bbC^*$. This makes $\{x_{1s}\}$ be the $d$th-roots of $\lambda$ and $\{x_{ts}\}$ be the $d$th roots of $\lambda\cdot t$. Our cover of interest arises when $t\mapsto 0$, and we have to compute the order of vanishing of the cross-ratio $\lambda(x_0,x_{1q},x_\infty,x_{tp})$ for arbitrary $p$ and $q$. But we can readily see that $\lambda(x_0,x_{1q},x_\infty,x_{tp})=\frac{x_{tp}}{x_{1q}}$, which is a $d$th root of $t$, and therefore it vanishes at $t=0$ with order $\frac{1}{d}$.\\

To finish the lemma we look at the following form (Corollary 3.13 of \cite{JunLiDF}) of the degeneration formula \eqref{eq: degeneration formula}
    \begin{equation}
    [\mathfrak{M}_{0,2(d+1)}(\mathfrak{W}_0,d)]^{\virt}
    =\sum_{\eta\in\overline{\Omega}}m(\eta)[\mathfrak{M}(\mathfrak{Y}_1^{\rel}\sqcup \mathfrak{Y}_2^{\rel},\eta)]^{\virt},\label{eq: degeneration formula2}
\end{equation}
There are multiple $\eta\in\overline{\Omega}$ that contribute to $\ord_{\cC_\Theta}\ft_{I^c}^*(ij|kl)$. Namely, the different $\eta\in\overline{\Omega}$ that glue to this curve. For any of these, the factor $m(\eta)$ coincides (by definition) with $m(\eta) = \prod_{e\in E(G)}d_\pi(e)$, and the above procedure shows that each $\eta$ glueing to this curve contributes
    \begin{equation}
        m(\eta)\sum_{e\in G(ij|kl)} \frac{1}{d_\pi(e)}.\label{mults}
    \end{equation}
On the other hand, the number of $\eta$'s glueing to $\cC_\Theta$ is the product $\prod_{V\in V(G)}\mathrm{CF}(V)$, because we have to account the reordering possibilities of the legs (with their weights) of the admissible graphs of $\eta$ and the (non-simultaneous) permutation of the roots that have the same weight.
\end{proof}

\begin{remark}
As we have previously mentioned, the alternative approach with the log geometric version of the degeneration formula \cite{KimLhoRuddat} can be used to arrive directly at \eqref{mults} and bypasses some of our technicalities in the previous proof.
\end{remark}

\begin{thm}\label{thm: balancing of Htrop}
    Let $J\subset\{1,\dots,2(d+1)\}$. The standard weight $\varpi^J_{d,0,(\alpha,\beta,\gamma,\delta)}$ is a Minkowski weight.
\end{thm}
\begin{proof}
   We first do the case of $J=\{1,\dots,2(d+1)\}$, and use it to establish the others. We first focus on the weight $\varpi_{d,0,(\alpha,\beta,\gamma,\delta)}$ and show that \eqref{eq: local balancing} holds. This is an equation in the space $\cM^{\trop}_{0,2(d+1)}\times\cM^{\trop}_{0,4}$. We show it in the $\cM_{0,4}^\trop$-coordinates (target) and the $\cM^\trop_{0,2(d+1)}$-coordinates (source).
   \begin{itemize}
       \item For the $\cM_{0,4}^\trop$-coordinates, let $\mathrm{CF}(\alpha,\beta,\gamma,\delta)$ denote the combinatorial factor of the vertex of an initial object of $\bbEAC_{d,h}(\alpha,\beta,\gamma,\delta)$. In the underlying vector space of $\cM_{0,4}^\trop$ the equation $v_{\{1,2\}}+v_{\{1,3\}}+v_{\{1,4\}} = 0$ holds (following the notation of Section \ref{ssection: realization of mod spaces of rat trop curves}), and to show our equation of interest we will make use of a result from \cite{CavalieriMarkwigRanganathan} (as well as their notation). An object $\pi\colon G\to H$ of $\bbEAC_{d,h}(\alpha,\beta,\gamma,\delta)$ corresponds to a combinatorial type of loc. cit. and furthermore (following the notation of loc. cit.) 
    \begin{equation}
        \#\Aut_0(\pi) = \frac{\mathrm{CF}(\alpha,\beta,\gamma,\delta)}{\prod_{V\in V(G)}\mathrm{CF}(V)}.\label{eq: aut0}
    \end{equation}
    In the sum
    \begin{equation}
        \sum_{\Theta\in[\bbEAC_{d,h}(\alpha,\beta,\gamma,\delta)](1)} \varpi_{d,0,(\alpha,\beta,\gamma,\delta)}(\Theta)\cdot \trgt(u_\Theta),\label{eq: sum of weights}
    \end{equation}
    the vector $v_{\{1,2\}}$ appears with factor given by the weighted sum
    \begin{equation}
        \sum \varpi_{d,0,(\alpha,\beta,\gamma,\delta)}(\Theta),\label{eq: teorema 2 tsac}
    \end{equation}
    where the index $\Theta$ runs over the isomorphism classes $[\bbEAC_{d,h}(\alpha,\beta,\gamma,\delta)](1)$ where the legs of the target tree with ramification $\alpha$ and $\beta$ share a common vertex. If we divide \eqref{eq: teorema 2 tsac} by $ \mathrm{CF}(\alpha,\beta,\gamma,\delta)$, then because of \eqref{eq: aut0} it follows from Theorem 2 of loc. cit. and its proof (Section 5.4.1), that we obtain $H_{0\to 0}(\alpha,\beta,\gamma,\delta)$ (the corresponding rational quadruple Hurwitz number). In orther words, the vector $v_{\{1,2\}}$ appears in the sum \eqref{eq: sum of weights} with factor $H_{0\to 0}(\alpha,\beta,\gamma,\delta)\cdot \mathrm{CF}(\alpha,\beta,\gamma,\delta)$ 
    Analogously, the vector $v_{\{1,3\}}$ appears with factor $H_{0\to 0}(\alpha,\gamma,\beta,\delta)\cdot \mathrm{CF}(\alpha,\beta,\gamma,\delta)$, and the vector $v_{\{1,4\}}$ appears with factor $H_{0\to 0}(\alpha,\delta,\beta,\gamma)\cdot \mathrm{CF}(\alpha,\beta,\gamma,\delta)$. Since these quadruple Hurwitz number coincide, it follows that
    \begin{equation*}
        \sum_{\Theta\in[\bbEAC_{d,h}(\alpha,\beta,\gamma,\delta)](1)}\hspace{-0.2cm} \varpi_{d,0,(\alpha,\beta,\gamma,\delta)}(\Theta)\cdot \trgt(u_\Theta) =H_{0\to0}(\alpha,\beta,\gamma,\delta)\cdot (v_{\{1,2\}}+v_{\{1,3\}}+v_{\{1,4\}})=0.
    \end{equation*}
    \item For the $\cM^\trop_{0,2(d+1)}$-coordinates, we will use Lemma \ref{lem: zero vectors in QN}. Let $I=\{i,j,k,l\}$ be an arbitrary subset of $\{0,\dots,2(d+1)\}$, and consider a triple $(T_s,T_t,\pi)$ of $\bbEAC_{d,h}(\alpha,\beta,\gamma,\delta)$ with isomorphism class $\Theta\in[\bbEAC_{d,h}(\alpha,\beta,\gamma,\delta)](1)$. An edge $e\in E(G)$ gives a partition of the legs of $G$ into two sets, say $A$ and $B$, which by definition must each be of size $\geq 2$ and therefore (following the notation of Section \ref{ssection: realization of mod spaces of rat trop curves}) $v_A=v_B$. We let $v_e$ denote this common vector. The definition of the associated matrix to $\pi$ \eqref{eq: assoc matrix} implies that 
    \begin{equation*}
        \varpi_{d,0,(\alpha,\beta,\gamma,\delta)}(\Theta)\cdot \src(u_\Theta) = \sum_{e\in E(G)}\left(\prod_{V\in V(G)}\mathrm{CF}(V)\cdot \prod_{h\in E(G)}d_\pi(h)\right) \frac{1}{d_\pi(e)}\cdot v_e.
    \end{equation*}
    It follows from Lemma \ref{lem: estimation of the order} that $v_{\{i,j\}}$ appears in the sum
    \begin{equation}
        \sum_{\Theta\in[\bbEAC_{d,h}(\alpha,\beta,\gamma,\delta)](1)} \varpi_{d,0,(\alpha,\beta,\gamma,\delta)}(\Theta)\cdot \ft_{I^c}\src(u_\Theta)\label{eq: sum srcs}
    \end{equation}
    with a factor of $\deg \ft_{I^c}^*(ij|kl)$. This same argument, as well as Lemma \ref{lem: estimation of the order}, hold for the other splittings of $I$, so that the sum \ref{eq: sum srcs} contains the vectors $v_{\{i,k\}}$ and $v_{\{i,l\}}$ with corresponding factors of $\deg \ft^*_{I^c}(ik|jl)$ and $\deg \ft^*_{I^c}(il|jk)$. Since the divisors $(ij|kl)$, $(ik|jl)$, and $(il|jk)$ are linearly equivalent, it follows then that 
    \begin{equation}
        \sum_{\Theta\in[\bbEAC_{d,h}(\alpha,\beta,\gamma,\delta)](1)}\hspace{-1cm} \varpi_{d,0,(\alpha,\beta,\gamma,\delta)}(\Theta)\cdot\ft_{I^c} \src(u_\Theta) = \deg \ft_{I^c}^*(ij|kl)\cdot  \left( v_{\{i,j\}}+v_{\{i,k\}}+v_{\{i,l\}}\right) =0.\label{eq: sum srcs 2}
    \end{equation}
    We apply Lemma \ref{lem: zero vectors in QN} to conclude that
        \begin{equation*}
        \sum_{\Theta\in[\bbEAC_{d,h}(\alpha,\beta,\gamma,\delta)](1)} \varpi_{d,0,(\alpha,\beta,\gamma,\delta)}(\Theta)\cdot \src(u_\Theta)=0.
    \end{equation*}
   \end{itemize}
    For $J\subsetneq \{1,\dots,2(d+1)\}$, we remark that, in this very special case, forgetting the marking is a proper morphism of linear poic-complexes, since the linear poic-complexes that we are dealing with can be considered as $1$-dimensional closed fans in a vector space. Thereofore, taking the pushforward through the respective forgetting the marking morphism gives the desired result, because the lattice indexes that arise are always $1$. 
\end{proof}

We conclude with a proof of Theorem \ref{thm: global balancing}.

\globalbalancing*

\begin{proof}\label{proof: proof of big theorem}
    Equivariance can readily be seen from \eqref{eq: equivariance}, so it is only necessary to show that $\varpi^J_{d,h,\vec{\mu}}$ is, in fact, a Minkowski weight. Similar to the case of Lemma \ref{thm: balancing of Htrop}, we first study the case of $J=\{1,\dots,n\}$, and subsequently explain how to deduce the others from this.\\ 
    Just as in Theorem \ref{thm: balancing of Htrop}, the balancing of $\varpi_{d,h,\vec{\mu}}$ follows from the balancing in both coordinates: $\cM^\trop_{0,m+2h}\times \bbR^h_{>0}$ and $\cM^\trop_{0,n+2g}\times \bbR^g_{>0}$. By default, this balancing just involves the $\cM^{\trop}_{0,m+2h}$- and $\cM^{\trop}_{0,n+2g}$-coordinates. 
    \begin{itemize}
        \item For the balancing along the $\cM^\trop_{0,m+2h}$-coordinates, observe that only the edges of the target graph whose fiber is acyclic are being contracted. Therefore we proceed exactly as in Lemma \ref{thm: balancing of Htrop} with $\cM^\trop_{0,4}$-coordinates but with several components. We omit the argument, since it is analogous to this case.
        \item The balancing along the $\cM^\trop_{0,n+2g}$-coordinates demands a more careful examination. Consider a triple $(T_s,T_t,\pi)$ of  $\bbEAC_{d,h}(\vec{\mu})$ of codimension $1$, and let $G=\st_{g,n}(T_s)$ and $H=\st_{h,m}(T_t)$. The triple being of codimension $1$ means that $H$ has a unique $4$-valent vertex, which we will denote by $W$ and let $\pi^{-1}(W)=\{V_1,\dots,V_k\}\subset V(G)$. Without loss of generality we will refer to the triple simply by $\pi$, and additionally assume that for any isomorphism class $\Theta\in\Star^1(\pi)$ there is given a representative object $(T_{s,\Theta},T_{t,\Theta},\pi_\Theta)$ with $G_\Theta:= \st_{g,n}(T_{s,\Theta})$ and $H_\Theta:= \st_{h,m}(T_{t,\Theta})$. For each $1\leq j\leq k$ and each $T_{s,\Theta}$, we let $T_{\Theta,j}$ denote the connected legless subtree of $G_\Theta$ that contracts to $V_j$.\\
        Just as in the proof of Lemma \ref{thm: balancing of Htrop}, any edge $e\in E(T_{\Theta,j})$ gives a partition of the legs of $G$ into two sets, say $A\sqcup B = L(G)$, which by definition must each be of size $\geq2$ and therefore $v_A=v_B$ (again following the notation of the vectors \ref{ssection: realization of mod spaces of rat trop curves}). We let $v_e$ denote this common vector, and observe that (just as in the previous proof) for $\Theta\in\Star^1(\pi)$ we have
    \begin{equation}
        \varpi_{d,h,\vec{\mu}}(\Theta)\src(u_\Theta) = K\cdot \sum_{t=1}^k\sum_{e\in E(T_{\Theta,t})}\left(\prod_{V\in V(T_{\Theta,t})}\mathrm{CF}(V)\cdot \prod_{h\in E(T_{\Theta,t})}d_\pi(h)\right)\frac{1}{d_\pi(e)}v_e, \label{eq: product of the weight with direction vector}
    \end{equation}
    where $K$ is a factor coming from $\varpi_{d,h,\vec{\mu}}(\Theta)$ but common along all the elements of $\Star^1(\pi)$ as it pertains to the vertices and edges not in the subtrees $T_{\Theta,j}$ (for $1\leq t\leq k)$. Therefore, to show balancing it is sufficient to show that for each $1\leq t\leq k$, the following equation holds in the corresponding quotient space
    \begin{equation}
        \sum_{\Theta\in\Star^1(\pi)}\sum_{e\in E(T_{\Theta,t})}\left(\prod_{V\in V(T_{\Theta,t})}\mathrm{CF}(V)\cdot \prod_{h\in E(T_{\Theta,t})}d_\pi(h)\right)\frac{1}{d_\pi(e)}v_e = 0.\label{eq: j-th block}
    \end{equation}
    As before, we seek to apply Lemma \ref{lem: zero vectors in QN}. For this, let $I=\{i,j,k,l\}\subset \{1,\dots,n\}$ be arbitrary. We observe that we can turn each $T_{\Theta,t}$ to the case of Theorem \ref{thm: balancing of Htrop} by letting the other edges and legs of $T_{s,\Theta}$ that are adjacent to vertices of $T_{\Theta,t}$ be legs and keeping their corresponding weights given by the cover $\pi_\Theta$. Then after applying $\ft_{I^c}$ to \eqref{eq: j-th block}, we arrive at the same equation \eqref{eq: sum srcs 2} and, in particular, its vanishing. Since $I$ and $1\leq t\leq k$ were arbitrary, we obtain the equation \eqref{eq: j-th block}. Hence, $\varpi_{d,h,\vec{\mu}}$ is an equivariant $\est_{d,h,\vec{\mu}}$-Minkowski weight.
    \end{itemize}
    To finalize the proof of this theorem, we consider an arbitrary $J\subsetneq\{1,\dots,n\}$ and explain the balancing of $\varpi^J_{d,h,\vec{\mu}}$. As before, we must consider the $\cM^{\trop}_{0,2h+m}$- and $\cM^{\trop}_{0,J\sqcup \mathbbm{g}}$-coordinates. The balancing along the first coordinates is identical to the previous situation (namely, the same as in Lemma \ref{thm: balancing of Htrop} with $\cM^\trop_{0,4}$-coordinates), so we focus on balancing along the $\cM^{\trop}_{0,J\sqcup\mathbbm{g}}$-coordinates. We remark that the forgetful morphism $\ft_{J^c}$ yields a morphism of poic-complexes $\EST_{d,h}(\vec{\mu})\to \EST_{d,h}^J(\vec{\mu})$, and at isomorphism classes we have the equality $\varpi^J_{d,h,\vec{\mu}}\circ\ft_{J^c}=\varpi_{d,h,\vec{\mu}}$. We do observe that this morphism is neither necessarily proper nor weakly-proper (in any dimension). The idea is to proceed in exactly the same manner as before, keeping track of the edges and legs that are being forgotten.\\
    Consider an object $(T_s,T_t,\pi)$ of  $\bbEAC^J_{d,h}(\vec{\mu})$ of codimension $1$. The triple being of codimension $1$ means that $\trgt(\pi)$ has a unique $4$-valent vertex, say $W$, and let $\pi^{-1}(W)=\{V_1,\dots,V_k\}\subset V(\src(\pi))$. In contrast to the previous case ($J$ being the whole set $\{1,\dots,n\}$), we do not have to consider the deformation around every $V_i$, but only around those that after applying $\ft_{J^c}$ to $\src(\pi)$ lie on an actual edge of $T_s$. Let us assume without loss of generality that these are all the $V_i$ with $i\leq k^\prime$ for a $k^\prime\leq k$.\\
    Just as before, we simply refer to the triple $(T_s,T_t,\pi)$ by $\pi$ and assume that for any isomorphism class $\Theta\in\Star^1(\pi)$ there is given a representative object $(T_{s,\Theta},T_{t,\Theta},\pi_\Theta)$. For each $1\leq j\leq k^\prime$ and each $T_{s,\Theta}$, we let $T_{\Theta,j}$ denote the connected legless subtree of $\src(\pi_\Theta)$ that contracts to $V_j$. Proceeding in the exact same way as above, for $\Theta\in\Star^1(\pi)$ we obtain the following analog of \eqref{eq: product of the weight with direction vector} in this situation 
    \begin{equation}
        \varpi^J_{d,h,\vec{\mu}}(\Theta)\src(u_\Theta) = K\cdot \sum_{t=1}^{k^\prime}\sum_{e\in E(T_{\Theta,t})}\left(\prod_{V\in V(T_{\Theta,t})}\mathrm{CF}(V)\cdot \prod_{h\in E(T_{\Theta,t})}d_\pi(h)\right)\frac{1}{d_\pi(e)}\ft_{J^c}v_e, \label{eq: product of the weight with direction vector J}
    \end{equation}
    where once again $K$ is a factor coming from $\varpi_{d,h,\vec{\mu}}^J(\Theta)$ but common along all the elements of $\Star^1(\pi)$ as it pertains to the vertices and edges not in the subtrees $T_{\Theta,j}$ (for $1\leq t\leq k^\prime)$. At the risk of being too repetitive, we observe, just as in the previous case, that for balancing it is sufficient to show that for each $1\leq t\leq k^\prime$, the following equation holds in the corresponding quotient space
    \begin{equation}
        \sum_{\Theta\in\Star^1(\pi)}\sum_{e\in E(T_{\Theta,t})}\left(\prod_{V\in V(T_{\Theta,t})}\mathrm{CF}(V)\cdot \prod_{h\in E(T_{\Theta,t})}d_\pi(h)\right)\frac{1}{d_\pi(e)}\ft_{J^c}v_e = 0.\label{eq: j-th block J}
    \end{equation}
    But this equation readily follows from what we have already argued after \eqref{eq: j-th block}. More precisely, for $I=\{i,j,k,l\}\subset\{1,\dots,n\}$ we have that $\ft_{I^c}\circ \ft_{J^c} = \ft_{I^c\cup J^c}$, then the only non-trivial case that arises is when $I\subset J$ and in this case $\ft_{I^c}\circ\ft_{J^c} =\ft_{I^c}$. Thus, it follows from the vanishing of \eqref{eq: j-th block J}, that $\varpi^J_{d,h,\vec{\mu}}$ is a $\est^J_{d,h,\vec{\mu}}$-equivariant Minkowski weight.
\end{proof}

\subsection{Covers of trees}

In this section we apply our results to the case of discrete admissible covers of trees with simple ramification over the legs. More precisely, we consider the following set-up: Let $g,r\geq0$ be integers with $g +r \equiv 0 \mod 2$, and let
\begin{equation}
\begin{aligned}
&h =0,
&&m = 3\cdot g+r,&&d = \frac{g+r}{2}+1,\\
&\vec{\mu} = (\mu_1,\dots,\mu_m),\textnormal{ with}\hspace{-0.5em}&&\mu_i=(2,1^{d-2}),&&\text{for all }1\leq i\leq m,\\
&n=(d-1)\cdot3g=\frac{3}{2}(g+r)g,\hspace{-4em}&&&&J=\{(i-1)(d-1)+1\}_{i=1}^r.
\end{aligned}
\label{eq: big align3}
\end{equation}
We can readily observe that $\EST_{d,h,\vec{\mu}}^J$ is of dimension $m-3 = 3g+r-3$, which is equal to the dimension of $\ST_{g,J}$. Since $\mathfrak{src}$ is weakly proper in top dimension, we can pushforward the $\est^J_{d,h,\vec{\mu}}$-equivariant Minkowski weight $\varpi^J_{d,h,\vec{\mu}}$ through $\mathfrak{src}$ (as in \eqref{eq: pushforward minkowkski weights of fibration}) to produce a top dimensional $\st_{g,J}$-equivariant tropical cycle. It follows from Proposition 4.5.8 of \cite{DARB} that $\st_{g,J}$ is an irreducible poic-fibration, which implies that this pushforward is an scalar multiple of the fundamental cycle. It turns out that it is a non-trivial integral multiple of $C_{\frac{g+r}{2}}$, the $\left(\frac{g+r}{2}\right)$th Catalan number. We first bring forward some combinatorial notations and definitions, as well as the statement of a (combinatorial) lemma whose proof we postpone to the end of the section. Subsequently, we apply our machinery and compute the aforementioned pushforward. We consider the poic-fibration $\est^J_{d,h,\vec{\mu}}$. It might be evident from \eqref{eq: big align3} that this choice of $J$ is motivated mainly by $\st_{g,J}$. More precisely, for an object $\pi$ of $\bbAC_{d,h}(\vec{\mu})$, we want to keep the first $r$-marked legs of $\src(\pi)$ that have weight $2$.
\begin{defi}
    Suppose $T$ is a discrete graph and let $V\in V(T)$. The \emph{leg valency} of $V$ is defined as $\legval(V):=\#\{\ell\in L(T): r_T(\ell) = V\}$. We say that an edge $e\in E(T)$ is a \emph{leaf} if there is $V\in\partial e$ with $\val(V)=3$ and $\legval(V)=2$.
\end{defi}

The following lemma is basically Lemma 79 of \cite{VargasDraisma}. We postpone its proof to the end of the section. The structure of the discrete admissible cover described in Lemma \ref{lem: loopandbridge} is depicted in Figure \ref{fig: lemloopandbridge}, where the legs are color coded to depict the correspondence of the cover $\pi$. 
\begin{restatable}{lem}{loopandbridge}\label{lem: loopandbridge}
    Following the notation of \eqref{eq: big align3}, let $\pi$ be an object of $\bbAC_{d,h}(\vec{\mu})$ of top dimension and such that the composition $\left(\eta_{\ft_{J^c},\src(\pi)}\circ F_\pi\right)$ is invertible, where $F_\pi$ is as in \eqref{eq: assoc matrix}. If $W\in V(\src(\pi))$ is such that 
    \begin{itemize}
        \item $W$ is also a vertex of $\ft_J(\src(\pi))$,
        \item $\val_{\ft_{J^c}(\src(\pi))}(W)=3$,
        \item $W$ is incident to a loop and an edge in $\ft_{J^c}(\src(\pi))$,
    \end{itemize}
    then $\pi(W)$ is incident to a leaf of $\trgt(\pi)$ and $\legval(\pi(W))=1$. Furthermore, if $e_1$ denotes the leaf of $\trgt(\pi)$ that $\pi(W)$ is incident to, $V$ denotes the other vertex of $\trgt(\pi)$ incident to $e_1$, and $e_2$ denotes the edge of $\trgt(\pi)$ that is not incident to $\pi(W)$, then we have the following:
    \begin{itemize}
        \item The ramification profile above $e_1$ is $(1^d)$.
        \item The ramification profile above $V$ is $(2,1^{d-2})$.
        \item The ramification profile above $e_2$ is $(2,1^{d-2})$.
    \end{itemize} 
\end{restatable}

\begin{figure}[!ht]
    \centering
    \begin{tikzpicture}
        
        \draw[line width = 2pt] (-4.75,4) ellipse [x radius = 1.25, y radius=0.75];
        \draw[line width =3pt] (-3.5,4)--(-1.5,4) node [midway, above] {$2$};
        
        \draw[olive,line width =3pt, dashed] (-6,4)--(-7,5) node [midway, above] {$2$};
        \draw[purple, line width =3pt, dashed] (-3.5,4)--(-5,3) node [midway, above] {$2$};
        
        \draw[line width =2pt] (-6,2.5)--(-1.5,2.5);
        \draw[line width =2pt] (-6,1)--(-1.5,1);
        \draw[line width =2pt] (-6,-1)--(-1.5,-1);
        
        \draw[olive,line width =2pt, dashed] (-6,2.5)--(-7,3.5);
        \draw[teal,line width =3pt, dashed] (-6,4)--(-7.5,3) node [midway, above] {$2$};
        \draw[purple, line width =2pt, dashed] (-3.5,2.5)--(-5,1.5);
        
        \draw[teal,line width =2pt, dashed] (-6,1)--(-7.5,0);
        \draw[olive,line width =2pt, dashed] (-6,1)--(-7,2);
        \draw[purple, line width =2pt, dashed] (-3.5,1)--(-5,0);
        
        \draw[teal,line width =2pt, dashed] (-6,2.5)--(-7.5,1.5);

        \node at (-4.75,0) {\Huge$\vdots$};
        
        \draw[teal,line width =2pt, dashed] (-6,-1)--(-7.5,-2);
        \draw[olive,line width =2pt, dashed] (-6,-1)--(-7,0);
        \draw[purple, line width =2pt, dashed] (-3.5,-1)--(-5,-2);
        
        \node at (-0.75,4) {\Huge$\dots$};
        \node at (-0.75,2.5) {\Huge$\dots$};
        \node at (-0.75,1) {\Huge$\dots$};
        \node at (-0.75,-1) {\Huge$\dots$};

        \draw[fill=black] (-1.5,4) circle (3pt);
        \draw[fill=black] (-1.5,2.5) circle (3pt);
        \draw[fill=black] (-1.5,1) circle (3pt);
        \draw[fill=black] (-1.5,-1) circle (3pt);
        
        \draw[fill=black] (-3.5,4) circle (3pt);
        \draw[fill=black] (-3.5,2.5) circle (3pt);
        \draw[fill=black] (-3.5,1) circle (3pt);
        \draw[fill=black] (-3.5,-1) circle (3pt);

        \node at (7.25,1.5) {\Huge$\dots$};
        \draw[line width = 2pt] (2,1.5)--(4.5,1.5) node [midway, above] {$e_1$};
        \draw[line width = 2pt] (4.5,1.5)--(6.5,1.5) node [near end, above] {$e_2$};
        \draw[purple, line width = 2pt, dashed ] (4.5,1.5)--(3,0.5);
        \draw[teal, line width = 2pt, dashed ] (2,1.5)--(0.5,0.5);
        \draw[olive, line width = 2pt, dashed ] (2,1.5)--(1,2.5);
        \draw[fill=black] (2,1.5) circle (3pt);
        \draw[fill=black] (4.5,1.5) circle (3pt);
        \draw (0,6)--(8,6) node [midway, above] {$\trgt(\pi)$};
        \draw (-8,6)--(0,6) node [midway, above] {$\src(\pi)$};
        \draw (0,6.7)--(0,-3);
        \draw (-8,6.7)--(-8,-3)--(8,-3)--(8,6.7)--(-8,6.7);
        
        \node at (2,2) {$V$};
        \node at (4.5,2) {$\pi(W)$};
        \node at (-3.4,4.5) {$W$};
        
    \end{tikzpicture}
    \caption{The cover $\pi$ of Lemma \ref{lem: loopandbridge}}
    \label{fig: lemloopandbridge}
\end{figure}
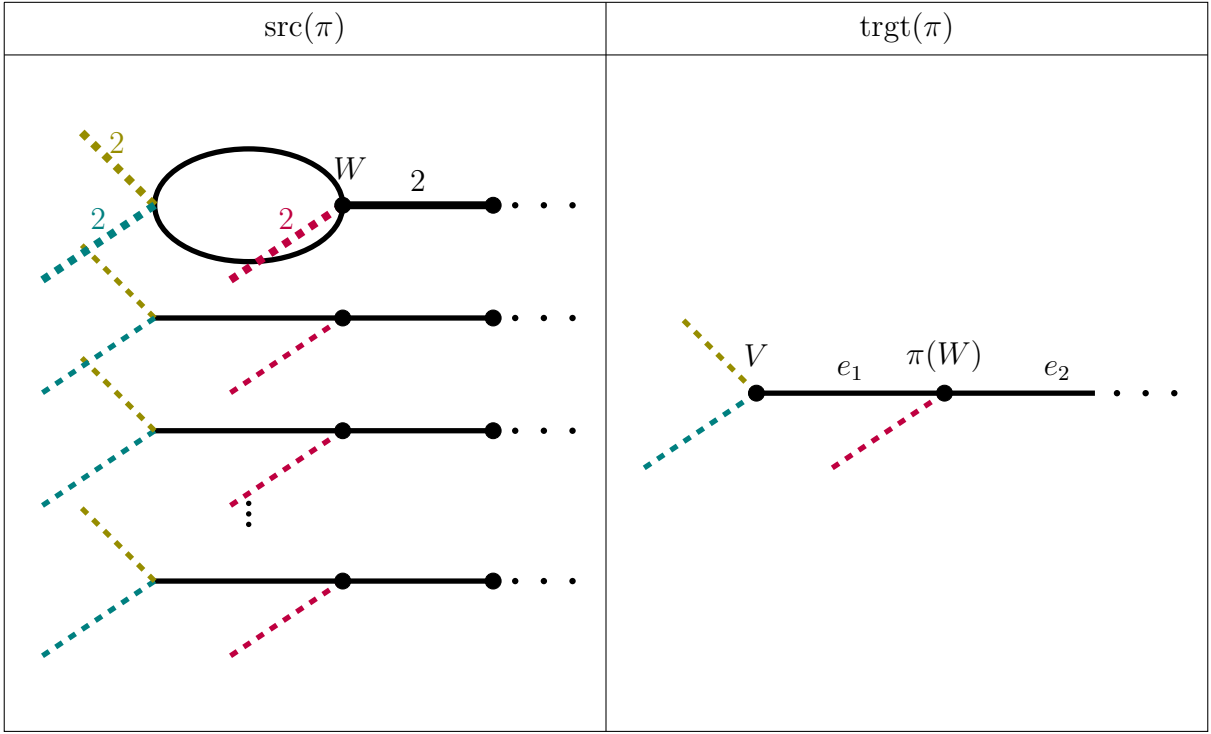

\begin{nota}\label{nota: notacion J}
    For the sake of simplicity, when $r>0$ we set $j_i = (i-1)(d-1)+1$ for $1\leq i \leq r$. With this notation we have that $J=\{j_1<j_2<\dots j_r\}$.
\end{nota}

\begin{defi}
    Suppose $r>0$ in the notation of \eqref{eq: big align3}. The genus-$g$ $J$-marked discrete graph $O_{g,J}$ (we illustrate $O_{g,J}$ for $g=4$ and $r=2$ in Figure \ref{fig: ogj}) is the object of $\bbG_{g,J}$ given by the following discrete graph:
    \begin{itemize}
        \item It has vertices $V_1$, $\dots$ ,$V_g$, $W_0$, $\dots$, $W_{g-3}$, and $C_1$, $\dots$, $C_r$. If $g=1$, we set $C_{r-1}=C_r$.
        \item There are $g$ loops $l_i$, for $1\leq i\leq g$, with $\partial l_i = \{V_i\}$.
        \item There are edges $e_0$, $\dots$, $e_{g-4}$ with $\partial e_i = \{W_i,W_{i+1}\}$ for $0\leq i\leq g-4$.
        \item There is an edge $h_1^\prime$ with $\partial h_1^\prime= \{V_1,C_1\}$. If $g>2$, there are edges $h_1$, $h_2$, $\dots$, $h_g$ with $\partial h_1^\prime = \{V_1,C_1\}$, $\partial h_1=\{C_r,W_0\}$, $\partial h_i = \{V_i,W_{i-2}\}$ for $2\leq i\leq g-1$, and $\partial h_g = \{V_g,W_{g-3}\}$. If $g=2$, there is an edge $h_1$ with $\partial h_1 = \{C_r, V_2\}$.
        \item There are edges $c_1$, $\dots$, $c_{r-2}$ with $\partial c_i = \{C_i,C_{i+1}\}$ for $0< i< r-1$. If $g>1$, then there is an additional edge $c_{r-1}$ with $\partial c_{r-1} = \{C_{r-1},C_r\}$.
        \item There are legs\footnote{We apologize for notational inconsistency. The legs of a graph have always been denoted by $\ell$, but this is quite similar to the way we are denoting loops. So in the spirit of clarity, we decided to denote the legs in these very special cases by $L$.} $L_{j_1}$, $\dots$, $L_{j_r}$ with $\partial L_{j_i} = C_i$ for $1\leq i\leq r$.
    \end{itemize}
\end{defi}
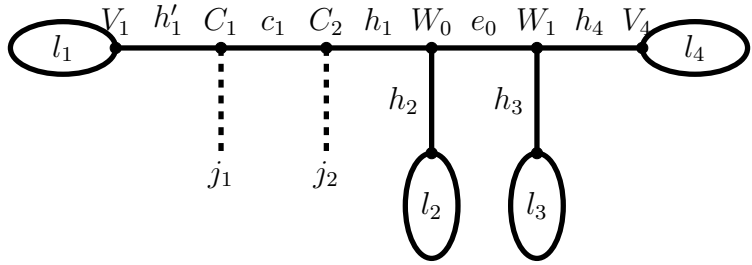
\begin{figure}[ht!]
        \centering
        \begin{tikzpicture}[scale=0.7]
            \draw[line width = 2pt] (-6,0) ellipse [x radius = 1 , y radius = 0.5];
            \draw[line width = 2pt] (-5,0)--(-3,0) node [midway, above]{$h_1^\prime$};
            \draw[line width = 2pt] (-3,0)--(-1,0) node [midway, above]{$c_1$};
            \draw[line width = 2pt] (-1,0)--(1,0) node [midway, above]{$h_1$};
            \draw[line width = 2pt] (1,0)--(3,0) node [midway, above]{$e_0$};
            \draw[line width = 2pt] (3,0)--(5,0) node [midway, above]{$h_4$};
            \draw[line width =2pt, dashed] (-3,0)--(-3,-2) node[pos=1.2] {$j_1$};
            \draw[line width =2pt, dashed] (-1,0)--(-1,-2) node[pos = 1.2] {$j_2$};
            \draw[line width =2pt] (1,0)--(1,-2) node[midway, left] {$h_2$};
            \draw[line width = 2pt] (1,-3) ellipse [y radius = 1 , x radius = 0.5];
            \draw[line width =2pt] (3,0)--(3,-2) node[midway, left] {$h_3$};
            \draw[line width = 2pt] (3,-3) ellipse [y radius = 1 , x radius = 0.5];
            \draw[line width = 2pt] (6,0) ellipse [x radius = 1 , y radius = 0.5];
            \draw[fill = black] (-5,0) circle (3pt);
            \draw[fill = black] (-3,0) circle (3pt);
            \draw[fill = black] (-1,0) circle (3pt);
            \draw[fill = black] (1,0) circle (3pt);
            \draw[fill = black] (3,0) circle (3pt);
            \draw[fill = black] (1,-2) circle (3pt);
            \draw[fill = black] (3,-2) circle (3pt);
            \draw[fill = black] (5,0) circle (3pt);
            
            \node at (-5,0.5) {$V_1$};
            \node at (-3,0.5) {$C_1$};
            \node at (-1,0.5) {$C_2$};
            \node at (1,0.5) {$W_0$};
            \node at (3,0.5) {$W_1$};
            \node at (4.9,0.5) {$V_4$};
            \node at (-6,0) {$l_1$};
            \node at (1,-3) {$l_2$};
            \node at (3,-3) {$l_3$};
            \node at (6,0) {$l_4$};
        \end{tikzpicture}
        \caption{Depiction of $O_{g,J}$ for $g=4$ and $r=2$}
        \label{fig: ogj}
\end{figure}

\begin{defi}
 Suppose $r=0$ in the notation of \eqref{eq: big align3}, so that $J=\varnothing$. The genus-$g$ discrete graph $O_{g}$ (we illustrate $O_{g}$ for $g=6$ in Figure \ref{fig: og}) is the object of $\bbG_{g,0}$ given by the following discrete graph:
        \begin{itemize}
        \item It has vertices $V_1,\dots,V_g$ and $W_0,\dots,W_{g-3}$.
        \item There are $g$ loops $l_i$, $1\leq i\leq g$, with $\partial l_i = \{V_i\}$.
        \item There are edges $e_0,\dots, e_{g-4}$ with $\partial e_i = \{W_i,W_{i+1}\}$ for $0\leq i\leq g-4$.
        \item There are edges $h_1,\dots, h_g$, with $\partial h_1=\{W_0,V_1\}$, $\partial h_i =\{V_i,W_{i-2}\}$ for $2\leq i\leq g-1$, and $\partial h_g = \{V_g,W_{g-3}\}$.
    \end{itemize}
\end{defi}
\begin{figure}[ht!]
        \centering
        \begin{tikzpicture}[scale = 0.7]
            \draw[line width = 2pt] (-6,0) ellipse [x radius = 1 , y radius = 0.5];
            \draw[line width = 2pt] (-5,0)--(-3,0) node [midway, above]{$h_1$};
            \draw[line width = 2pt] (-3,0)--(-1,0) node [midway, above]{$e_0$};
            \draw[line width = 2pt] (-1,0)--(1,0) node [midway, above]{$e_1$};
            \draw[line width = 2pt] (1,0)--(3,0) node [midway, above]{$e_2$};
            \draw[line width = 2pt] (3,0)--(5,0) node [midway, above]{$h_6$};
            \draw[line width =2pt] (-3,0)--(-3,-2) node[midway, left] {$h_2$};
            \draw[line width =2pt] (-1,0)--(-1,-2) node[midway, left] {$h_3$};
            \draw[line width =2pt] (1,0)--(1,-2) node[midway, left] {$h_4$};
            \draw[line width = 2pt] (1,-3) ellipse [y radius = 1 , x radius = 0.5];
            \draw[line width =2pt] (3,0)--(3,-2) node[midway, left] {$h_5$};
            \draw[line width = 2pt] (3,-3) ellipse [y radius = 1 , x radius = 0.5];
            \draw[line width = 2pt] (6,0) ellipse [x radius = 1 , y radius = 0.5];
            \draw[line width = 2pt] (-1,-3) ellipse [y radius = 1 , x radius = 0.5];
            \draw[line width = 2pt] (-3,-3) ellipse [y radius = 1 , x radius = 0.5];
            \draw[fill = black] (-5,0) circle (3pt);
            \draw[fill = black] (-3,0) circle (3pt);
            \draw[fill = black] (-1,0) circle (3pt);
            \draw[fill = black] (-1,-2) circle (3pt);
            \draw[fill = black] (-3,-2) circle (3pt);
            \draw[fill = black] (1,0) circle (3pt);
            \draw[fill = black] (3,0) circle (3pt);
            \draw[fill = black] (1,-2) circle (3pt);
            \draw[fill = black] (3,-2) circle (3pt);
            \draw[fill = black] (5,0) circle (3pt);
            
            \node at (-5,0.5) {$V_1$};
            \node at (-3,0.5) {$W_0$};
            \node at (-1,0.5) {$W_1$};
            \node at (1,0.5) {$W_2$};
            \node at (3,0.5) {$W_3$};
            \node at (4.9,0.5) {$V_4$};
            \node at (-6,0) {$l_1$};
            \node at (-1,-3) {$l_2$};
            \node at (-3,-3) {$l_3$};
            \node at (1,-3) {$l_4$};
            \node at (3,-3) {$l_5$};
            \node at (6,0) {$l_6$};
        \end{tikzpicture}
        \caption{Depiction of $O_{6}$}
        \label{fig: og}
\end{figure}
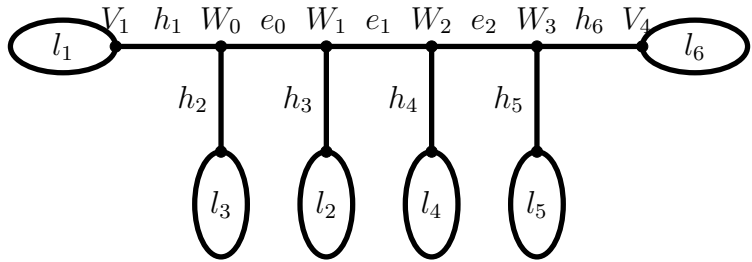

We now deal with some generalities concerning the permutation of the markings, and make use of the set-up \eqref{eq: big align3}. The group $(S_{d-2})^m$ acts naturally on $\AC_{d,h}(\vec{\mu})$ as follows: For an element $(\alpha_1,\dots,\alpha_m) \in (S_{d-2})^m$, we examine the following association
    \begin{equation*}
        {\vec{\alpha}}\cdot(\bullet) \colon \bbAC_{d,h}(\vec{\mu})\to \bbAC_{d,h}(\vec{\mu}), \pi\mapsto {\vec{\alpha}}\cdot (\pi),
    \end{equation*}
where ${\vec{\alpha}}\cdot(\pi)$ is the discrete admissible cover obtained by letting $\alpha_i$ permute the $(d-2)$ marked legs of weight $1$ of $\src(\pi)$ in the fiber of the $i$th marked leg of $T_t$. At morphisms this association is defined as the morphism between the targets obtained by the corresponding permutations of the marked legs. This actually gives rise to a functor, and we let $\eta_{\vec{\alpha}}$ denote the natural transformation $\AC_{d,h}(\vec{\mu})\implies \AC_{d,h}(\vec{\mu})\circ\phi_{\vec{\alpha}}$ given by the linear maps corresponding to the permutations of the marked legs. In other words, this defines an automorphism ${\vec{\alpha}}\cdot(\bullet)$ of the poic-space $\AC_{d,h}(\vec{\mu})$. It can be further observed that we actually obtain an action of $(S_{d-2})^m$ on $\AC_{d,h}(\vec{\mu})$.

Similarly, the group $S_m$ acts naturally on $\AC_{d,h}(\vec{\mu})$. Namely, if $\pi$ is an object of $\bbAC_{d,h}(\vec{\mu})$ and $\beta\in S_m$, then we let $\beta(\pi)$ denote the object of $\bbAC_{d,h}(\vec{\mu})$ obtained by letting $\beta$ permute the marking of the target tree $\trgt(\pi)$, and correspondingly permuting the marking of the source graph $\src(\pi)$. Just as in the previous paragraph, the analogous constructions actually give rise to an action of $S_m$ on the poic-space $\AC_{d,h}(\vec{\mu})$. 

\begin{thm}\label{thm: catalan many general}
     Suppose $g$, $h$, $d$, $m$, $\vec{\mu}$, $n$, and $J$ are as in \eqref{eq: big align3}. Then 
     \begin{equation}
         \mathfrak{src}_*\varpi^J_{d,h,\vec{\mu}}=\left( (3g)!\cdot \left(\left(\frac{g+r}{2}-1\right)!\right)^{3g+r}\cdot C_{\frac{g+r}{2}}\right)\cdot [\st_{g,J}].\label{eq: catalan many general}
     \end{equation}
\end{thm}
\begin{proof}
    We have readily observed that under these conditions if $\mathfrak{src}_*\varpi^J_{d,h,\vec{\mu}}$ is non-trivial, then it is a top dimensional cycle and hence a rational multiple of the fundamental cycle $[\st_{g,J}]$. To show that it is non-trivial we will compute its multiplicity. For this, our attention is shifted towards a particular cone, and we count the covers in the fiber with the correct multiplicity. We have to keep in mind that this is a pushforward, hence, in principle, the cone where we count should be a cone of a ($\mathfrak{src}^\pcmplxs$-fine) subdivision. Nevertheless, we place ourselves at the very special case of $O_{g,J}$ where we do not need to subdivide. This computation basically follows that of \cite{VargasThesis} section 13.2.\\
    Consider a cover $\pi$ of $\bbAC_{d,h}(\vec{\mu})$ with $\ft_{J^c}(\src(\pi)) \cong O_{g,J}$. It follows from Lemma \ref{lem: loopandbridge} that $\src(\pi)$ must have the edges of $h_1^\prime,h_2,h_3,\dots,h_6$ with weight $2$ (the condition on the matrix being invertible means that they give non-trivial contributions in the pushforward). We abuse notation and denote the corresponding edges of $\src(\pi)$ in the same way as those of $O_{g,J}$. The vanishing of RH numbers forces that:
    
    \begin{tabular}{c l c l}
         (a)& $d_\pi(c_1)-d_\pi(h_1^\prime) = \pm 1$,  &   (d)& $d_\pi(e_0)-d_\pi(h_{1})=\pm1$, \\
         (b)&  $d_\pi(c_{i})-d_\pi(c_{i-1}) = \pm 1$ for $2\leq i\leq r-1$, &   (e)& $d_\pi(e_{i})-d_\pi(e_{i-1}) = \pm 1$ for $1\leq i\leq g-4$,  \\
         (c)&  $d_\pi(h_1)-d_\pi(c_{r-1}=\pm1$, &   (f)&$d_\pi(h_g)-d_\pi(e_{g-4}) = \pm1$.
    \end{tabular}
    
    The previous items, Lemma \ref{lem: loopandbridge}, and the vanishing of RH numbers imply that that 
    \begin{enumerate}
        \item[(i)] The ramification profiles at $\pi(h_1^\prime),\pi(h_2),\pi(h_3),\dots,\pi(h_g)$ are $(2,1^{d-2})$.
        \item[(ii)] The ramification profile at $\pi(c_i)$ is $({d_\pi(c_i)},1^{d-d_\pi(c_i)})$ for $1\leq i\leq r-1$.
        \item[(iii)] The ramification profile at $\pi(e_i)$ is $({d_\pi(e_i)},1^{d-d_\pi(e_i)})$ for $0\leq i\leq g-4$.
    \end{enumerate}
    These last conditions allow us to classify the isomorphism class of the target tree $T_t$ up to permutation of the marking. The structure of the tree is as follows:
    \begin{itemize}
        \item If $J\neq \varnothing$, the target tree is isomorphic, up to permutation of the marking, to the tree $T_{g,J}$ specified as follows (we depict this tree for $g=3$ and $r=1$ in Figure \ref{fig: treetgj}):
        \begin{itemize}
            \item It has vertices $B_1$, $\dots$ ,$B_g$, $B_1^\prime$, $\dots$, $B_g^\prime$, $M_0$, $\dots$, $M_{g-3}$, and $R_1$, $\dots$, $R_r$.
            \item There are $g$ edges $b_i$, for $1\leq i\leq g$, with $\partial b_i = \{B_i,B^\prime_i\}$.
            \item There are edges $m_0$, $\dots$, $m_{g-4}$ with $\partial m_i = \{M_i,M_{i+1}\}$ for $0\leq i\leq g-4$.
            \item There are edges $q_1^\prime$, $q_1$, $q_2$, $\dots$, $q_g$ with $\partial q_1^\prime = \{B_1,R_1\}$, $\partial q_1=\{R_r,M_0\}$, $\partial q_i = \{B_i,M_{i-2}\}$ for $2\leq i\leq g-1$, and $\partial q_g = \{B_g,M_{g-3}\}$.
            \item There are edges $p_1$, $\dots$, $p_{r-1}$ with $\partial p_i = \{R_i,R_{i+1}\}$ for $1\leq i\leq r-1$. 
            \item There are legs $L_1$, $\dots$, $L_{3g+r}$, with:
            \begin{itemize}
                \item $\partial L_i = R_i$, for $1\leq i\leq r$.
                \item $\partial L_{r+3\cdot i-2} = \partial L_{r+3\cdot i-1} = \{B_i^\prime\} $ and $\partial L_{r+3\cdot i} = \{B_i\}$, for $1\leq i \leq g$.
            \end{itemize}
    \end{itemize}
        \item If $J=\varnothing$, the target tree is isomorphic, up to permutation of the marking, to the tree $T_{g,\varnothing}$ specified as follows:
        \begin{itemize}
            \item It has vertices $B_1$, $\dots$ , $B_g$, $B_1^\prime$, $\dots$, $B_g^\prime$, and $M_0$, $\dots$, $M_{g-3}$.
            \item There are $g$ edges $b_i$, $1\leq i\leq g$, with $\partial b_i = \{B_i,B_i^\prime\}$.
            \item There are edges $m_0,\dots, m_{g-4}$ with $\partial m_i = \{M_i,M_{i+1}\}$ for $0\leq i\leq g-4$.
            \item There are edges $q_1,\dots, q_g$, with $\partial q_1=\{M_0,B_1\}$, $\partial q_i =\{B_i,M_{i-2}\}$ for $2\leq i\leq g-1$, and $\partial q_g = \{B_g,M_{g-3}\}$.
            \item There are legs $L_1$, $\dots$, $L_{3g}$, with $\partial L_{3\cdot i-2}=\partial L_{3\cdot i-1} = \{B_i^\prime\}$ and $L_{3\cdot i} = \{B_i\}$ for $1\leq i\leq g$.
        \end{itemize}
    \end{itemize}
    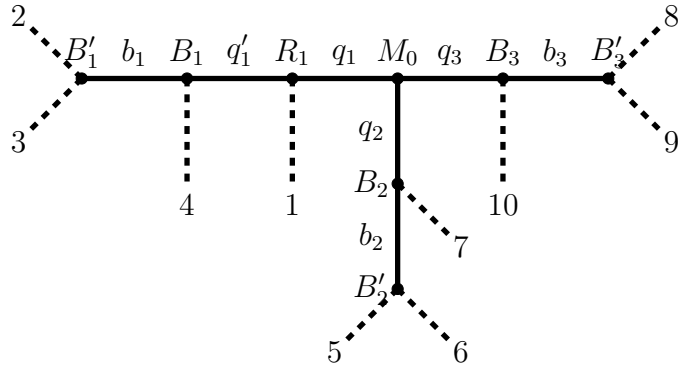
\begin{figure}[ht!]
        \centering
        \begin{tikzpicture}[scale=0.7]
            \draw[line width = 2pt, dashed] (-5,0)--(-6,1) node [pos = 1.2] {$2$};
            \draw[line width = 2pt, dashed] (-5,0)--(-6,-1) node [pos = 1.2] {$3$};
            \draw[line width = 2pt, dashed] (-3,0)--(-3,-2) node [pos = 1.2] {$4$};
            \draw[line width = 2pt, dashed] (1,-4)--(0,-5) node [pos = 1.2] {$5$};
            \draw[line width = 2pt, dashed] (1,-4)--(2,-5) node [pos = 1.2] {$6$};
            \draw[line width = 2pt, dashed] (1,-2)--(2,-3) node [pos = 1.2] {$7$};
            \draw[line width = 2pt, dashed] (5,0)--(6,1) node [pos = 1.2] {$8$};
            \draw[line width = 2pt, dashed] (5,0)--(6,-1) node [pos = 1.2] {$9$};
            \draw[line width = 2pt, dashed] (3,0)--(3,-2) node [pos = 1.2] {$10$};
            \draw[line width = 2pt] (-5,0)--(-3,0) node [midway, above]{$b_1$};
            \draw[line width = 2pt] (-3,0)--(-1,0) node [midway, above]{$q_1^\prime$};
            \draw[line width = 2pt] (-1,0)--(1,0) node [midway, above]{$q_1$};
            \draw[line width = 2pt] (1,0)--(3,0) node [midway, above]{$q_3$};
            \draw[line width = 2pt] (3,0)--(5,0) node [midway, above]{$b_3$};
            \draw[line width =2pt, dashed] (-1,0)--(-1,-2) node[pos =1.2] {$1$};
            \draw[line width =2pt] (1,0)--(1,-2) node[midway, left] {$q_2$};
            \draw[line width = 2pt] (1,-2)--(1,-4) node [midway, left ]{$b_2$};
            \draw[fill = black] (-5,0) circle (3pt);
            \draw[fill = black] (-3,0) circle (3pt);
            \draw[fill = black] (-1,0) circle (3pt);
            \draw[fill = black] (1,0) circle (3pt);
            \draw[fill = black] (3,0) circle (3pt);
            \draw[fill = black] (1,-2) circle (3pt);
            \draw[fill = black] (1,-4) circle (3pt);
            \draw[fill = black] (5,0) circle (3pt);
            
            \node at (-5,0.5) {$B_1^\prime$};
            \node at (0.5,-2) {$B_2$};
            \node at (0.5,-4) {$B_2^\prime$};
            \node at (-3,0.5) {$B_1$};
            \node at (-1,0.5) {$R_1$};
            \node at (1,0.5) {$M_0$};
            \node at (3,0.5) {$B_3$};
            \node at (5,0.5) {$B_3^\prime$};
        \end{tikzpicture}
        \caption{Depiction of $T_{g,J}$ for $g=4$ and $r=2$}
        \label{fig: treetgj}
\end{figure}
    We observe that if $T_t\cong T_{g,J}$, then under the cover $\pi$ we have that:
    \begin{enumerate}
        \item[(1)] The image of $h_1^\prime$ corresponds to $q_1^\prime$, and furthermore the image of $\pi(h_i)$ corresponds to $q_i$ for $1\leq i\leq g$.
        \item[(2)] The image of $c_i$ corresponds to $p_i$ for $1\leq i\leq r-1$.
        \item[(3)] The image of $e_i$ corresponds to $m_i$ for $0\leq i\leq g-4$.
        \item[(4)] The image of the loop $l_i$ corresponds to $b_i$ for $1\leq i\leq g$. 
    \end{enumerate}
    These previous correspondences fix the behavior of the cover $\pi$at the vertices.\\
    Coming back to the multiplicity, we observe that due to the markings of the legs of $\src(\pi)$, different sequences of weights for the edges $c_i$ (with $1\leq i\leq r-1$) and the edges $e_k$ (for $0\leq k\leq g-4$), satisfying the conditions (a) to (f), will give rise to different isomorphism classes that contribute to the computation of the multiplicity. More precisely, if we fix the isomorphism class of $T_s$ and of $T_t$, then different sequences of weights give rise to non-isomorphic covers. From this we obtain directly the Catalan factor of $C_{\frac{g+r}{2}}$. Indeed, given weights $d(h_1^\prime)$, $d(c_1)$, $d(c_2)$, $\dots$, $d(c_{r-1})$, $d(e_0)$, $\dots$, $d(e_{g-4})$, and $d(h_g)$ that satisfy the set of conditions (a) to (f) and $d(h_1^\prime)=d(h_g)=2$, then we obtain a Dyck path from $0$ to $g+r$ by considering the ensuing points of $\bbZ^2:$
    \begin{multline*}
        (0,0),\, (1,d(h_1^\prime)-1), \,(2,d(c_1)-1),\,\dots,\,(r,d(c_{r-1})-1),),\, (r+1,d(h_1)-1),\\
        (r+2,d(e_0)-1),\,\dots,\,(r+g-2,d(e_{g-4})-1),\,(r+g-1,d(h_g)-1),\,(r+g,0).
    \end{multline*}
    Since the weights are always positive, this path never goes below the $x$-axis. Therefore, it defines a Dyck path from $0$ to $g+r$, and, vice versa, any Dyck path from $0$ to $g+r$ gives in this precise manner a set of weights that gives rise to an isomorphism class that we count.\\
    The symmetric group $S_g$ acts on $\bbG_{0,\mathbbm{g}\sqcup J}$ by permuting the $\mathbbm{g}$-marked legs accordingly. The isomorphism classes of trees of $\bbG_{0,\mathbbm{g}\sqcup J}$ whose image is isomorphic to $O_{g,J}$ are all related by permutations of $S_g$. We fix one object $T_O$ of $\bbG_{0,\mathbbm{g}\sqcup J}$ with $\st_{g,J}(T_O)\cong O_{g,J}$ and proceed to count the isomorphism classes of the triples of $\bbEAC^J_{d,h}(\vec{\mu})$ whose image under the morphism $\mathfrak{src}^\pcmplxs$ is in the isomorphism class of $T_O$. If $(T_s,T_t,\pi)$ is an object of $\bbEAC^J_{d,h}(\vec{\mu})$ with $\mathfrak{src}^\pcmplxs(T_s,T_t,\pi) \cong T_O$, then the matrix corresponding to the map $\eta_{\src,(T_s,T_t,\pi)}\colon \EST_{(T_s,T_t,\pi)}\to \ST_{T_s}$ is square and diagonal. In fact, it coincides with the composition $\left(\eta_{\ft_{J^c},\src(\pi)} \circ F_\pi\right)$ up to reordering of columns and rows. Furthermore, we have that the only non-trivial entries are the ones corresponding to the loops (all others cancel due to the ramification profiles) which are $2$. This shows that the index of this map is $2^g$.\\
    We now study the contribution of the weight $\varpi^J_{d,h,\vec{\mu}}$. It follows from Lemma \ref{lem: loopandbridge} that the ramification profiles above the leaves of $T_t$ are always $(1^d)$, and we also have the conditions (i) to (iii) on ramification profiles above the other edges. This shows that in this case
    \begin{equation}
        \varpi^J_{d,h,\vec{\mu}}(\pi) = \prod_{V\in\src(\pi)} H(V)\cdot \mathrm{CF}(V).\label{eq: the weight}
    \end{equation}
    Since the RH numbers are trivial, the conditions (i) to (iii) on the ramification profiles at the edges of $T_t$ imply that there are the following possibilities for partitions of $d_\pi(V)$ at a vertex $V\in V(\src(\pi))$:
    \begin{enumerate}
        \item[(A)] If $d_\pi(V)\geq 2$, then we will see $(2,1^{d_\pi(V)-2})$, $(d_\pi(V)-1,1)$, and $(d_\pi(V))$.
        \item[(B)]  If $d_\pi(V)=1$, then we will just have $(1)$, $(1)$, and $(1)$.
    \end{enumerate}
    We observe that in (A) we obtain the following triple rational Hurwitz number
    \begin{enumerate}
        \item[(I)] If $d_\pi(V)= 2$, then $H_{0\to0}((2),(1,1),(2))=\frac{1}{2}$.
        \item[(II)] If $d_\pi(V)\geq 3$, then $H_{0\to0}((2,1^{d_\pi(V)-2}),(d_\pi(V)-1,1),(d_\pi(V))=1$. Indeed, this is best seen by observing that this Hurwitz number corresponds to the number of transpositions and $(d_\pi(V)-1)$-cycles that multiply to a $d_\pi(V)$-cycle in $S_{d_\pi(V)}$ divided by $d_\pi(V)!$. After fixing the $(d_\pi(V)-1)$-cycle, we obtain that there are $(d_\pi(V)-1)$ transpositions that when multiplied with the fixed $(d_\pi(V)-1)$-cycle give a $d_\pi(V)$-cycle. Since the number of $(d_\pi(V)-1)$-cycles is $\frac{d_\pi(V)!}{(d_\pi(V)-1)}$, it follows that $H_{0\to0}((2,1^{d_\pi(V)-2}),(d_\pi(V)-1,1),(d_\pi(V))=1$.
    \end{enumerate}
    Of course in (B) the rational triple Hurwitz number is also $1$. At the vertices where we forcibly run into case (I) from above, this fraction is canceled by the factor of $\mathrm{CF}(V)$ that is coming from the two edges with weight $1$ (this is the case, for instance, at all the $V_1,\dots,V_g$). The factors $\mathrm{CF}(V)$ of the vertices $V$ arising in case (II) are relevant when computing the number of isomorphism classes of marked discrete graphs realizing to the isomorphism class of the given $T_s$. Namely, after marking the target tree $T_t$, the markings of each weight-$1$ leg on the source can be permuted along the same fiber. This is the action of $S_{d-2}^{m}$ that we have previously described, and acting by an $\vec{\alpha}\in S_{d-2}^m$ may not necessarily produce new isomorphism classes. The action of the permutations $\vec{\alpha}$ is reflected locally at the vertices of $\src(\pi)$ (or $T_s$) by the combinatorial factors $\mathrm{CF}(\bullet)$, and the only vertices of interest are those with local degree bigger than $1$ that are not lying in the fibers of the $B_1$, $\dots$, $B_g$, $B_1^\prime$, $\dots$, $B_g^\prime$ (these are vertices where case (a) above applies and have nothing to do with the action of $S_{d-2}^m$). Our conditions (i) to (iii) on ramification profiles, the vanishing of the RH numbers, and Lemma \ref{lem: loopandbridge} imply that these vertices with local degree bigger than $1$ not in the fiber of the $B_i$ and $B_i^\prime$ ($1\leq i\leq g$) are exactly the $C_1$, $\dots$, $C_r$, $W_0$, $\dots$, $W_{g-3}$. We have already explained the behavior of the $\mathrm{CF}(\bullet)$ at the $V_i$ for $1\leq i\leq g$, so we focus on the others. If all these had weight $\leq 2$, then we would have $\#S_{d-2}^m=((d-2)!)^m$ possibilities. If there are vertices with higher weights, then we have to divide by the stabilizer, which is given by the factors from $\mathrm{CF}(\bullet)$. These cancel out with the remaining $\mathrm{CF}(V)$ in \eqref{eq: the weight}, and all together this shows that we obtain the factor of $((d-2)!)^m$. More succinctly, we are counting the orbits of the $S_{d-2}^m$ action as many times as the remaining factors $\mathrm{CF}(V)$ in \eqref{eq: the weight}, which turn out to be the size of the corresponding stabilizer.\\
    To finish, we just have to understand what happens when we permute the marking of the base tree. In other words, we study the previously defined action of $S_m$. Observe that when permuting the marking of the legs of $T_t$ (and change $\pi$ correspondingly) we will only stay in the same isomorphism class of $\bbEAC^J_{d,h}(\vec{\mu})$ if and only if the permutation switches legs incident to the same vertex and leaves the first $r$ legs fixed (because of the $J$-marked legs). This shows that there are $\frac{(m-r)!}{2^N}$ different permutations of the marking that produce different isomorphism classes, where $N$ is the number of vertices of $T_t$ with leg valency $2$. It is clear that, in this case, $N=g$, and therefore we have shown that
    \begin{align*}
        \mathfrak{src}_*\varpi^J_{d,h,\vec{\mu}} &= \left( \frac{(m-r)!}{2^g} \cdot ((d-2)!)^m\cdot 2^g\cdot C_{\frac{g+r}{2}}\right)[\st_{g,J}]\\
        &=\left((3g)! \cdot \left(\left(\frac{g+r}{2}-1\right)!\right)^{3g+r}\cdot C_{\frac{g+r}{2}}\right)[\st_{g,J}].
    \end{align*}
\end{proof}

By further studying the actions of $S_m$ and $S_{d-2}^m$, it is possible to just keep the Catalan factor $C_{\frac{g+r}{2}}$ in \eqref{eq: catalan many general}. We make this more precise in Theorem \ref{thm: generic} and for this purpose, we introduce a multiplicity for objects of $\bbEAC^J_{d,h}(\vec{\mu})$, an equivalence relation generated by the actions of these groups, and then we prove the corresponding enumerative statement.
\begin{defi}
    Let $(T_s,T_t,\pi)$ be an object of $\bbEAC^J_{d,h}(\vec{\mu})$. Its \emph{multiplicity} is the number
\begin{equation}
    \mt(T_s,T_t,\pi):=\frac{\varpi^J_{d,0,\vec{\mu}}\left(\pi\right) \cdot \left|\det(\eta_{\ft_{J^c},\src(\pi)}\circ F_\pi)\right|}{\HS(\pi)\cdot \VS(\pi)},\label{eq: multiplicity}
\end{equation}
    where $\VS(\pi)$ and $\HS(\pi)$ are the sizes of the stabilizers of the action of $(S_{d-2})^m$ and $S_m$ correspondingly.  
\end{defi}
We now proceed with the equivalence relation, which is the notion of $J$-marked discrete admissible covers.
\begin{defi}
    The \emph{set of $J$-marked discrete admissible covers} is the set of equivalence classes of $[\bbAC_{d,h}(\vec{\mu})]$ under the following relation: $[\pi]\sim[\pi^\prime]$ if there exist $\vec{\alpha}\in\left(S_{d-2}\right)^m$ and $\beta\in S_m$ such that $[\pi] = \vec{\alpha}\left(\beta\left([\pi^\prime]\right)\right)$ and $\ft_{J^c}\left(\src\left([\pi]\right)\right) =\ft_{J^c}\left(\src\left([\pi^\prime]\right)\right)$. It is clear that the multiplicity \eqref{eq: multiplicity} gives a well-defined function on the set of $J$-marked discrete admissible covers.
\end{defi}

In the following theorem, we make use of the notion of tropical modifications of a tropical curve. A tropical modification of a tropical curve $\Gamma$ is simply a tropical curve $\Gamma^\prime$ that is obtained by grafting several rational tropical curves (metric trees possibly with other marked legs) to vertices of $\Gamma$. In the context of the theorem, these tropical modifications arise as points in the fibers of forgetting the marking morphisms. To be much more precise, given a genus-$g$ $J$-marked tropical curve $\Gamma\in |\Mtrop_{g,J}|$, the fiber $|\ft_{J^c}|^{-1}(\Gamma)$ consists of tropical modifications of $\Gamma$.

\begin{thm}\label{thm: generic}
    Suppose $g$, $h$, $d$, $m$, $\vec{\mu}$, $n$, and $J$ are as in \eqref{eq: big align3}. Then for a generic genus-$g$ $J$-marked tropical curve $\Gamma$, there are $C_{\frac{g+r}{2}}$ many degree-$d$ $J$-marked discrete admissible covers (counted with multiplicity induced from \eqref{eq: multiplicity}) that have a tropical modification of $\Gamma$ as a source.
\end{thm}
\begin{proof}
    Following Theorem \ref{thm: catalan many general} we take the geometric realization of $\mathfrak{src}\colon\EST^J_{d,h}(\vec{\mu})\to\ST_{g,J}$, and apply Corollary 2.36 of \cite{GathmannKerberMarkwigTFMSTC} (the argument follows verbatim). This shows that the fiber of a generic point $p\in\left|\ST_{g,J}\right|$ has $\left( (m-r)!\cdot ((d-2)!)^m\cdot C_{\frac{g+r}{2}}\right)$-many points counted with their respective multiplicity. By definition, if $q\in  |\mathfrak{src}^\pcmplxs|^{-1}(p)$ and lies in the realization of the cone of a triple $(T_s,T_t,\pi)$, then this multiplicity is
    \begin{equation}
        \varpi^J_{d,h,\vec{\mu}}(\pi)\cdot \left[ N^{\ST_{T_s}} : \mathfrak{src}^\pcmplxs\left( N^{\EST^J_{(T_s,T_t,\pi)}}\right)\right]= \varpi^J_{d,h,\vec{\mu}}(\pi)\cdot \left|\det\left(\eta_{\ft_{J^c},\src(\pi)}\circ F_\pi\right)\right|.\label{eq: multiplicity of fibre}
    \end{equation}
    We first consider the action of $(S_{d-2})^m$, which can be naturally extended to $\EST^J_{d,h}(\vec{\mu})$ (and hence to its realization). The action restricts to an action on the fiber $|\mathfrak{src}^\pcmplxs|^{-1}(p)$ preserving the corresponding multiplicities, hence if we divide by the stabilizers we can remove the factor $((d-2)!)^m$. By definition of the action, if a point $q\in |\mathfrak{src}^\pcmplxs|^{-1}(p)$ lies in the realization of the cone of a triple $(T_s,T_t,\pi)$, then the size of the stabilizer if $\VS(\pi)$. Let $G_p$ denote the subgroup of $S_m$ such that:
    \begin{itemize}
        \item $G_p\left(\{1,\dots,r\}\right)=\{1,\dots,r\}$.
        \item The induced action of $G_p$ on $J$ and its extension to $\ST_{g,J}$ leaves the point $p$ fixed. 
    \end{itemize}
    Similarly to the case of $(S_{d-2})^m$, the action of $S_m$ extends naturally to an action of $G_p$ on $\EST^J_{d,h}(\vec{\mu})$. By construction it restricts to an action on the fiber $|\mathfrak{src}^\pcmplxs|^{-1}(p)$, and the analogous argument shows that if we divide by the stabilizers then we can remove the factor $(m-r)!$. From the definition of $G_p$, the $S_m$-stabilizer of $\pi$ coincides with the $G_p$-stabilizer, and hence the size of the stabilizer if $\HS(\pi)$. It now follows from \eqref{eq: multiplicity of fibre} that counting the points in the fiber with the multiplicity induced from \eqref{eq: multiplicity} gives precisely $C_{\frac{g+r}{2}}$.
\end{proof}

We close with a proof of Lemma \ref{lem: loopandbridge}. This is a multi-step process for which we now build toward. The results and methods are simple, but the proofs are somewhat lengthy. We begin with Lemma \ref{lem: above nodes}, which consists of a useful classification of the structure of admissible covers that contribute to the multiplicity. We then proceed with Lemma \ref{lem: expungednoglue}, which requires the introduction of a definition. Afterward, we finally arrive at the proof of Lemma \ref{lem: loopandbridge}, which naturally makes use of the previous lemmas. The content of these lemmas and proofs is already known in the literature. However, in the spirit of clarity and self containment we provide concrete statements in our context and, correspondingly, proofs. More precisely, Lemma \ref{lem: above nodes} and its proof consist of a simple adaptation to our context of Remark 51, Lemmas 52 and 53, and their proofs of \cite{VargasDraisma}. Whereas Lemma \ref{lem: loopandbridge} and its proof are a minimalistic adaptation of Lemmas 60 and 79 of loc. cit..

\begin{lem}\label{lem: above nodes}
Following the notation of \eqref{eq: big align3}, let $\pi$ be an object of top dimension of $\bbAC_{d,h}(\vec{\mu})$ with $\left(\eta_{\ft_{J^c},\src(\pi)}\circ F_\pi\right)$ invertible. If $V$ is a vertex of $\trgt(\pi)$ with $\legval(V)=2$, then the fiber $\pi^{-1}(V)$ is given as follows:
\begin{itemize}
    \item There is a unique vertex $\hat{V}$ with $d_\pi(\hat{V})=2$.
    \item All the other vertices have local degree $1$. 
\end{itemize}
Furthermore, if $e\in E(\trgt(\pi))$ is the unique edge with $V\in\partial e$, then the ramification profile of $e$ is $(1^d)$. 
\end{lem}
\begin{proof}
Observe that every vertex of $\trgt(\pi)$ must be $3$-valent, because $\pi$ is of top dimension. Let $\ell_i$ and $\ell_j$ denote the marked legs of $\trgt(\pi)$ that are incident to $V$. Suppose $W\in \pi^{-1}(V)$ and consider partitions $\lambda_i,\lambda_j,\lambda \vdash d_\pi(W) $ such that
\begin{itemize}
    \item $\lambda_i$ is the partition given by the weights of the marked legs incident to $W$ in the fiber of $\ell_i$.
    \item $\lambda_j$ is the partition given by the weights of the marked legs incident to $W$ in the fiber of $\ell_j$.
    \item $\lambda$ is the partition given by the weights of the edges incident to $W$ in the fiber of $e$.
\end{itemize}
Since the ramification profile of $\ell_i$ and $\ell_j$ is $(2,1^{d-2})$, it follows that there are only two possibilities for $\lambda_i$ and $\lambda_j$: either $(2^{d_\pi(W)},1^{d_\pi(W)-2})$ or $(1^{d_\pi(W)})$. By an exhaustive argument we will show that the only cases are:
\begin{enumerate}
    \item[(\textasteriskcentered)] $d_\pi(W)= 2$, $\lambda_i = (2)$, $\lambda_j=(2)$, and $\lambda= (1,1)$.
    \item[(\textasteriskcentered\textasteriskcentered)] $d_\pi(W)= 1$, $\lambda_i = (1)$, $\lambda_j=(1)$, and $\lambda= (1)$. 
\end{enumerate}
We consider all the possible cases:
\begin{itemize}
    \item Suppose that both $\lambda_i=\lambda_j = (1^{d_\pi(W)})$. The vanishing of the RH number at $W$ yields the equation
    \begin{equation*}
        d_\pi(W)= \ell(\lambda_i)+\ell(\lambda_j)+\ell(\lambda)-2= d_\pi(W)+d_\pi(W)+\ell(\lambda)-2.
    \end{equation*}
    From this it follows that $d_\pi(W) + \ell(\lambda) = 2$, and since $\ell(\lambda),d_\pi(W)\geq 1$, we necessarily have that $d_\pi(W) = 1$ and $\lambda= (1)$. 
    \item Suppose that both $\lambda_i=\lambda_j = (2,1^{d_\pi(W)-2})$. The vanishing of the RH number at $W$ yields the equation
    \begin{equation*}
        d_\pi(W)= \ell(\lambda_i)+\ell(\lambda_j)+\ell(\lambda)-2= 2(d_\pi(W)-1)+\ell(\lambda)-2.
    \end{equation*}
    This implies that $d_\pi(W)+\ell(\lambda) = 4$. Observe that $d_\pi(W)\geq 2$ due to $\lambda_i$ and $\lambda_j$, and $\ell(\lambda)\geq 1$. Therefore, either $d_\pi(W) = 3$ and $\lambda=(3)$ or $d_\pi(W)=2$ and $\lambda= (1,1)$. We claim that the former is impossible. Indeed, in this case every other vertex of $\pi^{-1}(V)$ would necessarily have local degree $1$. This would imply that in the matrix $\left(\eta_{\ft_{J^c},\src(\pi)}\circ F_\pi\right)$ the column corresponding to the edge $e$ would be trivial, but this is not the case as this matrix is invertible. In conclusion, in this case $d_\pi(W)=2$ and $\lambda= (1,1)$.
    \item Suppose that $\lambda_i=(2^{d_\pi(W)},1^{d_\pi(W)-2})$ and $\lambda_j=(1^{d_\pi(W)})$ (or vice versa, the argument is the same). We argue that this is not a possibility. The vanishing of the RH number at $W$ yields the equation
    \begin{equation*}
        d_\pi(W)=\ell(\lambda_i)+\ell(\lambda_j)+\ell(\lambda)-2= (d_\pi(W)-1)+d_\pi(w)+\ell(\lambda)-2.
    \end{equation*}
    In particular, we obtain that $3=d_\pi(W)+\ell(\lambda)$. Observe that $d_\pi(W)\geq 2$ due to $\lambda_i$. Since $\lambda$ is a partition, necessarily we have that $\ell(\lambda)\geq 1$. These conditions force that $d_\pi(W) = 2$, $\lambda= (2)$. In this case, there would be another vertex $W^\prime\in \pi^{-1}(V)$ with $d_\pi(W^\prime)=2$ but the partitions above the $\ell_i$ and $\ell_j$ legs reversed, and every other vertex of $\pi^{-1}(V)$ would necessarily have local degree $1$. This is not possible due to $\left(\eta_{\ft_{J^c},\src(\pi)}\circ F_\pi\right)$ being invertible. Indeed, the column corresponding to the edge $e$ of $\trgt(\pi)$ would be trivial and the matrix could not possibly be invertible.
\end{itemize}
It readily follows from the above that the only possibilities are (\textasteriskcentered) and (\textasteriskcentered\textasteriskcentered).
Since the ramification profiles over $\ell_i$ and $\ell_j$ are $(2,1^{d-2})$, we must necessarily run into (\textasteriskcentered) only once in the fiber $\pi^{-1}(V)$ and at all the other vertices we run into (\textasteriskcentered\textasteriskcentered).
\end{proof}

We continue with a useful but technical definition, which is analogous to the dangling edges and vertices of \cite{VargasDraisma}.
\begin{defi}
Let $A$ be an arbitrary finite set, let $I\subset A$, let $q$ be a non-negative integer with $2q+\#(A\backslash I)-2\geq 0$ and let $G$ be an object of $\bbG_{g,A}$. We say that an edge $e\in E(G)$ is \emph{expunged after forgetting the $I$-marked legs} if the column corresponding to $e$ in the matrix $\eta_{\ft_I, G}$ is trivial. A vertex $V\in V(G)$ is \emph{expunged after forgetting the $I$-marked legs} if every leg incident to $V$ is $I$-marked and every edge incident to $V$ is expunged after forgetting the $I$-marked legs.
\end{defi}

\begin{lem}\label{lem: expungednoglue}
Following the notation of \eqref{eq: big align3}, let $\pi$ be an object of top dimension of $\bbAC_{d,h}(\vec{\mu})$ with $\left(\eta_{\ft_{J^c},\src(\pi)}\circ F_\pi\right)$ invertible. Any vertex of $\src(\pi)$ that is expunged after forgetting the $J^c$-marked legs has local degree $1$. Analogously, every edge of $\src(\pi)$ that is expunged after forgetting the $J^c$-marked legs has weight $1$. 
\end{lem}
\begin{proof}
  Since $\pi$ is of top dimension, every vertex of $\trgt(\pi)$ is $3$-valent. We begin with some observations concerning the structure of vertices of $\src(\pi)$ that are expunged after fogetting the $J^c$-marked legs. Let $V\in V(\src(\pi))$ be expunged after forgetting the $J^c$-marked legs, then we have the following:
    \begin{enumerate}
        \item[(\textbullet)] If $\legval(\pi(V))=2$, it necessarily holds $d_\pi(V)=1$.\\
        Indeed, because of Lemma \ref{lem: above nodes} we know that $d_\pi(V)=1$ or $d_\pi(V)=2$. In fact, every vertex in $\pi^{-1}(\pi(V))$ with local degree $1$ is expunged after forgetting the $J^c$-marked legs. If $d_\pi(V)=2$, then the column of the matrix $\left(\eta_{\ft_{J^c},\src(\pi)}\circ F_\pi\right)$ correspoding to the unique edge of $\trgt(\pi)$ incident to $\pi(V)$ would be trivial. Since the matrix is invertible, this is impossible and therefore $d_\pi(V)=1$.
        \item[(\textbullet\textbullet)] If $\legval(\pi(V))=1$ and $e_1,e_2\in E(\trgt(\pi))$ denote the two edges incident to $\pi(V)$, then every leg of $\src(\pi)$ incident to $V$ has degree $1$.\\
        We argue by contradiction and assume that there is a leg of $\src(\pi)$ incident to $V$ with degree $2$. Observe that there must be at least one vertex in the fiber $\pi^{-1}(\pi(V))$ that is not expunged after forgetting the $J^c$-marked legs because the matrix $\left(\eta_{\ft_{J^c,\src(\pi)}}\circ F_\pi\right)$ is invertible. If $W\in\pi^{-1}(\pi(V))$ is one of these vertices that are not expunged after forgetting the $J^c$-marked legs, then every leg of $\src(\pi)$ incident to $W$ must have weight $1$. Together with the vanishing of the RH number at $W$, this implies that there are two edges $e^W_1,e^W_2\in E(\src(\pi))$ incident to $W$ with $\pi(e^W_i) = e_i$ and $d_\pi(e^W_i) = d_\pi(W)$ for $i=1,2$. Since the vertex $W$ was arbitrary, this shows that the columns of the matrix $\left(\eta_{\ft_{J^c},\src(\pi)}\circ F_\pi\right)$ corresponding to the edges $e_1$ and $e_2$ are identical. But this is impossible because the matrix is invertible. Therefore, there is no leg of $\src(\pi)$ incident to $V$ with degree $2$.
        \item[(\textbullet\textbullet\textbullet)] If $A,B\in V(\src(\pi))$ are two vertices that are not expunged after forgetting the $J^c$-marked legs, then there cannot be a path between $A$ and $B$ consisting of edges that are expunged after forgetting the $J^c$-marked legs.\\
        This follows from the behavior of forgetting the marking morphisms. 
    \end{enumerate}
    After these previous observations, we turn our attention toward edges. We will show that every edge that is expunged after forgetting the $J^c$-marked legs must have weight $1$. We argue by contradiction and suppose that $a_1\in E(\src(\pi))$ is expunged after forgetting the $J^c$-marked legs with $d_\pi(e)>1$.\\
    Let us set $a_1 = e$ and $\partial a_1=\{W_1,V_1\}$. Because of (\textbullet\textbullet\textbullet) there is a vertex of $\partial a_1$ that is expunged after forgetting the $J^c$-marked legs, say $V_1$. Due to (\textbullet) and (\textbullet\textbullet) we can assume that every leg (if there are any) of $\src(\pi)$ that is incident to $V_1$ has weight $1$. Observe that the vanishing of the RH number at $V_1$ implies that there must exist an edge $a_2\in E(\src(\pi))$ incident to $a_1$ with $d_\pi(a_2)>1$. Let $V_2$ denote the other vertex of $\src(\pi)$ that is incident to $a_2$. Either $V_2$ is expunged after forgetting the $J^c$-marked legs or not. In the former case, clearly $V_2\neq V_1$ and we repeat the argument to produce $a_3$ and $V_3$. Again, either $V_3$ is expunged after forgetting the $J^c$-marked leg or not. In the former situation, we additionally have that $V_3\neq V_2$ and $V_3\neq V_1$ (because forgetting the marking preserves the genus of the graphs). Since there are only finitely many edges and vertices, reiterating this process we produce a sequence of edges $a_1,\dots,a_s\in E(\src(\pi))$ and different vertices $V_1,V_2,\dots,V_s,V_{s+1}\in V(\src(\pi))$ such that 
    \begin{itemize}
        \item All the edges $a_i$ and vertices $V_i$ are expunged after forgetting the $J^c$-marked legs and $d_\pi(a_i)>1$ (for $1\leq i\leq s$).
        \item The vertex $V_{s+1}$ is incident to $a_s$ and is not expunged after forgetting the $J^c$-marked leg.
    \end{itemize}
   Now, (\textbullet\textbullet\textbullet) implies that $W_1$ is expunged after forgetting the $J^c$-marked legs. Reiterating the process above we obtain a sequence of edges $b_1,\dots,b_t\in E(\src(\pi))$ (with $b_1=a_1=e$) and different vertices $W_1,\dots,W_t,W_{t+1}\in V(\src(\pi))$ such that
    \begin{itemize}
        \item All the edges $b_i$ and vertices $W_i$ are expunged after forgetting the $J^c$-marked legs and $d_\pi(b_i)>1$ (for $1\leq i\leq t$).
        \item The vertex $W_{s+1}$ is incident to $b_t$ and is not expunged after forgetting the $J^c$-marked leg.
    \end{itemize}
    If $\{W_1,\dots,W_t\}\cap \{V_1,\dots, V_s\} =\varnothing$, we arrive at a contradiction because of (\textbullet\textbullet\textbullet). Otherwise, we would also arrive at a contradiction. Because if $\{W_1,\dots,W_t\}\cap \{V_1,\dots, V_s\} \neq \varnothing$, then a cycle of $\src(\pi)$ is expunged after forgetting the $J^c$-marked legs, but this is not possible. In conclusion, if $e\in E(\src(\pi))$ is expunged after forgetting the $J^c$-marked legs, then $d_\pi(e)=1$.\\
    To finish the proof of the lemma we must show that a vertex $V\in V(\src(\pi))$ that is expunged after forgetting the $J^c$-marked legs has local degree $1$. If $\legval(\pi(V))=2$, then this readily follows from (\textbullet). If $\legval(\pi(V))<2$, then we have shown that every edge (and leg) of $\src(\pi)$ incident to $V$ has weight $1$. In this case, the vanishing of the RH number at $V$ forces that $d_\pi(V)=1$.
\end{proof}

Before finally proving Lemma \ref{lem: loopandbridge}, we introduce, in the spirit of clarity and conciseness, one last definition. 

\begin{defi}\label{defi: add}
Following the notation of \eqref{eq: big align3}, let $\pi$ denote an object of $\bbAC_{d,h}(\vec{\mu})$ and let $h\in E(\ft_{J^c}(\src(\pi))$. We say that an edge $e\in E(\src(\pi))$ \emph{adds} to the edge $e$ if the $(h,e)$-entry of the matrix $\eta_{\ft_{J^c},\src(\pi)}$ is non-zero.
\end{defi}

For the convenience of the reader, we restate Lemma \ref{lem: loopandbridge}. 
\loopandbridge*

\begin{proof}
    As we have mentioned before, every vertex of $\trgt(\pi)$ is $3$-valent because $\pi$ is of top dimension. We examine all the possibilities for $\legval(\pi(W))$. We first show that $\legval(\pi(W)) = 2$ and $\legval(\pi(W))=0$ are impossible. Then we show that the only possibility is $\legval(\pi(W))=1$ together with the ensuing statements from the lemma. 
    \begin{enumerate}
        \item[(\textopenbullet)] We observe that if $\legval(\pi(W))=2$, then $d_\pi(W)\geq 3$. This is impossible because of Lemma \ref{lem: above nodes}.
        \item[(\textopenbullet\textopenbullet)] Suppose that $\legval(\pi(W))=0$, and let us argue by contradiction. Let $a,b,c\in E(\trgt(\pi))$ denote the three edges that are incident to $\pi(W)$. Let $S$ denote the set of edges of $\src(\pi)$ that are incident to $W$. Recall that Lemma \ref{lem: expungednoglue} shows that every edge of $\src(\pi)$ that is expunged after forgetting the $J^c$-marked legs must have weight $1$.
        \begin{itemize}
            \item If every edge of $\pi^{-1}(a)\cap S$ and every edge of $\pi^{-1}(b)\cap S$ is expunged after forgetting the $J^c$-marked legs, then $d_\pi(W)\geq 3$ and $\pi^{-1}(c)\cap S\geq 3$. This is impossible due to the vanishing of the RH number at $W$.
            \item If every edge of $\pi^{-1}(a)\cap S$ is expunged after forgetting the $J^c$-marked legs, then $\left(\pi^{-1}(b)\cup\pi^{-1}(c)\right)\cap S\geq3$. Again, this is impossible due to the vanishing of the RH number at $W$.   
        \end{itemize}
        This shows that there are edges $e_a\in \pi^{-1}(a)$,  $e_b\in\pi^{-1}(b)$, and $e_c\in \pi^{-1}(c)$ that are not expunged after forgetting the $J^c$-markings and such that $e_a,e_b,e_c\in S$. In fact, $e_a$ (resp. $e_b$ and $e_c$) is the unique edge of $\pi^{-1}(a)\cap S$ (resp. $\pi^{-1}(b)\cap S$ and $\pi^{-1}(c)\cap S$) that is not expunged after forgetting the $J^c$-marked legs. Two of these, say $e_a$ and $e_b$ must {add} to the loop of $\ft_{J^c}(\src(\pi))$ incident to $W$, and the other, in this case $e_c$, {adds} to the other edge. Most importantly, since $e_a$ and $e_b$ {add} to the loop, each of these edges traces a path in $\src(\pi)$ that eventually leads to different edges $h_a$ and $h_b$ (respectively) of $\src(\pi)$. These edges $h_a$ and $h_b$ {add} to the loop after forgetting the $J^c$-marked legs, and $\pi(h_a)$ and $\pi(h_b)$ are different leaves of $\trgt(\pi)$. In this case, the columns of the matrix $\left(\eta_{\ft_{J^c,\src(\pi)}}\circ F_\pi\right)$ corresponding to the leaves $\pi(h_a)$ and $\pi(h_b)$ would be scalar multiples of each other, and the matrix would not be invertible. This is a contradiction, and therefore $\legval(W)=0$ is impossible.
    \end{enumerate}
    From the above, it follows that $\legval(W)=1$ is the only possibility.\\
    For the remaining statements from the lemma, let $e_1$ and $e_2$ denote the edges of $\trgt(\pi)$ that are incident to $\pi(W)$. Let $S\subset E(\src(\pi))$ denote the set of edges that are incident to $W$. Since the vertex $W$ is $3$-valent in $\ft_{J^c}(\src(\pi))$, it follows that there are exactly three edges in $\left(\pi^{-1}(e_1)\cup\pi^{-1}(e_2)\right)\cap S$ that are not expunged after forgetting the $J^c$-marked legs. Let us denote these edges by $a$, $b$ and $c$. Furthermore, the vanishing of the RH number at $W$ implies that $\{a,b,c\}\not\subset\pi^{-1}(e_i)$ for $i=1,2$. In addition, two of these edges, say $a$ and $b$, {add} to the loop of $\ft_{J^c}(\src(\pi))$ incident to $W$, and the other edge, in this case $c$, {adds} to the other edge. The same argument at the end of (\textopenbullet\textopenbullet) shows that $a$ and $b$ have to lie in the same fiber. Furthermore, since these edges {add} to the loop, each of these edges traces a path in $\src(\pi)$ that eventually leads to edges $h_a$ and $h_b$ (respectively) of $\src(\pi)$ such that: these edges {add} to the loop after forgetting the $J^c$-marked legs, and $\pi(h_a)$ and $\pi(h_b)$ are leaves of $\trgt(\pi)$. Once again, the same argument at the end of (\textopenbullet\textopenbullet) shows that it is impossible to have $\pi(h_a)\neq \pi(h_b)$. Therefore, it necessarily holds that $\pi(h_a)=\pi(h_b)$.\\
    We will show that $h_a=a$ and $h_b=b$. Let $p=\pi(h_a)=\pi(h_b)$ and $\partial p = \{P_1,P_2\}$ with $\legval(P_2) = 2$. In addition, we have that $\legval(P_2)<2$ (there is at least an additional edge since $\val_{\ft_{J^c}}(\src(\pi))=3$). It follows from Lemma \ref{lem: above nodes} that the ramification profile of $p$ is $(1^{d})$, and that there is a unique vertex $\hat{P}_1\in \pi^{-1}(P_1)$ with $d_\pi(\hat{P}_1)=2$. Furthermore, $\hat{P}_1$ is incident to both $h_a$ and $h_b$, and we also have that $d_\pi(h_a)=d_\pi(h_b)=1$. Let $P_{2,a}$ denote the other vertex of $h_a$, and $P_{2,b}$ denote the other vertex of $h_b$. It will be first shown that $\legval(P_2)\neq 0$.\\
    For the sake of contradiction, assume that $\legval(P_2)=0$, and let $q$ and $z$ denote the other edges of $\trgt(\pi)$ that are incident to $P_2$. Since $\trgt(\pi)$ is a tree and the edges $h_a$ and $h_b$ {add} to a loop of $\ft_{J^c}(\src(\pi))$, only one of the following is possible:
    \begin{itemize}
        \item Every edge of $\pi^{-1}(q)$ incident to $P_{2,a}$ is expunged after forgetting the $J^c$-marked legs, and the same happens for every edge of this fiber incident to $P_{2,b}$.
        \item Every edge of $\pi^{-1}(z)$ incident to $P_{2,a}$ is expunged after forgetting the $J^c$-marked legs, and the same happens for every edge of this fiber incident to $P_{2,b}$.
    \end{itemize}
    Without loss of generality, we assume the latter and carry on with the argument. In this case there is a unique edge $q_a\in\pi^{-1}(q)$ that is incident to $P_{2,a}$ and a unique edge $q_b\in \pi^{-1}(q)$ that is incident to $P_{2,b}$. Both $q_a$ and $q_b$ must have weight $1$. We also observe that the ramification profile of $p$ implies that every vertex $X\in\pi^{-1}(P_2)$ different from $P_{2,a}$ and $P_{2,b}$ has a unique incident edge $q_X\in\pi^{-1}(q)$ and a unique incident edge $z_X\in\pi^{-1}(z)$. In addition, the vanishing of the RH number at $X$ implies that $d_\pi(q_X) = d_\pi(z_X)=d_\pi(X)$. At this point we have arrived at a contradiction, since the previous conclusions give rise to a non-trivial linear relation among the columns of the matrix $\left(\eta_{\ft_{J^c,\src(\pi)}}\circ F_\pi\right)$. To be much more precise, the column corresponding to the edge $q$ of $\trgt(\pi)$ is the sum of the columns corresponding to the edges $z$ and $p$.\\
    The previous work implies that $\legval(P_2) = 1$. Let $q\in E(\trgt(\pi))$ denote the additional edge incident to $P_2$. We observe that every vertex of $\pi^{-1}(P_2)$ different from $P_{2,a}$ and $P_{2,b}$ does not give rise to a vertex of $\ft_{J^c}(\src(\pi))$. Indeed, such a vertex is either expunged after forgetting the $J^c$-marked legs or incident to a $J$-marked leg (in which case it gets ``absorbed'' by the leg). If $P_{2,a}\neq P_{2,b}$, then the columns of the matrix $\left(\eta_{\ft_{J^c},\src(\pi)} \circ F_\pi\right)$ corresponding to the edges $p$ and $q$ are identical. This is not possible since the matrix is invertible. Therefore, we have the equality $P_{2,a}=P_{2,b}$, which means that $h_a$ and $h_b$ define a cycle, and hence they must equal $a$ and $b$ respectively. Furthermore, we have that $W=P_{2,a}=P_{2,b}$, $e_2=q$, and $e_1$ is the leaf $p$. The vanishing of the RH numbers at all the other vertices of $\pi^{-1}(P_2) = \pi^{-1}(\pi(W))$ implies that $d_\pi(c)=2$, and the ramification profile above $q=e_2$ is $(2,1^{d-2})$. 
    \end{proof}
    
\subsection{Dictionary}

In this section we establish a dictionary between our framework and that of \cite{VargasDraisma} and \cite{VargasThesis}. In the spirit of concreteness, we avoid excessive length and refer to the specific results and terminology, while pointing toward the precise reference. 

In contrast to discrete admissible covers, in \cite{VargasDraisma} the authors consider \emph{DT-morphisms} (Definition 10 in loc. cit.) between a loopless unmarked graph and an unmarked tree, and these are allowed to have vertices of arbitrary positive valency. In our terms, a DT-morphism is simply a harmonic morphism with non-negative RH numbers. It can be shown that any DT-morphism can be obtained by forgetting the marked legs of a discrete admissible cover. The comparison is motivated by illustrating in Figure \ref{fig: vargas draisma example1} an example of \cite{VargasDraisma} (Case 7a of Section 4 in loc. cit.) and in Figure \ref{fig: vargas draisma example2} the same example within our context (namely, with markings).

\begin{figure}[!ht]
\centering
\begin{minipage}{0.5\textwidth}
    \centering
    \begin{tikzpicture}[xscale=0.7]
    
    \draw[line width= 1.5pt] (-3.5,3.6) -- (-3,2.8);
    \draw[line width= 1.5pt] (-4,1.6) -- (-3,2.8);
    \draw[line width= 1.5pt] (-3,2.8) -- (-1.5,2.8);
    \draw[line width= 1.5pt] (-2,1.8) -- (-1.5,2.8);
    \draw[line width= 1.5pt] (-1.5,2.8) -- (0,2.8);
    \draw[line width= 1.5pt] (0,2.8) -- (1.5,2.8);
    \draw[line width= 1.5pt] (1.5,2.8) -- (3,2.8);
    \draw[line width= 1.5pt] (3,2.8) -- (4,3.6);
    \draw[line width= 1.5pt] (3,2.8) -- (3.5,1.6); 
    
    \draw[line width= 1.5pt] (-3.5,3.6) -- (-3,2.4);
    \draw[line width= 1.5pt] (-4,1.6) -- (-3,2.4);
    \draw[line width= 1.5pt] (-3,2.4) -- (-1.5,2.4);
    \draw[line width= 1.5pt] (-2,1.2) -- (-1.5,2.4);
    \draw[line width= 1.5pt] (-1.5,2.4) -- (0,2.2);
    \draw[line width= 1.5pt] (1.5,2.2) -- (3,2.4);
    \draw[line width= 1.5pt] (3,2.4) -- (4,3.6);
    \draw[line width= 1.5pt] (3,2.4) -- (3.5,1.6); 
    
    \draw[line width= 1.5pt] (-3.5,3) -- (-3,2);
    \draw[line width= 1.5pt] (-4,1) -- (-3,2);
    \draw[line width= 1.5pt] (-3,2) -- (-1.5,2);
    \draw[line width= 1.5pt] (-2,1.2) -- (-1.5,2);
    \draw[line width= 1.5pt] (-1.5,2) -- (0,2.2);
    \draw[line width= 4pt] (0,2.2) -- (1.5,2.2) node [midway, above] {$2$};
    \draw[line width= 1.5pt] (1.5,2.2) -- (3,2);
    \draw[line width= 1.5pt] (3,2) -- (4,3);
    \draw[line width= 1.5pt] (3,2) -- (3.5,1);

    \draw[black, fill=black] (-3.5,3.6) circle (3pt);
    \draw[black, fill=black] (-3.5,3) circle (3pt);
    \draw[black, fill=black] (-4,1.6) circle (3pt);
    \draw[black, fill=black] (-4,1) circle (3pt);
    \draw[black, fill=black] (-3,2) circle (3pt);
    \draw[black, fill=black] (-3,2.4) circle (3pt);
    \draw[black, fill=black] (-3,2.8) circle (3pt);
    \draw[black, fill=black] (-2,1.2) circle (3pt);
    \draw[black, fill=black] (-2,1.8) circle (3pt);
    \draw[black, fill=black] (-1.5,2) circle (3pt);
    \draw[black, fill=black] (-1.5,2.4) circle (3pt);
    \draw[black, fill=black] (-1.5,2.8) circle (3pt);
    \draw[black, fill=black] (0,2.2) circle (3pt);
    \draw[black, fill=black] (0,2.8) circle (3pt);
    \draw[black, fill=black] (1.5,2.2) circle (3pt);
    \draw[black, fill=black] (1.5,2.8) circle (3pt);
    \draw[black, fill=black] (3,2) circle (3pt);
    \draw[black, fill=black] (3,2.4) circle (3pt);
    \draw[black, fill=black] (3,2.8) circle (3pt);
    \draw[black, fill=black] (4,3) circle (3pt);
    \draw[black, fill=black] (4,3.6) circle (3pt);
    \draw[black, fill=black] (3.5,1) circle (3pt);
    \draw[black, fill=black] (3.5,1.6) circle (3pt);

    \draw[line width= 1.5pt] (-3.5,-0.7) -- (-3,-1.7);
    \draw[line width= 1.5pt] (-4,-2.7) -- (-3,-1.7);
    \draw[line width= 1.5pt] (-3,-1.7) -- (-1.5,-1.7);
    \draw[line width= 1.5pt] (-2,-2.7) -- (-1.5,-1.7);
    \draw[line width= 1.5pt] (-1.5,-1.7) -- (0,-1.7);
    \draw[line width= 1.5pt] (0,-1.7) -- (1.5,-1.7);
    \draw[line width= 1.5pt] (1.5,-1.7) -- (3,-1.7);
    \draw[line width= 1.5pt] (3,-1.7) -- (4,-0.7);
    \draw[line width= 1.5pt] (3,-1.7) -- (3.5,-2.7);

    \draw[black, fill=black] (-3.5,-0.7) circle (3pt);
    \draw[black, fill=black] (-4,-2.7) circle (3pt);
    \draw[black, fill=black] (-3,-1.7) circle (3pt);
    \draw[black, fill=black] (-2,-2.7) circle (3pt);
    \draw[black, fill=black] (-1.5,-1.7) circle (3pt);
    \draw[black, fill=black] (0,-1.7) circle (3pt);
    \draw[black, fill=black] (1.5,-1.7) circle (3pt);
    \draw[black, fill=black] (3,-1.7) circle (3pt);
    \draw[black, fill=black] (4,-0.7) circle (3pt);
    \draw[black, fill=black] (3.5,-2.7) circle (3pt);    

    \draw (-4.5,4)--(4.5,4) node [midway, above]{Source};
    \draw (-4.5,4.7)--(-4.5,-3.4)--(4.5,-3.4)--(4.5,4.7)--(-4.5,4.7);
    \draw (-4.5,0.5)--(4.5,0.5);
    \draw (-4.5,-0.2)--(4.5,-0.2)node[midway, above]{Target};
    
    \end{tikzpicture}
    \caption{An example from \cite{VargasDraisma}}
    \label{fig: vargas draisma example1}
\end{minipage}%
\begin{minipage}{0.5\textwidth}
    \centering
    \begin{tikzpicture}[xscale=0.8]
    \definecolor{urobilin}{rgb}{0.4, 0.3, 0.05}
    \definecolor{cadmiumgreen}{rgb}{0.0, 0.6, 0.24}
    
    \draw[red, dashed, line width= 2pt] (-3.5,3.6) -- (-2.5,3.6) node [pos=1.1] {$2$};
    \draw[blue, dashed, line width= 2pt] (-3.5,3.6) -- (-4.5,3.6) node [pos=1.1] {$2$};
    \draw[red, dashed, line width= 2pt] (-3.5,3) -- (-2.5,3);
    \draw[blue, dashed, line width= 2pt] (-3.5,3) -- (-4.5,3);

    \draw[cadmiumgreen, dashed, line width= 2pt] (4,3.6) -- (5,3.6) node [pos=1.1] {$2$};
    \draw[cyan, dashed, line width= 2pt] (4,3.6) -- (3,3.6) node [pos=1.1] {$2$};
    \draw[cadmiumgreen, dashed, line width= 2pt] (4,3) -- (5,3);
    \draw[cyan, dashed, line width= 2pt] (4,3) -- (3,3);

    \draw[line width= 1.5pt] (-3.5,3.6) -- (-3,2.8);
    \draw[line width= 1.5pt] (-4,1.6) -- (-3,2.8);
    \draw[line width= 1.5pt] (-3,2.8) -- (-1.5,2.8);
    \draw[line width= 1.5pt] (-2,1.8) -- (-1.5,2.8);
    \draw[line width= 1.5pt] (-1.5,2.8) -- (0,2.8);
    \draw[line width= 1.5pt] (0,2.8) -- (1.5,2.8);
    \draw[line width= 1.5pt] (1.5,2.8) -- (3,2.8);
    \draw[line width= 1.5pt] (3,2.8) -- (4,3.6);
    \draw[line width= 1.5pt] (3,2.8) -- (3.5,1.6); 
    
    \draw[line width= 1.5pt] (-3.5,3.6) -- (-3,2.4);
    \draw[line width= 1.5pt] (-4,1.6) -- (-3,2.4);
    \draw[line width= 1.5pt] (-3,2.4) -- (-1.5,2.4);
    \draw[line width= 1.5pt] (-2,1.2) -- (-1.5,2.4);
    \draw[line width= 1.5pt] (-1.5,2.4) -- (0,2.2);
    \draw[line width= 1.5pt] (1.5,2.2) -- (3,2.4);
    \draw[line width= 1.5pt] (3,2.4) -- (4,3.6);
    \draw[line width= 1.5pt] (3,2.4) -- (3.5,1.6); 
    
    \draw[line width= 1.5pt] (-3.5,3) -- (-3,2);
    \draw[line width= 1.5pt] (-4,1) -- (-3,2);
    \draw[line width= 1.5pt] (-3,2) -- (-1.5,2);
    \draw[line width= 1.5pt] (-2,1.2) -- (-1.5,2);
    \draw[line width= 1.5pt] (-1.5,2) -- (0,2.2);
    \draw[line width= 4pt] (0,2.2) -- (1.5,2.2) node [midway, above] {$2$};
    \draw[line width= 1.5pt] (1.5,2.2) -- (3,2);
    \draw[line width= 1.5pt] (3,2) -- (4,3);
    \draw[line width= 1.5pt] (3,2) -- (3.5,1);   

    \draw[magenta, dashed, line width= 2pt] (-4,1.6) -- (-4.5,2.1) node [pos=1.3] {$2$};
    \draw[gray, dashed, line width= 2pt] (-4,1.6) -- (-3.5,1.1) node [pos=1.3] {$2$};
    \draw[magenta, dashed, line width= 2pt] (-4,1) -- (-4.5,1.5);
    \draw[gray, dashed, line width= 2pt] (-4,1) -- (-3.5,0.5);
    
    \draw[brown, dashed, line width= 2pt] (-2,1.8) -- (-3,1.8);
    \draw[violet, dashed, line width= 2pt] (-2,1.8) -- (-1,1.8);
    \draw[brown, dashed, line width= 2pt] (-2,1.2) -- (-3,1.2) node [midway,below] {$2$};
    \draw[violet, dashed, line width= 2pt] (-2,1.2) -- (-1,1.2) node [midway,below] {$2$};
    
    \draw[olive, dashed, line width= 2pt] (0,2.8) -- (0.5,1.8);
    \draw[olive, dashed, line width= 2pt] (0,2.2) -- (0.5,1.2) node [pos=1.2] {$2$};
    
    \draw[orange, dashed, line width= 2pt] (1.5,2.8) -- (2,1.8);
    \draw[orange, dashed, line width= 2pt] (1.5,2.2) -- (2,1.2) node [pos=1.2] {$2$};
    
    \draw[purple, dashed, line width= 2pt] (3.5,1) -- (3,0.5);
    \draw[urobilin, dashed, line width= 2pt] (3.5,1) -- (4,1.5);
    \draw[purple, dashed, line width= 2pt] (3.5,1.6) -- (3,1.1) node [pos=1.3] {$2$};
    \draw[urobilin, dashed, line width= 2pt] (3.5,1.6) -- (4,2.1) node [pos=1.3] {$2$};

    \draw[black, fill=black] (-3.5,3.6) circle (3pt);
    \draw[black, fill=black] (-3.5,3) circle (3pt);
    \draw[black, fill=black] (-4,1.6) circle (3pt);
    \draw[black, fill=black] (-4,1) circle (3pt);
    \draw[black, fill=black] (-3,2) circle (3pt);
    \draw[black, fill=black] (-3,2.4) circle (3pt);
    \draw[black, fill=black] (-3,2.8) circle (3pt);
    \draw[black, fill=black] (-2,1.2) circle (3pt);
    \draw[black, fill=black] (-2,1.8) circle (3pt);
    \draw[black, fill=black] (-1.5,2) circle (3pt);
    \draw[black, fill=black] (-1.5,2.4) circle (3pt);
    \draw[black, fill=black] (-1.5,2.8) circle (3pt);
    \draw[black, fill=black] (0,2.2) circle (3pt);
    \draw[black, fill=black] (0,2.8) circle (3pt);
    \draw[black, fill=black] (1.5,2.2) circle (3pt);
    \draw[black, fill=black] (1.5,2.8) circle (3pt);
    \draw[black, fill=black] (3,2) circle (3pt);
    \draw[black, fill=black] (3,2.4) circle (3pt);
    \draw[black, fill=black] (3,2.8) circle (3pt);
    \draw[black, fill=black] (4,3) circle (3pt);
    \draw[black, fill=black] (4,3.6) circle (3pt);
    \draw[black, fill=black] (3.5,1) circle (3pt);
    \draw[black, fill=black] (3.5,1.6) circle (3pt);

    \draw[line width= 1.5pt] (-3.5,-0.7) -- (-3,-1.7);
    \draw[line width= 1.5pt] (-4,-2.7) -- (-3,-1.7);
    \draw[line width= 1.5pt] (-3,-1.7) -- (-1.5,-1.7);
    \draw[line width= 1.5pt] (-2,-2.7) -- (-1.5,-1.7);
    \draw[line width= 1.5pt] (-1.5,-1.7) -- (0,-1.7);
    \draw[line width= 1.5pt] (0,-1.7) -- (1.5,-1.7);
    \draw[line width= 1.5pt] (1.5,-1.7) -- (3,-1.7);
    \draw[line width= 1.5pt] (3,-1.7) -- (4,-0.7);
    \draw[line width= 1.5pt] (3,-1.7) -- (3.5,-2.7);

    \draw[red, dashed, line width= 2pt] (-3.5,-0.7) -- (-2.5,-0.7);
    \draw[blue, dashed, line width= 2pt] (-3.5,-0.7) -- (-4.5,-0.7);
    \draw[magenta, dashed, line width= 2pt] (-4,-2.7) -- (-4.5,-2.2);
    \draw[gray, dashed, line width= 2pt] (-4,-2.7) -- (-3.5,-3.2);
    \draw[brown, dashed, line width= 2pt] (-2,-2.7) -- (-3,-2.7);
    \draw[violet, dashed, line width= 2pt] (-2,-2.7) -- (-1,-2.7);
    \draw[olive, dashed, line width= 2pt] (0,-1.7) -- (0.5,-2.7);
    \draw[orange, dashed, line width= 2pt] (1.5,-1.7) -- (2,-2.7);
    \draw[purple, dashed, line width= 2pt] (3.5,-2.7) -- (3,-3.2);
    \draw[urobilin, dashed, line width= 2pt] (3.5,-2.7) -- (4,-2.2);
    \draw[cadmiumgreen, dashed, line width= 2pt] (4,-0.7) -- (5,-0.7);
    \draw[cyan, dashed, line width= 2pt] (4,-0.7) -- (3,-0.7);

    \draw[black, fill=black] (-3.5,-0.7) circle (3pt);
    \draw[black, fill=black] (-4,-2.7) circle (3pt);
    \draw[black, fill=black] (-3,-1.7) circle (3pt);
    \draw[black, fill=black] (-2,-2.7) circle (3pt);
    \draw[black, fill=black] (-1.5,-1.7) circle (3pt);
    \draw[black, fill=black] (0,-1.7) circle (3pt);
    \draw[black, fill=black] (1.5,-1.7) circle (3pt);
    \draw[black, fill=black] (3,-1.7) circle (3pt);
    \draw[black, fill=black] (4,-0.7) circle (3pt);
    \draw[black, fill=black] (3.5,-2.7) circle (3pt);

    \draw (-5,4)--(5.5,4) node [midway, above]{Source};
    \draw (-5,4.7)--(-5,-3.4)--(5.5,-3.4)--(5.5,4.7)--(-5,4.7);
    \draw (-5,0.3)--(5.5,0.3);
    \draw (-5,-0.4)--(5.5,-0.4) node [midway,above]{Target};
    \end{tikzpicture}
    \caption{Same example as Figure \ref{fig: vargas draisma example1}, but with our markings}
    \label{fig: vargas draisma example2}
    \end{minipage}
\end{figure}
In loc. cit. the authors restrict their attention to \emph{full-dimensional combinatorial types of DT-morphisms} (Definitions 21 and 34 of loc. cit.), and to approach this class of morphisms they are compelled to introduce (Section 3.3 of loc. cit.) several combinatorial properties (which are of course attested by every morphism of this class). Among these properties lies \emph{change-minimality} (Definitions 47 and 48 of loc. cit.), which gives rise to their dimension formula (Proposition 29 of loc. cit.) for a DT-morphism $\phi\colon G\to T$. In addition, this condition also forces the target tree $T$ to be at most $3$-valent, and furthermore:
\begin{enumerate}
    \item[(A)] If $V\in V(T)$ is $1$-valent, then $\sum_{W\in\phi^{-1}(V)}r_\phi(W) = 2$.
    \item[(B)] If $V\in V(T)$ is $2$-valent, then $\sum_{W\in\phi^{-1}(V)}r_\phi(W) = 1$. In particular, there is a unique $W\in\phi^{-1}(V)$ with $r_\phi(W) = 1$, and all other have vanishing RH number.
    \item[(C)] If $V\in V(T)$ is $3$-valent, then $\sum_{W\in\phi^{-1}(V)}r_\phi(W) = 0$. In particular, the RH number of every $W\in \phi^{-1}(W)$ vanishes.
\end{enumerate}

In the proof of the following proposition we describe a process of obtaining a discrete admissible cover from a DT-morphism. In Figure \ref{fig: vargas draisma example1} we have depicted a DT-morphism from \cite{VargasDraisma}, and in Figure \ref{fig: vargas draisma example2} we have depicted the discrete admissible cover produced out of this DT-morphism through the process described in the proof. 

\begin{prop}\label{prop: dictionary}
    Suppose $T$ is a tree, whose vertices have valency at most $3$. Any change-minimal degree-$d$ DT-morphism $\phi:G\to T$ can be obtained by forgetting the marked legs from an admissible cover $(G^\prime,T^\prime,\pi_\phi)$ in $\bbAC_{d,0}(\vec{\mu})$ where
    \begin{itemize}
        \item $m := 3+\#E(T)$.
        \item $\vec{\mu} := ((2,1^{d-2})^m)$.
    \end{itemize}
\end{prop}

\begin{proof}
        We first describe how to obtain $T^\prime$ out of $T$, then we explain how to obtain $G^\prime$ out of $G$, and finally we construct $\pi_\phi$. For the first part, we let $T^\prime$ denote a tree obtained by:
        \begin{itemize}
            \item Grafting two legs at every $1$-valent vertex of $T$.
            \item Grafting one leg at every $2$-valent vertex of $T$.
        \end{itemize}
        Of course $T^\prime$ has genus $0$, and since $T$ was at most $3$-valent, it follows that every vertex of $T^\prime$ is $3$-valent. To show that $T^\prime$ is actually an object of $\bbG_{0,m}$, we first observe that 
        \begin{equation*}
            \#V(T) = \#\val_1(T)+\#\val_2(T)+\#\val_3(T).
        \end{equation*}
        The fact that $T$ is a tree implies that
        \begin{equation}
        \#E(T)+1         =\#\val_1(T)+\#\val_2(T)+\#\val_3(T).\label{eq: 4}
        \end{equation}
        Additionally, we remark that
        \begin{equation}
            2\cdot\#E(T) =\sum_{v\in V(T)} \val(v) = \#\val_1(T)+2\cdot\#\val_2(T)+3\cdot\#\val_3(T),\label{eq: 5}
        \end{equation}
        and then multiplying \eqref{eq: 4} by $3$ and substracting \eqref{eq: 5} from the result we obtain that
    \begin{equation}
        \#E(T)+3 = 2\cdot\#\val_1(T)+\#\val_2(T).\label{eq: rhs}
    \end{equation}
    We observe that the right-hand side of \eqref{eq: rhs} coincides, by construction, with $\#L(T^\prime)$. As $m:=\#E(T)+3$, we have thus shown that $T^\prime$ is actually an object of $\bbG_{0,m}$.\\
    We now describe how to produce $G^\prime$ out of $G$. The idea here is to graft the legs lying above the newly grafted legs of $T^\prime$. Before we describe the process, it is worth remarking that by definition of the objects of $\bbAC_{d,0}(\vec{\mu})$, the legs lying above $\ell_{i}(T^\prime)$ (for $1\leq i\leq m$) must be $\ell_{1+(i-1)\cdot(d-1)},\ell_{1+(i-1)\cdot(d-1)+1}\dots,\ell_{1+(i-1)\cdot(d-1)-1}$. Additionally, the leg $\ell_{1+(i-1)\cdot(d-1)}$ must have weight $2$ and all the others must have weight $1$. We proceed as follows:
    \begin{itemize}
        \item Let $V\in V(T^\prime)$ be such that $\legval(V) =1$, and let $\ell_i(T^\prime)$ denote the leg incident to $V$. From (B) it follows that there is a unique $W\in\phi^{-1}(V)$ with $r_\phi(W)=1$. At every vertex $W^\prime \in \phi^{-1}(V)$ with $W^\prime\neq W$ we graft $d_\phi(W^\prime)$ many legs to $W^\prime$ from the set \begin{equation}
            \{\ell_{1+(i-1)\cdot(d-1)+1}\dots,\ell_{1+i\cdot(d-1)-1}\}.\label{eq: set of legs }
        \end{equation}
        This forces the vanishing of RH number (now with the legs) at these vertices. At the vertex $W$ we graft the leg $\ell_{1+i\cdot(d-1)}$, and graft the remaining $d_\phi(W)-2$ legs from \eqref{eq: set of legs }. As before, this also forces the vanishing of the RH number (now with the legs) at these vertices.
        \item Let $V\in V(T^\prime)$ be such that $\legval(V)=2$, and let $\ell_i(T^\prime)$ and $\ell_j(T^\prime)$ denote the legs incident to $V$. Because of (A) there are only two possiblities:
        \begin{itemize}
            \item There is a unique vertex $W\in \phi^{-1}(V)$ with $r_\phi(W)=2$, and the RH number vanishes at all the other vertices in this fiber.\\
            In this case, we graft the two weight $2$ legs $\ell_{1+(i-1)\cdot(d-1)}$ and $\ell_{1+(j-1)\cdot (d-1)}$ to this vertex. If necessary, we graft weight $1$ edges until the weights in both directions sum to $d_\phi(W)$. This forces the vanishing of the RH number at the vertex $W$. At all the other vertices $W^\prime\in\phi^{-1}(V)$ with $W^\prime\neq W$, we just graft $d_\phi(W^\prime)$ many weight $1$ edges from each of the sets
            \begin{equation*}
              \{\ell_{1+(i-1)\cdot(d-1)+1}\dots,\ell_{1+i\cdot(d-1)-1}\}, \textnormal{ and }\{\ell_{1+(j-1)\cdot(d-1)+1}\dots,\ell_{1+j\cdot(d-1)-1}\}.
            \end{equation*}
            A routine check shows that this causes the vanishing of the RH number at all these vertices.
            
            \item There are two vertices $W_1,W_2\in\phi^{-1}(V)$ with $r_\phi(W_1)=r_\phi(W_2)=1$, and the RH number vanishes at all the other vertices in this fiber. The vertices $W_1$ and $W_2$ are necessarily $1$-valent, and therefore $d_\phi(W_1)=d_\phi(W_2)=2$. We graft to $W_1$ the weight $2$ leg $\ell_{1+(i-1)\cdot(d-1)}$ and two weight $1$ legs from 
            \begin{equation*}
               \{\ell_{1+(j-1)\cdot(d-1)+1}\dots,\ell_{1+j\cdot(d-1)-1}\}.
            \end{equation*}
            This implies the vanishing of the RH number at $W_1$. Similarly, we graft to $W_2$ the weight $2$ leg $\ell_{1+(j-1)\cdot (d-1)}$ and two weight $1$ legs from 
            \begin{equation*}
               \{\ell_{1+(i-1)\cdot(d-1)+1}\dots,\ell_{1+i\cdot(d-1)-1}\}.
            \end{equation*}
            As before, this implies the vanishing of the RH number at $W_2$. We observe that every vertex $W\in\phi^{-1}(V)$ with $W\neq W_1$ and $W\neq W_2$ must satisfy $d_\phi(W)=1$. Therefore, we graft to every one of these vertices one weight $1$ edge (of the remaining) from each of the sets  
            \begin{equation*}
              \{\ell_{1+(i-1)\cdot(d-1)+1}\dots,\ell_{1+i\cdot(d-1)-1}\}, \textnormal{ and }\{\ell_{1+(j-1)\cdot(d-1)+1}\dots,\ell_{1+j\cdot(d-1)-1}\}.
            \end{equation*}
            This forces the vanishing of the RH number at the vertices $W\in\phi^{-1}(V)$ with $W\neq W_1$ and $W\neq W_2$.
        \end{itemize}
    \end{itemize}
    The previous concoction gives rise to a discrete graph $G^\prime$ which is forcibly $m\cdot (d-1)$-marked. Moreover, we obtain a harmonic morphism $\pi_\phi\colon G^\prime\to T^\prime$ by just using $\phi$ and specifying for each $0\leq i\leq m-1$
    \begin{equation*}
        \pi_\phi\left(\left\{ \ell_{1+i\cdot(d-1)},\dots,\ell_{1+i\cdot(d-1)+(d-2)}\right\}\right) :=\ell_{i+1}(T^\prime).
    \end{equation*}
    To finalize, we claim that this is, in fact, an admissible cover. We have already remarked the vanishing at the RH numbers at vertices of $G^\prime$ that lie over vertices of $T^\prime$ arising from $1$- or $2$-valent vertices of $T$. Since we are not modifying the behavior of $\phi$ at vertices of $G^\prime$ lying over vertices of $T^\prime$ that arise from $3$-valent vertices of $T$, the vanishing of the RH numbers readily follows from (C).
\end{proof}
\begin{remark}
    The previous argument can be carefully extended to arbitrary DT-morphisms. The change-minimality condition only simplified the execution of the argument, but was by no means fundamental to the whole procedure. The underlying idea is clear from an algebrogeometric perspective: A positive RH number represents the number of marked legs with simple ramification (weight $2$) that have to be grafted. In addition, we observe that grafting legs to vertices of the target and modifying the corresponding vertices of the source by grafting weight $1$ legs according to the local degrees preserves the harmonicity of the morphism as well as the RH numbers. 
\end{remark}

Following Theorem \ref{eq: catalan many general}, we can recover Theorem 2 of \cite{VargasDraisma} through Proposition \ref{prop: dictionary} by means of the following instances:
    \begin{itemize}
        \item Let $g$ be even, $r=0$, $d=\frac{g}{2}+1$, and $J=\varnothing$.
        \item Let $g$ be odd, $r=1$, $d=\frac{g+1}{2}$, and $J=\{1\}$.
    \end{itemize}
In both cases, Theorem \ref{eq: catalan many general} shows that any genus-$g$ tropical curve has a tropical modification that appears as the source of a degree-$d$ DT-morphism (in the second case we rely on the surjective of the forgetting the marking morphism between the realizations of the corresponding moduli spaces). Following Corollary 32 of \cite{VargasDraisma}, this is sufficient for Theorem 2 of loc. cit.. The second statement of the following proposition is the same enumerative result as that of Theorem 13.20 of \cite{VargasThesis}.
\begin{prop}
    Suppose that $g$ is even and $d=\frac{g}{2}+1$. Let $\phi\colon G\to T$ be a change-minimal full-rank DT-morphism and let $(G^\prime,T^\prime,\pi_\phi)$ be an object of $\bbAC_{d,h}(\vec{\mu})$ that gives $\phi$ after forgetting the marked legs as in Proposition \ref{prop: dictionary}. For any object $(T_s,T_t,\pi_\phi)$ of $\bbEAC_{d,h}^\varnothing(\vec{\mu})$, we have that 
    \begin{equation}
        \mt(T_s,T_t,\pi_\phi) = \textnormal{m}(\phi),\label{eq: VDmultiplicity}
    \end{equation}
    where $\textnormal{m}(\phi)$ is the Vargas-Draisma multiplicity of the DT-morphism $\phi$ (Definition 13.18 of \cite{VargasThesis}). In particular, for a generic genus-$g$ tropical curve $\Gamma$, there are $C_{\frac{g}{2}}$ many degree-$d$ $\varnothing$-marked discrete admissible covers (counted with multiplicity \eqref{eq: VDmultiplicity}) that have a tropical modification of $\Gamma$ as a source.
\end{prop}

\begin{proof}
    The second statement of the proposition is a consequence of Theorem \ref{thm: generic} and the first statement. Hence, we focus on the first statement of the proposition. It is shown in Lemma 13.23 of \cite{VargasThesis} that $\textnormal{m}(\phi) = D_\phi \left|\det(\hat{A}_\phi)\right|$, where $\hat{A}_\phi$ is the reduced edge-length matrix and the $D_\phi$ is the product of all the smallest positive integers $d_h$ that make the $h$-row of $\hat{A}_\phi$ integral for $h\in E(\ft_{J^c}(G_{\pi_\phi})$. The reduced edge-length matrix is just the edge-length matrix $A_\phi$ (Section 2.4 of \cite{VargasDraisma}) with the columns corresponding to leaves of $T^\prime$ divided by $2$. In other words, $\left|\det(\hat{A}_\phi)\right|= \frac{1}{2^N}\left|\det({A}_\phi)\right|$, where $N$ is the number of vertices of $T^\prime$ with leg valency $2$. Observe that the product of $\left|\det\left(\eta_{\ft_{J^c},\src(\pi)}\circ F_\pi\right)\right|$ with the denominator of $\varpi^J_{d,0,\vec{\mu}}(\pi_\phi)$ from \eqref{eq: standard weight}, give precisely $\left|\det\left(A_\phi\right)\right|$. More interestingly, $\HS(\pi_\phi) = 2^N$, so that if $G_{\pi_\phi}$ denotes $\src\left(\pi_\phi\right)$, then we just have to show that
    \begin{equation}
        \frac{\left(\prod_{e\in E(G_{\pi_\phi})}d_\pi(e)\right)\cdot\left(\prod_{V\in V(G_{\pi_\phi})} H(V)\cdot \mathrm{CF}(V)\right)}{\VS(\pi_\phi)} = D_\phi.\label{eq: reduction}
    \end{equation}
    We first seek to understand the local Hurwitz numbers, we follow an exhaustive classification of all the local possibilities arising at the vertices of $G_{\pi_\phi}$ following the classification preceding Proposition \ref{prop: dictionary} of the vertices of $\phi$. Consider a vertex $V\in V\left(G_{\pi_\phi}\right)$ with $D=d_{\pi_\phi}(V)$. The coloring of the following pictures describes the edges that lie in the same fiber.
    \begin{enumerate}
        \item[(I)] If $\pi_\phi(V)$ is incident to three edges and no legs, then we get the unique possibility for $G_{\pi_\phi}$ around $V$ depicted in Figure \ref{fig: 3valent case}.
        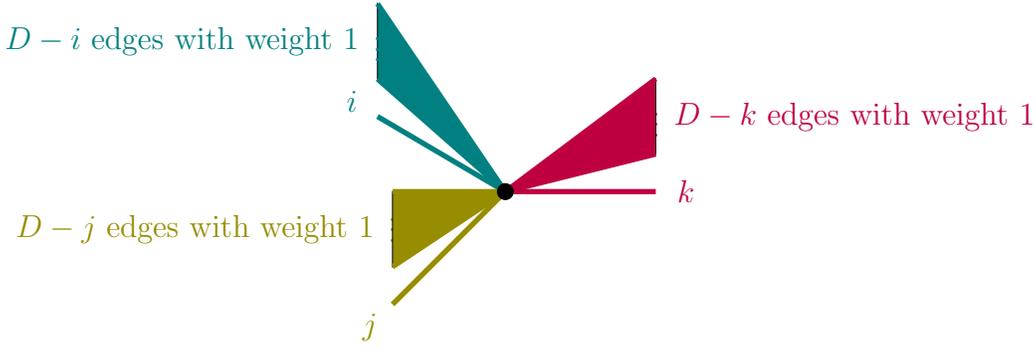
\begin{figure}[ht]
            \centering
            \begin{tikzpicture}
            
                \draw[fill = teal, opacity=0.1,line width=0pt] (0,0)--(-1.7,2.5)--(-1.7,1.5)--(0,0);
                \node[teal] at (-1.7,2) {$\boldsymbol\vdots$};
                \node[teal,anchor=east,align=right] at (-1.8,2) {$D-i$ edges with weight $1$};
                \draw[teal, line width =2pt] (0,0)--(-1.7,2.5);
                \draw[teal, line width =2pt] (0,0)--(-1.7,1.5);
                \draw[teal, line width =2pt] (0,0)--(-1.7,1) node [pos =1.2] {$i$};

                \node[olive] at (-1.5,-0.4) {$\boldsymbol\vdots$};
                \node[olive,anchor=east,align=right] at (-1.6,-0.5) {$D-j$ edges with weight $1$};
                \draw[fill = olive, opacity=0.1,line width=0pt] (0,0)--(-1.5,0)--(-1.5,-1)--(0,0);
                \draw[olive, line width =2pt] (0,0)--(-1.5,0);
                \draw[olive, line width =2pt] (0,0)--(-1.5,-1);
                \draw[olive, line width =2pt] (0,0)--(-1.5,-1.5) node [pos =1.2] {$j$};

                \node[purple] at (2,1) {$\boldsymbol\vdots$};
                \node[purple,anchor=west,align=left] at (2.1,1) {$D-k$ edges with weight $1$};
                \draw[fill = purple, opacity=0.1,line width=0pt] (0,0)--(2,1.5)--(2,0.5)--(0,0);
                \draw[purple, line width =2pt] (0,0)--(2,1.5);
                \draw[purple, line width =2pt] (0,0)--(2,0.5);
                \draw[purple, line width =2pt] (0,0)--(2,0) node [pos =1.2] {$k$};
                \draw[fill = black] (0,0) circle (3pt);
                
            \end{tikzpicture}
            \caption{Unique possiblity of $G_{\pi_\phi}$ around $V$ in case (I)}
            \label{fig: 3valent case}
        \end{figure}
        
        Since the RH number at $V$ is $0$, it follows from Figure \ref{fig: 3valent case} that $i+j+k=2D+1$. If one of $i$, $j$, or $k$ is $1$, then the other two must necessarily be equal to $D$ and therefore $H(V) = \frac{1}{D}$. In this special situation, the combinatorial factor is killed by $\VS(\pi_\phi)$. We now assume that $\min\{i,j,k\}>1$ and observe that the $H(V)$ is the number of permutations  $\sigma_i$, $\sigma_j$, $\sigma_k$ of $S_D$ with cycle types $(i,1^{D-i})$, $(j,1^{D-j})$, and $(k,1^{D-k})$ respectively, that multiply to the identity and generate a transitive group, divided by $D!$. In this case, we observe that since $\langle \sigma_i,\sigma_j\rangle$ must generate a transitive group, then the $i$-cycle of $\sigma_i$ and the $j$-cycle of $\sigma_j$ must intersect in $(i+j-D)$ numbers. Separately, the permutation $\sigma_k$ (and hence $\sigma_i\sigma_j=\sigma_k^{-1}$) has to fix $(D-k)$ numbers. Observe that a number $q\in \{1,\dots,D\}$ is fixed by $\sigma_i\sigma_j$ if and only if the permutations intersect in both $q$ and $\sigma_j(q)$, and contain these in opposite order. Since $D-k = i+j-D-1$, it follows that the cycles of $\sigma_i$ and $\sigma_j$ must consist of $(i+j-D)$ common numbers in opposite order, followed by $(D-j)$ and $(D-i)$ other numbers respectively. There are $\binom{D}{i+j-D}\cdot (i+j-D)!$ different possibilities for these $(i+j-D)$ common numbers because the order matters. Similarly, there are $\binom{D-(i+j-D)}{(D-j)}(D-j)!$ different possibilities arising from the ensuing $(D-j)$ numbers, and there are $\binom{D-(i+j-D)-(D-j)}{(D-i)}(D-i)!=(D-i)!$ different possibilities arising from the remaining $(D-i)$ numbers. In conclusion, we obtain that
        \begin{align*}
            H(V) &= \frac{1}{D!}\cdot \left(\binom{D}{i+j-D}\cdot (i+j-D)!\cdot \binom{D-(i+j-D)}{(D-j)}(D-j)!\cdot (D-i)! \right)\\
            &=\frac{1}{D!}\cdot \left(\frac{D!}{(D-(i+j-D))!}\cdot \frac{(D-(i+j-D))!}{(D-i)!}\cdot (D-i)! \right)\\
            &= 1.
        \end{align*}
        
        \item[(II)] If $\pi_\phi(V)$ is incident to only one leg of $T_t=\trgt\left(\pi_\phi\right)$, then we get the two possibilities depicted in Figures \ref{fig: interesting 2valent case} and \ref{fig: not so interesting 2valent case}. In the second possibility, it is well known that $H(V) = \frac{1}{D}$. Hence, we focus on computing the Hurwitz numbers appearing in the first possibility. Here, we get the additional condition that $D=i+j$. To compute $H(V)$ we compute the number of transpositions $\tau$ and permutations $\sigma$ of cycle type $(i)(j)$ of $S_D$, such that $\sigma\tau$ is a $D$-cycle (the transitivity of $\langle \sigma,\tau\rangle$ is therefore immediate). Assume without loss of generality that $D-i\geq j$. If $i\neq j$, then the number of different possibilities for $\sigma$ is simply  $\frac{D!}{i\cdot (D-i)!}\cdot\frac{(D-i)!}{j\cdot (D-i-j)! } =\frac{D!}{i\cdot j}$. If $i=j$, then we have to divide the previous computation by $2$. Since $\sigma\tau$ has to be a transposition, after fixing $\sigma$ there are only $i\cdot j$ different possibilities for $\tau$ satisfying this condition. In conclusion:
        \begin{itemize}
            \item[(i)] If $i\neq j$, then $H(V)=1$.
            \item[(ii)] If $i=j$, then $H(V)=\frac{1}{2}$.
        \end{itemize}
        
        \begin{figure}[ht]
        \centering
        \begin{minipage}{0.5\textwidth}
            \centering
            \begin{tikzpicture}
                \node[teal] at (-1.7,2) {$\boldsymbol\vdots$};
                \node[teal,anchor=east,align=right] at (-1.8,2) {$D-2$ edges\\with weight $1$};
                \draw[fill = teal, opacity=0.1,line width=0pt] (0,0)--(-1.7,2.5)--(-1.7,1.5)--(0,0);
                \draw[teal, line width =2pt] (0,0)--(-1.7,2.5);
                \draw[teal, line width =2pt] (0,0)--(-1.7,1.5);
                \draw[teal, line width =2pt] (0,0)--(-1.7,1) node [pos =1.2] {$2$};

                \draw[olive, line width =2pt] (0,0)--(-1.5,-1) node [pos =1.2] {$D$};
                
                \draw[purple, line width =2pt] (0,0)--(2,0.5) node [pos=1.2] {$i$};
                \draw[purple, line width =2pt] (0,0)--(2,-0.5) node [pos =1.2] {$j$};
                \draw[fill = black] (0,0) circle (3pt);
                
            \end{tikzpicture}
            \caption{First possiblity of $G_{\pi_\phi}$ around $V$ in case (II)}
            \label{fig: interesting 2valent case} 
            \end{minipage}%
            \begin{minipage}{0.5\textwidth}
            \centering
            \begin{tikzpicture}
                \node[teal] at (-1.7,1.5) {$\boldsymbol\vdots$};
                \node[teal,anchor=east,align=right] at (-1.8,1.5) {$D$ edges\\with weight $1$};\draw[fill = teal, opacity=0.1,line width=0pt] (0,0)--(-1.7,2)--(-1.7,1)--(0,0);
                \draw[teal, line width =2pt] (0,0)--(-1.7,2);
                \draw[teal, line width =2pt] (0,0)--(-1.7,1);

                \draw[olive, line width =2pt] (0,0)--(-1.5,-1) node [pos =1.2] {$D$};
                
                \draw[purple, line width =2pt] (0,0)--(2,0) node [pos=1.2] {$D$};
                
            \end{tikzpicture}
            \caption{Second possibility of $G_{\pi_\phi}$ around $V$ in case (II)}
            \label{fig: not so interesting 2valent case}
            \end{minipage}
        \end{figure}
        \item[(III)] If $\pi_\phi(V)$ is a node, then either $\val(V)=3$ or $\val(V)=4$. This means that we get the two possibilities depicted in Figures \ref{fig: first 1valent case} and \ref{fig: second 1valent case}. In the former we get $H(V)=1$, and in the latter we get $H(V)=\frac{1}{2}$. 
        \begin{figure}[ht]
            \centering
            \begin{minipage}{0.5\textwidth}
            \begin{tikzpicture}
                \draw[teal, line width =2pt] (0,0)--(-1.7,1) node [pos =1.2] {$1$};
                
                \draw[olive, line width =2pt] (0,0)--(-1.5,-1.5) node [pos =1.2] {$1$};
                
                \draw[purple, line width =2pt] (0,0)--(2,0) node [pos =1.2] {$1$};
                \draw[fill = black] (0,0) circle (3pt);
                
            \end{tikzpicture}
            \caption{Possibility of $G_{\pi_\phi}$ around $V$ with $\val(V)=3$}
            \label{fig: first 1valent case}
            \end{minipage}%
            \begin{minipage}{0.5\textwidth}
            \begin{tikzpicture}
                \draw[teal, line width =2pt] (0,0)--(-1.7,1) node [pos =1.2] {$2$};

                \draw[olive, line width =2pt] (0,0)--(-1.5,-1.5) node [pos =1.2] {$2$};

                \draw[purple, line width =2pt] (0,0)--(2,-0.5) node [pos=1.2]{$1$};
                \draw[purple, line width =2pt] (0,0)--(2,0.5) node [pos =1.2] {$1$};
                \draw[fill = black] (0,0) circle (3pt);
                
            \end{tikzpicture}
            \caption{Possibility of $G_{\pi_\phi}$ around $V$ with $\val(V)=4$}
            \label{fig: second 1valent case}
            \end{minipage}
        \end{figure}
    \end{enumerate}
    We use this same classification and Lemma 13.24 of \cite{VargasThesis} to continue with the proof. We observe that if $e\in E(\ft_{J^c}G_{\pi_\phi})$, then the edges of $G_{\pi_\phi}$ that add (Definition \ref{defi: add}) to $e$ form a path of $G_{\pi_\phi}$. In this sense, edges of $\ft_{J^c}G_{\pi_\phi}$ are given by paths of $G_{\pi_\phi}$, and this lemma classifies all the possibilities and explains the contribution of such an edge to $D_\phi$. It consists of $3$ cases, which we now explicit, and, additionally, we also explain the behavior of the contributions of the vertices in the interior of these paths:
    \begin{enumerate}
        \item[($\triangle$)] If the edge comes from a path passing through through a leaf, then it contributes a factor of $1$. In this case the only non-trivial contribution that we see from a vertex in the path is that of the Hurwitz number of the vertex above a node, which gets killed by the corresponding $\mathrm{CF}(\bullet)$. In this case, the $\VS(\pi_\phi)$ does nothing, and the contributions from our multiplicity coincide.
        \item[($\triangle\triangle$)] If the edge comes from a path that passes through vertices of $G_{\pi_\phi}$ whose number $r_\phi$ vanishes, then all edges of $G_{\pi_\phi}$ in this path have the same weight $D$ and the edge given by the path contributes a factor of $D$. It is perhaps a good moment to recall that $r_\phi$ referes to the RH numbers of $\phi$, whereas $r_{\pi_\phi}$ refers to the RH numbers of $\pi_\phi$. Since $\pi_\phi$ is a discrete admissible covers, the latter always vanish. In our terms, the vanishing of the $r_\phi$ corresponds to the following two possibilities for interior vertices of the path (i. e. vertices not in the boundary of the path):
        \begin{itemize}
            \item we are placed in case (I) from our previous classification, where (without loss of generality) $i=1$ and every {\color{teal}teal edge} is expunged after forgetting the marked legs,
            \item we are placed in the second possibility of case (II) from our previous classification, where every {\color{teal}teal edge} actually is a marked leg of $G_{\pi_\phi}$, and hence is expunged after forgetting the marked legs.
        \end{itemize}
        In these cases, the factors $\mathrm{CF}(\bullet)$ (which are all just $D!$) coming from the {\color{teal}edges (or legs)} are killed by the $\VS(\pi_\phi)$. All the edges of $G_{\pi_\phi}$ in this path contribute a factor of $D$, and all but one are killed by the local Hurwitz numbers. So, the contributions coming from our multiplicity coincide.
        \item[($\triangle\triangle\triangle$)] Otherwise the edge comes from two paths of $G_{\pi_\phi}$ joined at a vertex with a local picture as in the first possibility of our case (II) from the previous classification, and each of these paths passes through vertices whose images are vertices with trivial leg valency. Moreover, for each of these paths, the edges have all the same weight, and the difference between the weights from one path and the other is just $1$. In this case, the edges of each path have the same weight, the contribution of this edge to the $D_\phi$ is the product of both weights. In our terms, this means that the vertex where these paths are joined has a local picture as in the first possibility of our case (II) with $i=1$, $j=D-1$, and all the {\color{teal} edges} are actually marked legs of $G_{\pi_\phi}$. Here, the edge with weight $i$ is erased after forgetting the marking, and the same analysis as in the previous case yields that we get the same contribution from each path, and we just have the explain the contribution from the vertex where these paths join. The factor $\mathrm{CF}(\bullet)$ coming from the possible repetitions in the {\color{teal} marked legs} is killed by the $\VS(\pi_\phi)$. The local Hurwitz number is either $1$ and there are no accompanying factors $\mathrm{CF}(\bullet)$, or $\frac{1}{2}$ (in case $D=2$) and the factor $\mathrm{CF}(\bullet)$ kills this contribution. So, in other words, the local contributions coming from our multiplicity coincide. 
    \end{enumerate}
    The previous reasoning shows that the products of the weights actually give rise to the $D_\phi$, taking into account the local Hurwitz numbers.\\
    To finish the proof it is necessary to show that the contribution of the extremal vertices of the paths given by the cases ($\triangle$) to ($\triangle\triangle\triangle$) is trivial or cancels with the what remains of $\VS(\pi_\phi)$. A change-minimal full-rank DT-morphism satisfies the dangling-no-glue condition (Definition 48 and Lemma 53 of loc. cit.), which implies that these extremal vertices are only of the following form:
    \begin{itemize}
        \item The local picture is the unique possibility of case (I), where every edge except the three edges with weights $i$, $j$, and $k$ is expunged after forgetting the marked legs. If $\min\{i,j,k\}>1$, then the local Hurwitz number is $1$ and the factors $\mathrm{CF}(\bullet)$ are cancelled by the remaining factors of $\VS(\pi)$. If (without loss of generality) $i=1$, then the local Hurwitz number is $\frac{1}{D}$ and the combinatorial factors give $D!$, but in this case $\VS(\pi_{\phi})$ is precisely $(D-1)!$ coming from the permutations of the {\color{teal}teal edges} that are erased after forgetting the marking. So, the contribution of the extremal vertices having this form is always $1$.
        
        \item The local picture is the first possibility of case (II), where the {\color{teal} teal edges} are actually marked legs, and none of the {\color{purple} purple edges} are erased after forgetting the marking. If $i\neq j$, then the local Hurwitz number is $1$, and the only factor $\mathrm{CF}(\bullet)$ arises from the {\color{teal} teal marked legs} which actually gets killed by what remains of $\VS(\pi_\phi)$. If $i=j$, then the local Hurwitz number is $\frac{1}{2}$, which gets killed by the factor $\mathrm{CF}(\bullet)$ coming from the {\color{purple} purple edges}, and the factor $\mathrm{CF}(\bullet)$ coming from the {\color{teal}teal marked legs} is, once again, killed by the remainings of $\VS(\pi_\phi)$.    \end{itemize}
        After exhausting all the possibilities, our comprehensive analysis shows that the extremal vertices contribute trivially, and hence the theorem follows.
\end{proof}

\end{document}